\documentclass{amsart}
\usepackage{graphicx} % Required for inserting images
\usepackage{amsfonts, array, tikz, amsmath}
\usepackage{amssymb}
\usepackage{amsthm} % Defines proof environment
\usepackage{blkarray} % Label rows and columns of matrices
\usepackage{stmaryrd}
\usepackage{xcolor}
\usetikzlibrary{shapes,snakes}
\usepackage{tikz-cd}
\usepackage[shortlabels]{enumitem}
\usepackage{booktabs}
\usepackage{float}
\usepackage[hidelinks]{hyperref}
\usepackage{tensor}
\usepackage{mathtools}
\usepackage{comment}
\usepackage{fullpage}

\usepackage{listings}
\lstdefinelanguage{Sage}[]{Python}
{morekeywords={False,sage,True},sensitive=true}
\definecolor{dblackcolor}{rgb}{0.0,0.0,0.0}
\definecolor{dbluecolor}{rgb}{0.01,0.02,0.7}
\definecolor{dgreencolor}{rgb}{0.2,0.4,0.0}
\definecolor{sgreencolor}{rgb}{0.0,0.5,0.2}
\definecolor{sbluecolor}{rgb}{0.0,0.5,0.5}
\definecolor{dgraycolor}{rgb}{0.30,0.3,0.30}

\pgfdeclarelayer{edgelayer}
\pgfdeclarelayer{nodelayer}
\pgfsetlayers{edgelayer,nodelayer,main}

\tikzstyle{none}=[inner sep=0mm]

\tikzstyle{red dot}=[line width=0.5mm, fill=white, draw=red, shape=circle]
\tikzstyle{blue dot}=[line width=0.5mm, fill=white, draw=blue, shape=circle]
\tikzstyle{green dot}=[line width=0.5mm, fill=white, draw=green, shape=circle]
\tikzstyle{violet dot}=[line width=0.5mm, fill=white, draw=violet, shape=circle]
\tikzstyle{yellow}=[line width=0.5mm, fill=white, draw=yellow, shape=circle]
\tikzstyle{black dot}=[line width=0.5mm, fill=white, draw=black, shape=circle]

\tikzstyle{red edge}=[line width=0.5mm, -, draw=red]
\tikzstyle{blue edge}=[line width=0.5mm, -, draw=blue]
\tikzstyle{green edge}=[line width=0.5mm, -, draw=green]
\tikzstyle{violet edge}=[line width=0.5mm, -, draw=violet]
\tikzstyle{yellow edge}=[line width=0.5mm, -, draw=yellow]
\tikzstyle{black edge}=[line width=0.5mm, -]

\tikzset{
glow/.style={
preaction={#1, draw, line join=round, line width=0.5pt, opacity=0.04,
preaction={#1, draw, line join=round, line width=1.0pt, opacity=0.04,
preaction={#1, draw, line join=round, line width=1.5pt, opacity=0.04,
preaction={#1, draw, line join=round, line width=2.0pt, opacity=0.04,
preaction={#1, draw, line join=round, line width=2.5pt, opacity=0.04,
preaction={#1, draw, line join=round, line width=3.0pt, opacity=0.04,
preaction={#1, draw, line join=round, line width=3.5pt, opacity=0.04,
preaction={#1, draw, line join=round, line width=4.0pt, opacity=0.04,
preaction={#1, draw, line join=round, line width=4.5pt, opacity=0.04,
preaction={#1, draw, line join=round, line width=5.0pt, opacity=0.04,
preaction={#1, draw, line join=round, line width=5.5pt, opacity=0.04,
preaction={#1, draw, line join=round, line width=6.0pt, opacity=0.04,
}}}}}}}}}}}}}}

\definecolor{purple}{RGB}{138,43,226}

\newcommand{\nocontentsline}[3]{}
\let\origcontentsline\addcontentsline
\newcommand\stoptoc{\let\addcontentsline\nocontentsline}
\newcommand\resumetoc{\let\addcontentsline\origcontentsline}

\counterwithin{figure}{section}
\counterwithin{equation}{section}

\title{Kirillov's conjecture on Hecke--Grothendieck polynomials}
\author{Ben Brubaker}
\author{A. Suki Dasher}
\author{Michael Hu}
\author{Nupur Jain}
\author{Yifan Li}
\author{Yi Lin}
\author{Maria Mihaila}
\author{Van Tran}
\author{I. Deniz \"Unel}
\date{\today}

\newtheorem{theorem}{Theorem}[section]
\newtheorem{proposition}[theorem]{Proposition}
\newtheorem{lemma}[theorem]{Lemma}
\newtheorem{corollary}[theorem]{Corollary}
\newtheorem{conjecture}[theorem]{Conjecture}
\theoremstyle{definition}
\theoremstyle{definition}
\newtheorem{definition}[theorem]{Definition}
\newtheorem{remark}[theorem]{Remark}

\makeatletter
\def\l@subsection{\@tocline{2}{0pt}{2pc}{6pc}{}} \makeatother

\begin{document}

% Prevent equations from spilling into margin
\emergencystretch 3em

\maketitle

\begin{abstract}
We use algebraic methods in statistical mechanics to represent a multi-parameter class of polynomials in several variables as partition functions of a new family of solvable lattice models.
The class of polynomials, defined by A. N. Kirillov, is derived from the largest class of divided difference operators satisfying the braid relations of Cartan type $A$.
It includes as specializations Schubert, Grothendieck, and dual-Grothendieck polynomials, among others.
In particular, our results prove positivity conjectures of Kirillov for the subfamily of Hecke--Grothendieck polynomials, while the larger family is shown to exhibit rare instances of negative coefficients.
\end{abstract}

\tableofcontents

\section{Introduction}

Solvable lattice models have been extensively used to represent important classes of special functions in the symmetric function theory arising in the study of algebraic groups and related areas.
Here ``solvability'' means that the Boltzmann weights of the lattice model satisfy certain Yang--Baxter equations, which allows for the generating function on the lattice model, the so-called ``partition function,'' to be written as an explicit formula.
Once a connection is made, the tools from solvable lattice models may be used to prove many interesting identities and relations for the special functions that arise as their generating functions.

It thus becomes important to associate classes of solvable lattice models with hierarchies of special functions.
Quantum group modules are an important source of Boltzmann weights which satisfy Yang--Baxter equations, and so may be used as an organizing principle in this regard.
For example, there are natural families of representations of the affine quantum algebras $U_q(\hat{\mathfrak{gl}}(r))$ and superalgebras $U_q(\hat{\mathfrak{gl}}(m|n))$ associated to standard evaluation modules $V(z)$ with $z \in \mathbb{C}^\times$.
Lattice models whose Boltzmann weights are taken from the $R$-matrices of these quantum group modules yield closed-form solutions for their partition functions as a function of the variables $z$.
Using such modules, various lattice models have been constructed which produce families of partition functions giving multi-parameter generalizations of non-symmetric Hall--Littlewood polynomials \cite{borodin-wheeler}, generalizations of LLT polynomials \cite{curran2022latticemodelsuperllt, ABW2023} and non-symmetric Whittaker functions \cite{BBB2019, BBBGIwahori, BBBG2024, metahori}, and polynomials in Schubert calculus \cite{KM, L2007, brubaker-bump-friedberg2011, MS2013, KZJ2017,wheeler--zinn-justin-2019, BSW2020, BBBGatoms2021, KZJ2021, ABPW2023}, including double $\beta$-Grothendieck polynomials \cite{BS2022, frozen-pipes}.

In many of these cases above, the identification of partition functions and special functions was made by using the divided difference operator recursions which followed from the solvability of the lattice model.
Thus it is natural to wonder whether solvable lattice models may be used to represent {\it all} functions arising from divided difference operators.
Of course, ``all'' needs to be qualified, but in particular we mean it to include the universal families of divided difference operators and resulting families of polynomials which were studied by A. N. Kirillov in~\cite{kirillov2016notes}.
In this paper, we give a new class of colored lattice models, which we prove to be solvable, whose partition functions match the families defined in~\cite{kirillov2016notes}.
As an application, we prove positivity conjectures of Kirillov's concerning these functions along the way.
The lattice model we will propose seems to lie outside of the known framework of those arising from quantum group modules, thus requiring considerable effort to demonstrate the solvability of the model.
To show that our Boltzmann weights satisfy the Yang--Baxter equation, we instead appeal to the fact that at most one color may occupy a horizontal edge of the lattice model, thereby permitting the proof to be reduced to a finite, albeit large, computation.
We comment more on the apparent absence of quantum group modules for our lattice model at the conclusion of the introduction.

To state our results precisely, recall that the \textit{$i$-th divided difference operator} $\partial_i$ is defined on $f$ in $\mathbb{C}[x_1, x_2, \ldots, x_n]$ by
$$
\partial_i(f(x_1, \ldots, x_n)) = \frac{f(x_1, \ldots, x_{i-1}, x_{i+1}, x_i, x_{i+2},\ldots, x_n) - f(x_1, \ldots, x_n)}{x_{i+1} - x_i}.
$$
In \cite{kirillov2016notes}, Kirillov defines a generalized family of operators in terms of five parameters\footnote{We adopt this choice of parameters in place of Kirillov's $(a,b,c,h,e)$.} $(a,b,c,d,e)$:
$$
T_i^{(a,b,c,d,e)} = a + (bx_i + cx_{i+1} + d + e x_i x_{i+1}) \partial_i.
$$
Kirillov furthermore determines that the operators $T_i \coloneqq T_i^{(a,b,c,d,e)}$ satisfy the braid relations $T_iT_{i+1} T_i = T_{i+1} T_i T_{i+1}$ if and only if 
\begin{equation}
\label{braid-condition}
    (a+b)(a-c)+de = 0.
\end{equation}
It is natural to define multi-parameter operators $T_i$ using only the variables $x_i$ and $x_{i+1}$, as these are the variables affected by the divided difference operator.
But in fact, this choice of parametrization is essentially universal among operators consisting of all combinations of divided difference operators $\partial_i$, the reflection operator $s_i$, and multiplication by polynomials (see \cite[Theorem 19]{Zemel} where ``essentially'' is made precise in terms of a non-degeneracy assumption).
If relation~(\ref{braid-condition}) on the set of parameters holds, then for any permutation $w$ in the symmetric group $S_n$, the operator
$$
T_w^{(a,b,c,d,e)} = T_{i_1}^{(a,b,c,d,e)} \circ \cdots \circ T_{i_\ell}^{(a,b,c,d,e)}
$$
is well-defined for any reduced decomposition $w = s_{i_1} \cdots s_{i_\ell}$.

These generalized divided difference operators may be applied recursively to various initial polynomial functions in $\mathbb{C}[x_1, x_2, \ldots, x_n]$ to obtain numerous families of important functions.
For example, the specializations of $(a,b,c,d,e)$ to $(q, -1, q, 0, 0)$ and $(-1, 1, q^{-1}, 0, 0)$ correspond to the divided difference operators arising from quantum groups in the aforementioned \cite{borodin-wheeler} and \cite{BBBGIwahori}, respectively.
Moreover, Kirillov defines for any permutation $w \in S_n$ the ``generalized Schubert polynomials''
$$
\mathfrak{S}_w^{(a,b,c,d,e)}(\boldsymbol{x}) \coloneqq T_{w^{-1} w_0}^{(a,b,c,d,e)} (\boldsymbol{x}^\rho),
$$
where $w_0$ is the long permutation defined by $w_0(i) = n + 1 - i$ and we have used the vector notation $\boldsymbol{x} = (x_1, \ldots, x_n)$, with $\boldsymbol{x}^\rho$ denoting the monomial $x_1^{n-1} x_2^{n-2} \cdots x_{n-1}$ corresponding to the vector $\rho = (n-1, n-2, \ldots, 1, 0)$ (see~\cite[Definition 4.17]{kirillov2016notes}).
The name for this family is apt, as various specializations of $(a,b,c,d,e)$ recover families of polynomials arising in Schubert calculus.
For example, the choices of $(a,b,c,d,e)$ equal to $(0,0,0,1,0)$, $(-\beta, \beta, 0,1,0)$, and $(0, \beta, 0, 1, 0)$ correspond to Schubert, $\beta$-Grothendieck, and dual $\beta$-Grothendieck polynomials, respectively \cite[Lemma 4.18]{kirillov2016notes}.
In keeping with the analogy with classical Schubert calculus, for a weak composition $\zeta$ with $n$ parts, Kirillov defines the ``generalized key polynomials''
$$
K_\zeta^{(a,b,c,d,e)}(\boldsymbol{x}) \coloneqq T^{(a,b,c,d,e)}_{v_\zeta}(\boldsymbol{x}^{\zeta^+}),
$$
where $\zeta^+$ is the unique partition obtained by reordering the parts of $\zeta$ and $v_\zeta$ is a minimal length coset representative in $S_n / \text{stab}(\zeta^+)$ such that $v_\zeta \cdot \zeta = \zeta^+$.
Specializations of $(a,b,c,d,e)$ to $(1,0,1,0,0)$, $(0, 0, 1, 0, 0)$, $(1, 0, 1, 0, \beta)$, and $(0, 0, 1, 0, \beta)$  correspond to key polynomials (Demazure characters), reduced key polynomials (Demazure atoms), key $\beta$-Grothendieck polynomials, and reduced key $\beta$-Grothendieck polynomials.
Note that an apparent dichotomy between analogues of Schubert polynomials and analogues of key polynomials is given in terms of the parameter $d$: in most familiar examples, the first class of polynomials has parameters with $d \neq 0$ and the latter class has parameters for which $d = 0$.

Supposing that $d \ne 0$ in (\ref{braid-condition}), we may reduce to a four-parameter family $(\alpha, \beta, \gamma, d)$ by setting
$$
a = -\beta, \,\ b = \alpha+\beta+\gamma, \,\ c = \gamma, \,\ d = d,
$$
and then use (\ref{braid-condition}) to solve for $e = (\alpha + \gamma) (\beta+\gamma) d^{-1}.$
The resulting operators satisfy
$$
T_w^{(\alpha, \beta, \gamma, d)} = d^{\ell(w)}T_w^{(\alpha d^{-1}, \beta d^{-1}, \gamma d^{-1}, 1)},
$$
where $\ell(w)$ is the length of $w$, so we may assume that $d = 1$ by rescaling $\alpha$, $\beta$, and $\gamma$ accordingly.
Thus we may focus on the subclass of divided difference operators for parameters\footnote{We use the ordering $(\alpha, \beta, \gamma)$ in place of Kirillov's $(\beta, \alpha, \gamma)$ in our notation for the operators, but our use of $\alpha$ and $\beta$ in the definition of (\ref{kirillop}) agrees with \cite{kirillov2016notes}.} $(\alpha, \beta, \gamma)$ given by
\begin{equation}
\label{kirillop}
    T_i^{(\alpha, \beta, \gamma)} \coloneqq -\beta + ((\alpha + \beta + \gamma)x_i + \gamma x_{i+1} + 1 + (\beta + \gamma)(\alpha + \gamma)x_i x_{i+1})\partial_i.
\end{equation}
In addition to braiding, they satisfy a quadratic relation $(T_i^{(\alpha, \beta, \gamma)})^2 = (\alpha - \beta) T_i^{(\alpha, \beta, \gamma)} + \alpha \beta$, so the $T_i^{(\alpha, \beta, \gamma)}$ with $i=1, \ldots, n-1$ generate a Hecke algebra of Cartan type $A_{n-1}$.
We are thus led to define the following family of polynomials, which we show in Section \ref{proofsection} recovers specializations of both the generalized Schubert polynomials and the generalized key polynomials.

\begin{definition}
    \label{HGdef}
    Let $\alpha$, $\beta$, and $\gamma$ be parameters.
    Given an integer $n > 0$, a permutation $w \in S_n$, and a partition $\lambda = (\lambda_1 \geq \lambda_2 \geq \cdots \geq \lambda_n)$, define the {\it twisted Kirillov polynomials}
    $$
    \mathcal{KN}_w^{(\alpha, \beta, \gamma)}(\boldsymbol{x}; \lambda) \coloneqq T_w^{(\alpha, \beta, \gamma)}(\boldsymbol{x}^{(d_1, d_2, \ldots, d_n)}),
    $$
    where $d_i = \lambda_1 - \lambda_{n + 1 - i} + n - i$ for $i = 1, \ldots, n$, and $\boldsymbol{x} = (x_1, x_2, \ldots, x_n)$.
\end{definition}

\noindent
The family of functions is so-named because the special case $\lambda = \boldsymbol{0}$ matches Kirillov's definition~\cite[Definition 4.6]{kirillov2016notes}.
In the further special case where $\gamma = 0$, Kirillov terms the resulting functions the \textit{Hecke--Grothendieck} polynomials.
We use the adjective ``twisted'' to refer to the presence of $\lambda$, a common practice in the presence of characters where we view the partition $\lambda$ applied to $\boldsymbol{x}$ as corresponding to the character of a complex torus.
Its placement in polynomials generated by divided difference operators as above often corresponds geometrically to twisting by a line bundle corresponding to the dominant weight $\lambda$ (see, for example, Theorem 1.1 in~\cite{MihalceaSuWhittaker}).
We may now state the main result of the present paper.

\begin{theorem}
    For each positive integer $n$, there exists a solvable lattice model depending upon parameters $\alpha, \beta, \gamma$ and variables $\boldsymbol{x} = (x_1, \ldots, x_n)$ such that for each permutation $w \in S_n$ and integer partition $\lambda = (\lambda_1 \geq \lambda_2 \geq \cdots \geq \lambda_n)$, there is a corresponding boundary condition for the lattice model such that the partition function is equal to $\mathcal{KN}_w^{(\alpha, \beta, \gamma)}(\boldsymbol{x}; \lambda)$.
\label{maintheorem}
\end{theorem}

\noindent
This theorem demonstrates the precise sense in which we claim that lattice models are a source of functions obtained by universal divided difference operators.
Note however that we have only handled the ``generic case'' $d \neq 0$ in the construction of our lattice model, as this is required to restrict from the parameters $(a, b, c, d, e)$ to the parameters $(\alpha, \beta, \gamma)$.
The resulting family of functions provides a common generalization of the $\beta$-Grothendieck polynomials, the dual $\beta$-Grothendieck polynomials, and the Di Francesco--Zinn-Justin polynomials \cite[Remark 4.7]{kirillov2016notes}, to name a few, but the lattice model is not able to compute their associated key polynomials, for which $d = 0$.
In subsequent work by the second author, a generalization of our lattice model is constructed in which the parameter $d$ is permitted to be arbitrary \cite{dasher_2025}, thereby demonstrating that lattice models are a source for universal divided difference operators in the sense of the remark above.

As an application of our main theorem, we prove positivity results for the coefficients of the functions it computes.
Kirillov was interested in geometric interpretations of the polynomials $\mathcal{KN}_w^{(\alpha, \beta, \gamma)}(\boldsymbol{x}; \boldsymbol{0})$ and the specializations $K_\zeta^{(\alpha, \beta, \gamma)}(\boldsymbol{x})$ of the generalized key polynomials with the further restriction that $\zeta \subset \rho$, as well as their analogues for which the parameter $d$ remains arbitrary.
In that direction, he makes the following conjecture:

\begin{conjecture}[{\cite[Conjectures 1.3 and 4.9]{kirillov2016notes}}]
\label{conjecture4.9}
    The polynomials have non-negative coefficients:
    $$
    \mathcal{KN}_w^{(\alpha, \beta, \gamma)}(\boldsymbol{x}; \boldsymbol{0}) \in \mathbb{N}[\alpha, \beta, \gamma][x_1, x_2, \ldots, x_n], \qquad K_\zeta^{(\alpha, \beta, \gamma)}(\boldsymbol{x}) \in \mathbb{N}[\alpha, \beta, \gamma][x_1, x_2, \ldots, x_n].
    $$
\end{conjecture}

\noindent
We show in Section \ref{applications} that the conjecture is false in general, although it is known to hold for many specializations of the parameters $\alpha$, $\beta$, $\gamma$ and the variables $x_i$; see for example \cite[Theorem A]{chen-zhang-2022}, \cite{kirillov2016notes}, and the citations therein.
While the conjecture does not hold, the lattice model we construct provides positive answers in several important cases.
Indeed, as a result of Theorem~\ref{maintheorem} we readily conclude the following.

\begin{theorem}
\label{HGcoeffs}
    The coefficients of the Hecke--Grothendieck polynomials are non-negative:
    $$
    \mathcal{KN}_w^{(\alpha, \beta, \gamma = 0)}(\boldsymbol{x}; \boldsymbol{0}) \in \mathbb{N}[\alpha,\beta][x_1, x_2, \ldots, x_n].
    $$
\end{theorem}

\noindent
As a consequence, we produce an argument that the $\beta$-Grothendieck polynomials ($\alpha = \gamma = 0$) and their duals ($\beta = \gamma = 0$), the Schubert polynomials ($\alpha = \beta = \gamma = 0$), and the Di Francesco--Zinn-Justin polynomials ($\alpha = \beta = 1$ and $\gamma = 0$) have non-negative coefficients, thereby recovering known results, as well as resolving a number of Kirillov's conjectures.
These results follow immediately from the fact that in the specialization $\gamma=0$, the only vertices which appear in admissible states of the lattice model have Boltzmann weights that are polynomials in $\mathbb{N}[\alpha, \beta][x_1, x_2, \ldots, x_n]$.
Thus the generating function, as a sum of products of Boltzmann weights, is similarly non-negative.
It is the proof of Theorem~\ref{maintheorem}, in which we show that these generating functions match the divided difference operator recursion, that requires the difficult step of demonstrating that the lattice model is solvable.

The lattice models in Theorem \ref{maintheorem} resemble those referenced above which are associated to the quantum algebras $U_q(\hat{\mathfrak{gl}}(n+1))$ and quantum superalgebras $U_q(\hat{\mathfrak{gl}}(n|1))$, but they are not associated to either quantum group.
More explicitly, we may parametrize a basis of weight vectors in a quantum group evaluation module according to color, so that states of the lattice model consist of tetravalent vertices with colored edges.
Since the action of the Cartan subalgebra via comultiplication commutes with the $R$-matrices of the modules, the $R$-matrices preserve weight spaces, which ensures that in admissible states, the edges form paths in each color traveling through the vertices of the lattice model.
In the cases of lattice models made with evaluation representations of $U_q(\hat{\mathfrak{gl}}(n+1))$ or $U_q(\hat{\mathfrak{gl}}(n|1))$, there will be paths of $n$ different colors traveling, say, from the top boundary of the lattice, downward and leftward through the lattice, and exiting out the left boundary.\footnote{The lattice models made from evaluation representations of $U_q(\hat{\mathfrak{gl}}(m|n))$ feature $m$ colored paths moving leftward and $n$ colored paths moving rightward in the lattice model.
These are quite interesting, but produce recursive relations on partition functions involving four terms, which may be understood as vector-valued divided difference operator actions.
Lattice models of this type are not used in this paper.}
The lattice models we construct share this property, but our Boltzmann weights are not recognizable as related to quantum group modules, and hence we must rely on intricate and uniform case methods to prove solvability.
Indeed, the Boltzmann weights we arrive at in Figure~\ref{coloredalabw} are much stranger than anything we have witnessed before, but are nevertheless solvable.
In Lemma \ref{gamma_zero_lemma} we show that a fascinating dichotomy appears in the admissible states of our lattice model according to whether the parameter $\gamma$ is $0$.
If $\gamma \ne 0$, then the lattice model allows for multiple colored paths to travel along the same vertical edge in the grid, while this is forbidden in the lattice models for which $\gamma=0$.
Correspondingly, while some of the Boltzmann weights involve complete homogeneous symmetric functions in the parameters $\alpha$ and $\beta$, these strange weights only appear in admissible states  when $\gamma \ne 0$.
We provide an additional positivity result (see Theorem~\ref{thm:coolerpositivity}) beyond the Hecke--Grothendieck case, which is all the more fascinating because it arises for $\gamma \ne 0$, where our Boltzmann weights are most complicated.

\medskip

\noindent
\textbf{Acknowledgements.}
This work began as part of the 2023 Polymath Jr Program, and we are pleased to acknowledge their efforts in facilitating this collaboration.
We thank the anonymous referee for many thoughtful comments that improved the exposition throughout the paper.
The second author thanks the Lodha Mathematical Sciences Institute for its generous hospitality and support during the final revisions of this manuscript.
This work was supported by the National Science Foundation [DMS-2218374, DMS-2101392].

\section{The lattice model}
~\label{latticemodelbasics}

A lattice model is a combinatorial tool with an associated generating function, which in the cases of predominant interest yields a graphical means of computing the matrix coefficients of morphisms of quantum group modules.
In this paper we define a rectangular lattice model, but one may analogously define other 2-dimensional lattice models, such as those which appear in Figures \ref{YBELM} and \ref{RRRLM}.
 
\subsection{Basic construction and Boltzmann weights}

Our rectangular \textit{lattice model} is comprised of the following data.
Given positive integers $n$ and $N$, we assemble a rectangular grid with $n$ rows and $N + 1$ columns.
The rows are indexed $1, 2, \ldots, n$ in ascending order from top to bottom, and to the $i$-th row we assign the \textit{spectral parameter} $x_i$.
The columns are indexed $N, N-1, \ldots, 1, 0$ in descending order from left to right.
The intersections of the rows and columns of the grid are called \textit{vertices}, and the line segments adjacent to vertices are called \textit{edges}.
Each vertex has four adjacent edges, and edges which are adjacent to only one vertex are called \textit{boundary edges}.
Horizontal edges (those which occupy a row) may be labeled by elements of the set $\{+, 1, \dots, n\}$, and vertical edges (those which occupy a column) may be labeled by subsets of $\{1, 2, \ldots, n\}$.
An assignment of labels to all outer edges of the grid is referred to as a \textit{boundary condition}, and a \textit{state} is an assignment of labels to all edges of the grid.
Below we define the \textit{Boltzmann weight} of a vertex $v$, which is a function of the spectral parameter assigned to the row which $v$ occupies and of the labels assigned to the edges adjacent to $v$.
The \textit{Boltzmann weight of a state} is defined to be the product of the Boltzmann weights of all vertices in the grid.
A state is said to be \textit{admissible} if its Boltzmann weight is non-zero.
Given a boundary condition $\mathcal{B}$, we define a \textit{system} $\mathcal{S}_\mathcal{B}$ to be the collection of all states $\mathfrak{s}$ with boundary condition $\mathcal{B}$, together with the Boltzmann weight of $\mathfrak{s}$.
The \textit{partition function} of the lattice model is a function of the boundary condition, and is defined to be the sum of the Boltzmann weights of all states in the system $\mathcal{S}_\mathcal{B}$, which we denote by $Z(\mathcal{S}_\mathcal{B})$.

The non-zero Boltzmann weights for our lattice model are presented in Figure \ref{coloredalabw} below.
Vertices with any other labeling of adjacent edges are assigned a Boltzmann weight of 0.
For brevity, we introduce the notation
$
x \oplus_{\beta, \gamma} 1 \coloneqq 1 + (\beta + \gamma) x,
$
as the operation $\oplus_{\beta, \gamma}$ realizes the multiplicative formal group law with parameter $\beta + \gamma - 1$.
The appearance of formal group laws in parametrized Yang--Baxter equations making use of binary operations is natural.
Recent papers on lattice models, including \cite{gorbounovkorff, frozen-pipes}, have invoked them in the description of Boltzmann weights,
 and in particular, our operation specializes to the formal group law in \cite{frozen-pipes} upon setting $\gamma = 0$ and $\beta \mapsto \beta - 1$.

\newpage

\newcommand{\arrowdiagram}[4]{
\begin{tikzpicture}[anchor=base, baseline]
  \draw[line width=1.0pt,<-, >=stealth] (-1,0) node[left]{$#1$} -- (.9,0) node[right]{$#3$};
  \draw[line width=4pt, ->, >=stealth] (0,.8) node[above]{$#2$} -- (0,-1) node[below]{$#4$};
\end{tikzpicture}
}

\newcommand{\arrowdiagramgg}[4]{
\begin{tikzpicture}[anchor=base, baseline]
  \draw[line width=1.0pt,<-, >=stealth, lightgray] (-1,0) node[left]{$#1$} -- (.9,0) node[right]{$#3$};
  \draw[line width=4pt, ->, >=stealth] (0,.8) node[above]{$#2$} -- (0,-1) node[below]{$#4$};
\end{tikzpicture}
}

\newcommand{\arrowdiagramgu}[4]{
\begin{tikzpicture}[anchor=base, baseline]
  \draw[line width=1.0pt,<-, >=stealth, lightgray] (-1,0) node[left]{$#1$} -- (0,0);
  \draw[line width=1.0pt] (0,0) -- (.9,0) node[right]{$#3$};
  \draw[line width=4pt, ->, >=stealth] (0,.8) node[above]{$#2$} -- (0,-1) node[below]{$#4$};
\end{tikzpicture}
}

\newcommand{\arrowdiagramug}[4]{
\begin{tikzpicture}[anchor=base, baseline]
  \draw[line width=1.0pt,<-, >=stealth] (-1,0) node[left]{$#1$} -- (0,0);
  \draw[line width=1.0pt, lightgray] (0,0) -- (.9,0) node[right]{$#3$};
  \draw[line width=4pt, ->, >=stealth] (0,.8) node[above]{$#2$} -- (0,-1) node[below]{$#4$};
\end{tikzpicture}
}

\begin{figure}[H]
\begin{equation*} \def\arraystretch{1.7}
  \begin{array}{|c|c|c|}\hline
        \arrowdiagramgg{+}{\Sigma}{+}{\Sigma} & \arrowdiagram{c}{\Sigma}{c}{\Sigma} & \arrowdiagramgu{+}{\Sigma}{c}{\Sigma^+_c} \\
        \hline
        (\dagger) & \begin{cases} (x \oplus_{\alpha + \beta, \gamma} 1) (\alpha \beta)^{|\Sigma_{[c+1,n]}|} & \text{if } c \in \Sigma \\ x (\alpha \beta)^{|\Sigma_{[c+1,n]}|} & \text{if } c \notin \Sigma\end{cases} & (-\alpha)^{|\Sigma_{[c+1,n]}|} (x \oplus_{\alpha, \gamma} 1)(x \oplus_{\beta, \gamma} 1) \\
        \hline
        \arrowdiagramug{c}{\Sigma}{+}{\Sigma^-_c}
        & \arrowdiagram{c}{\Sigma}{d}{\tensor*{\Sigma}{*^+_d^-_c}} & \arrowdiagram{d}{\Sigma}{c}{\tensor*{\Sigma}{*^+_c^-_d}} \\
        \hline
        (\ddagger) & (-\alpha)^{|\Sigma_{[d+1,n]}|} (-\beta)^{|\Sigma_{[c+1,n]}|}  (x \oplus_{\alpha, \gamma} 1) & (-\alpha)^{|\tensor*{\Sigma}{*^+_c^-_d}_{[c+1,n]}|} (-\beta)^{|\Sigma_{[d+1,n]}|} (x \oplus_{\beta, \gamma} 1) \\
        \hline
  \end{array}
\end{equation*}
\begin{align*}
    (\dagger) &= (-1)^{|\Sigma|+1} (((\alpha + \gamma)(\beta + \gamma)x  + \gamma) \cdot h_{|\Sigma|-1}(\alpha, \beta) + \alpha \beta \cdot h_{|\Sigma|-2}(\alpha, \beta)) \\
    &= (-1)^{|\Sigma|}(\beta^{|\Sigma|} - (\beta + \gamma)\cdot h_{|\Sigma|-1}(\alpha,\beta)\cdot(x\oplus_{\alpha,\gamma}1)) \\
    (\ddagger) &= (-1)^{|\Sigma|}(-\beta)^{|\Sigma_{[c+1,n]}|} (\alpha \beta \cdot h_{|\Sigma|-3}(\alpha,\beta)  + \gamma \cdot h_{|\Sigma|-2}(\alpha, \beta)) \\
    &= (-1)^{|\Sigma|-1}(-\beta)^{|\Sigma_{[c+1,n]}|}(\beta^{|\Sigma|-1} - (\beta + \gamma)\cdot h_{|\Sigma|-2}(\alpha,\beta))
\end{align*}
\vspace{-.75cm}
    \caption{Boltzmann weights in the style of~\cite[(2.2.2) or (2.2.6)]{borodin-wheeler}.
    In place of their multiset $\mathbf{I}$ of colors, we use a subset $\Sigma$ of the colors $\{ 1, 2, \ldots, n \}$.
    We assume $c < d$ above.
    The function $h_k(\alpha,\beta)$ denotes the $k$-th complete homogeneous symmetric function of degree $k$, where we interpret $h_{-1} = 0$ and $h_{-2} = \frac{-1}{\alpha \beta}$ according to the usual recursion $h_k = \alpha h_{k-1} + \beta^k$.
    We adopt the notation $\Sigma_{[c+1, n]} \coloneqq \{s \in \Sigma : s \geq c + 1\}$, as well as $\Sigma_c^+ \coloneqq \Sigma \cup \{ c \}$, $\Sigma_c^- \coloneqq \Sigma \setminus \{ c \}$, and $\tensor*{\Sigma}{*^+_c^-_d} = (\Sigma \cup \{ c \}) \setminus \{ d \}$ for any $c \notin \Sigma$ and $d \in \Sigma$.}
\label{coloredalabw}
\end{figure}

\subsection{Boundary conditions and partition functions}
\label{boundary_conditions}

We now present boundary conditions for the lattice model which may be specialized so that the partition function is equal to the polynomial $\mathcal{KN}_w^{(\alpha, \beta, \gamma)}(\boldsymbol{x}; \lambda)$.
For our choice of boundary conditions, an admissible state can be interpreted as a grid in which $n$ distinct paths enter via the top boundary edges, travel downward and leftward, and exit the grid via the left boundary edges.
Note that our Boltzmann weights imply that only one path is permitted to traverse a horizontal edge at a time, but a vertical edge may be traversed by an unbounded number of paths.

Given a grid comprised of $n$ rows and $N+1$ columns, we define boundary conditions which are determined by a choice of permutation $w \in S_n$ and an integer partition $\mu = (\mu_1, \mu_2, \ldots, \mu_n)$ such that $N \geq \mu_1$.
We follow the convention that $\mu_i \geq \mu_{i+1}$ for $1 \leq i < n$.
To the left boundary edge of row $i$ we assign the label $n + 1 - w^{-1}(i)$.
To each top boundary edge we assign the set of numbers $n + 1 - w_0^{-1}(i)$ for all $i \in \{1, 2, \ldots, n\}$ such that the index of the column which the edge occupies is equal to $\mu_i$.
All boundary edges on the right side of the grid are assigned the label $+$ and all boundary edges on the bottom of the grid are assigned the empty set.
For an example of these boundary conditions with $\mu = (3,2,0)$ and $w = s_2 \in S_3$, see Figure \ref{w0bdry} below.
For increased legibility, we use $+$ to represent both the horizontal edge label $+$ and the vertical edge label $\varnothing$, and omit set notation in our diagrams.

\begin{remark}
The boundary conditions we describe above will be sufficient for matching the families of twisted Kirillov polynomials, so we restrict our attention to them in the present paper.
However, much more general boundary conditions are possible in our lattice model.
First, we have fixed an ordering of colors along the top boundary, but this order may be permuted.
More precisely, given $\mu$ and a permutation $w'$, we may assign to each top boundary edge the set of numbers $n + 1 - (w')^{-1}(i)$ for all $i \in \{1, 2, \ldots, n\}$ such that the index of the column which the edge occupies is equal to $\mu_i$.
We continue to let the permutation $w$ describe the left-hand boundary.
For example, let $\mu = (3, 1, 1)$, let $w = s_1 s_2$ and $w' = s_2$ be permutations in $S_3$, and let $N = 6$.
Then the resulting lattice model with boundary condition determined by $\mu$, $w$, and $w'$ is depicted in the figure below.
One may calculate that it has three admissible states.

\begin{figure}[H]
    \centering
        \begin{tikzpicture}[scale=.9]
            % GRID:
            \foreach \x in {-6, -4, -2, 0, 2, 4, 6} {
                \draw (\x, -3) to (\x, 3);
            }
            \draw[white] (4,3) to (4,2);
            \foreach \y in {-2, 0, 2} {
                \draw (-7,\y) to (7,\y);
                }
            % EMPTY EDGES:
            \foreach \x in {-5,-3,...,5} {
                \foreach \y in {-2, 0, 2} {
                \draw[fill=white] (\x,\y) circle (.3);
                }
            }
            \foreach \x in {-6,-4,...,6} {
                \foreach \y in {-1, 1} {
                \draw[fill=white] (\x,\y) circle (.3);
                }
            }
            % BOUNDARY EDGES
            \foreach \x in {-6,-4,...,6} {
                \foreach \y in {-3, 3} {
                \draw[fill=white] (\x,\y) circle (.3);
                }
                \node at (\x, -3) {$+$};
            }
            \foreach \y in {-2, 0, 2} {
                \foreach \x in {-7, 7} {
                    \draw[fill=white] (\x,\y) circle (.3);
                }
                \node at (7, \y) {$+$};
            }
            \foreach \x in {-6, -4, -2, 2, 6} {
                \node at (\x,3) {$+$};
            }
            \draw[line width=0.5mm, magenta] (0,3) to (0,2);
            \draw[line width=0.5mm, magenta, fill=white] (0,3) circle (.3);
            \node at (0,3) {$3$};
            \draw[transform canvas={xshift=-1.1pt}, line width = .5mm, blue] (4,3) to (4,2);
            \draw[transform canvas={xshift=1.1pt}, line width = .5mm, green] (4,3) to (4,2);
            \draw[line width=0.5mm, blue, fill=white] (4,3) circle (.3);
            \draw[line width=0.5mm, green] (4,3) circle (.35);
            \node at (4,3) {\small $1,2$};
            % Left side
            \draw[line width=0.5mm, green] (-7, 2) to (-6,2);
            \draw[line width=0.5mm, green, fill=white] (-7,2) circle (.3);
            \node at (-7,2) {$1$};
            \draw[line width=0.5mm, magenta] (-7, 0) to (-6,0);
            \draw[line width=0.5mm, magenta, fill=white] (-7,0) circle (.3);
            \node at (-7,0) {$3$};
            \draw[line width=0.5mm, blue] (-7, -2) to (-6,-2);
            \draw[line width=0.5mm, blue, fill=white] (-7,-2) circle (.3);
            \node at (-7,-2) {$2$};
            % ROW NUMBERS
            \node at (-8,2) {$1$};
            \node at (-8,0) {$2$};
            \node at (-8,-2) {$3$};
            % COLUMN NUMBERS
            \node at (-6, 4) {$6$};
            \node at (-4, 4) {$5$};
            \node at (-2, 4) {$4$};
            \node at (0, 4) {$3$};
            \node at (2, 4) {$2$};
            \node at (4, 4) {$1$};
            \node at (6, 4) {$0$};
        \end{tikzpicture}
\caption{Boundary condition for the lattice model determined by $\mu = (3, 1, 1)$ and the permutations $w = s_1 s_2$ and $w' = s_2$ in $S_3$, with $N = 6$.}
\end{figure}
\noindent
Boundary conditions given by two permutations were used to prove branching rules for the partition functions of a similar lattice model \cite{frozen-pipes}.

Yet more generally, as our admissible states consist of colored paths moving downward and leftward through the lattice, it is natural to permit boundaries that allow some paths to exit out the left-hand boundary and the remaining paths to exit out the bottom boundary.
The resulting model then requires an additional permutation and partition to describe the bottom boundary.
Previous explorations of similar such boundary conditions produced skew versions of the polynomials under study (see for example~\cite{wheeler-zinn-justin, borodin-wheeler} in the case of symmetric and non-symmetric Hall--Littlewood polynomials, respectively).
However, we do not know of any families of polynomials realized as the partition functions of our lattice model with the most general boundary conditions, which would also permit colored paths to enter on the right-hand boundary.
\end{remark}

In the special case when the boundary condition is determined by the identity permutation, the partition function of the resulting system can be computed explicitly in terms of $\mu$, and we show it serves as the seed of the recursion for the twisted Kirillov polynomials, up to scaling.

\begin{proposition}
\label{basecase}
    Let $\mu = (\mu_1, \mu_2, \ldots, \mu_n)$ be an integer partition with precisely $k$ distinct parts $\mu_{i_1} > \mu_{i_2} > \cdots > \mu_{i_k}$ with multiplicities $n_1, n_2, \ldots, n_k$, respectively.
    Let $N \geq \mu_1$ be a positive integer.
    Then the partition function of the lattice model with $n$ rows and $N + 1$ columns whose boundary condition is determined by $\mu$ and the identity permutation in $S_n$ is
    \begin{equation}
    \label{unique}
        \bigg(\prod_{i = 1}^k \prod_{j = 1}^{n_i} (-1)^j(\alpha \beta \cdot h_{j-3}(\alpha,\beta) + \gamma \cdot h_{j-2}(\alpha,\beta)) \bigg) \bigg(\prod_{i = 1}^{n} x_i^{N - \mu_{n + 1 - i}} \bigg).
    \end{equation}
\end{proposition}

\begin{proof}[Proof.]
    When the boundary condition is determined by the identity permutation, there is a unique admissible state.
    To conclude this, we note that the left boundary edge of row $i$ is assigned the label $n + 1 - i = w_0(i)$, so labels are assigned to the left boundary edges in decreasing order from top to bottom.
    On the other hand, the numbers $1, 2, \ldots, n$ occur along the top boundary edges of the grid in increasing order from left to right, in the sense that the label assigned to the top boundary edge in the column indexed by $\mu_i$ contains $n + 1 - w_0(i) = i$.
    Since the labels exit the left side of the grid in the opposite order in which they enter the top, it is readily deduced from the path interpretation of admissible states that there is a unique such state.
    In this state, there are only three types of vertices that can occur, and their Boltzmann weights may be inferred by means of the path interpretation of admissible states.
    The path consisting of labels $n + 1 - i$ for $i \in \{1, 2, \ldots, n\}$ travels downward only in the column indexed $\mu_{n + 1 - i}$ and leftward only in row $i$.
    The number of columns the path crosses in row $i$ is equal to $N - \mu_{n + 1 - i}$, and as seen in Figure \ref{coloredalabw}, the Boltzmann weights of vertices at these crossings are simply the spectral parameter $x_i$.
    It is similarly immediate that in each column whose top boundary edge is assigned a non-empty label $\Sigma$, there are $|\Sigma|$ vertices with respective Boltzmann weights $(-1)^j(\alpha\beta \cdot h_{j - 3 }(\alpha, \beta) + \gamma \cdot h_{j - 2}(\alpha, \beta))$ for $1 \leq j \leq |\Sigma|$.
    All other vertices in the state have Boltzmann weight $1$.
    From this we obtain the formula for the partition function given above.
\end{proof}

We may thus state a version of Theorem \ref{maintheorem} which makes explicit the computation of the twisted Kirillov polynomial $\mathcal{KN}_w^{(\alpha, \beta, \gamma)}(\boldsymbol{x}; \lambda)$ as the partition function of our lattice model.

\begin{theorem}
\label{rephrase}
    Let $\lambda = (\lambda_1, \lambda_2, \ldots, \lambda_n)$ be an integer partition and let $\boldsymbol{x} = (x_1, x_2, \ldots, x_n)$.
    Denote by $\mathcal{S}_w^{\lambda+\rho}$ the system whose grid has $n$ rows and $\lambda_1 + n$ columns, and whose boundary condition is determined by the permutation $w$ in $S_n$ and the partition $\lambda + \rho$, where $\rho = (n - 1, n-2, \ldots, 1, 0)$.
    Then
    $$
    Z(\mathcal{S}_w^{\lambda+\rho}) = \mathcal{KN}_w^{(\alpha, \beta, \gamma)}(\boldsymbol{x}; \lambda).
    $$
\end{theorem}

\noindent
The proof of this theorem is given in Section \ref{proofsection}.

As an example of Theorem \ref{rephrase}, let $\lambda = (1, 1, 0)$ so that $n = 3$, and let $w = s_2 \in S_3$.
Then $N = 3$ and $\lambda + \rho = (3, 2, 0)$, and the boundary condition for $\mathcal{S}_w^{\lambda+\rho}$ is:

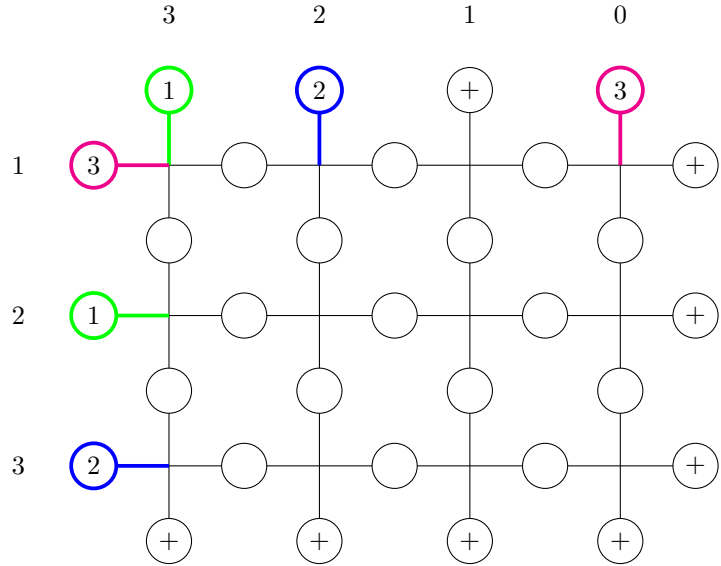
\vspace{1cm}

\begin{figure}[H]
        \centering
        \begin{tikzpicture}[scale=1.2]
            % GRID:
            \foreach \x in {0, 2, 4, 6} {
                \draw (\x, -3) to (\x, 3);
            }
            \foreach \y in {-2, 0, 2} {
                \draw (0,\y) to (7,\y);
                }
            % EMPTY EDGES:
            \foreach \x in {1,3,5} {
                \foreach \y in {-2, 0, 2} {
                \draw[fill=white] (\x,\y) circle (.3);
                }
            }
            \foreach \x in {0,2,...,6} {
                \foreach \y in {-1, 1} {
                \draw[fill=white] (\x,\y) circle (.3);
                }
            }
            % BOUNDARY EDGES
            \foreach \x in {0,2,...,6} {
                \foreach \y in {-3, 3} {
                \draw[fill=white] (\x,\y) circle (.3);
                }
                \node at (\x, -3) {$+$};
            }
            \foreach \y in {-2, 0, 2} {
                \foreach \x in {7} {
                    \draw[fill=white] (\x,\y) circle (.3);
                }
                \node at (7, \y) {$+$};
            }
            \foreach \x in {4} {
                \node at (\x,3) {$+$};
            }
            \draw[line width=0.5mm, green] (0,3) to (0,2);
            \draw[line width=0.5mm, green, fill=white] (0,3) circle (.3);
            \node at (0,3) {$1$};
            \draw[line width=0.5mm, blue] (2,3) to (2,2);
            \draw[line width=0.5mm, blue, fill=white] (2,3) circle (.3);
            \node at (2,3) {$2$};
            \draw[line width = .5mm, magenta] (6,3) to (6,2);
            \draw[line width=0.5mm, magenta, fill=white] (6,3) circle (.3);
            \node at (6,3) {$3$};
            % Left side
            \draw[line width=0.5mm, magenta] (-1, 2) to (0,2);
            \draw[line width=0.5mm, magenta, fill=white] (-1,2) circle (.3);
            \node at (-1,2) {$3$};
            \draw[line width=0.5mm, green] (-1,0) to (0,0);
            \draw[line width=0.5mm, green, fill=white] (-1,0) circle (.3);
            \node at (-1,0) {$1$};
            \draw[line width=0.5mm, blue] (-1, -2) to (0,-2);
            \draw[line width=0.5mm, blue, fill=white] (-1,-2) circle (.3);
            \node at (-1,-2) {$2$};
            % ROW NUMBERS
            \node at (-2,2) {$1$};
            \node at (-2,0) {$2$};
            \node at (-2,-2) {$3$};
            % COLUMN NUMBERS
            \node at (0, 4) {$3$};
            \node at (2, 4) {$2$};
            \node at (4, 4) {$1$};
            \node at (6, 4) {$0$};
        \end{tikzpicture}
    \end{figure}

% WHITE SPACING
\newpage

\noindent
The system has two admissible states, depicted below along with the Boltzmann weights of the vertices.

\begin{figure}[H]
    \centering
        % State 1
        \begin{tikzpicture}[scale=0.72]
            % GRID:
            \foreach \x in {0, 2, 4, 6} {
                \draw (\x, -3) to (\x, 3);
            }
            \foreach \y in {-2, 0, 2} {
                \draw (0,\y) to (7,\y);
                }
            % PATHS:
            \draw[line width=0.5mm, magenta] (0,2) to (6,2);
            \draw[line width=0.5mm, green] (0,2) to (0,0);
            \draw[line width=0.5mm, blue] (2,2) to (2,0) to (0,0) to (0,-2);
            % EMPTY EDGES:
            \foreach \x in {1,3,5} {
                \foreach \y in {-2, 0, 2} {
                \draw[fill=white] (\x,\y) circle (.3);
                }
            }
            \foreach \x in {0,2,...,6} {
                \foreach \y in {-1, 1} {
                \draw[fill=white] (\x,\y) circle (.3);
                }
            }
            % BOUNDARY EDGES
            \foreach \x in {0,2,...,6} {
                \foreach \y in {-3, 3} {
                \draw[fill=white] (\x,\y) circle (.3);
                }
                \node at (\x, -3) {$+$};
            }
            \foreach \y in {-2, 0, 2} {
                \foreach \x in {7} {
                    \draw[fill=white] (\x,\y) circle (.3);
                }
                \node at (7, \y) {$+$};
            }
            \foreach \x in {4} {
                \node at (\x,3) {$+$};
            }
            \draw[line width=0.5mm, green] (0,3) to (0,2);
            \draw[line width=0.5mm, green, fill=white] (0,3) circle (.3);
            \node at (0,3) {$1$};
            \draw[line width=0.5mm, blue] (2,3) to (2,2);
            \draw[line width=0.5mm, blue, fill=white] (2,3) circle (.3);
            \node at (2,3) {$2$};
            \draw[line width = .5mm, magenta] (6,3) to (6,2);
            \draw[line width=0.5mm, magenta, fill=white] (6,3) circle (.3);
            \node at (6,3) {$3$};
            % Left side
            \draw[line width=0.5mm, magenta] (-1, 2) to (0,2);
            \draw[line width=0.5mm, magenta, fill=white] (-1,2) circle (.3);
            \node at (-1,2) {$3$};
            \draw[line width=0.5mm, green] (-1,0) to (0,0);
            \draw[line width=0.5mm, green, fill=white] (-1,0) circle (.3);
            \node at (-1,0) {$1$};
            \draw[line width=0.5mm, blue] (-1, -2) to (0,-2);
            \draw[line width=0.5mm, blue, fill=white] (-1,-2) circle (.3);
            \node at (-1,-2) {$2$};
            % PATH EDGE NUMBERS
            \foreach \x in {1, 3, 5} {
                \draw[line width=0.5mm, magenta] (\x,2) circle (.3);
                \node at (\x, 2) {$3$};
            }
            \draw[line width=0.5mm, green] (0,1) circle (.3);
            \node at (0, 1) {$1$};
            \draw[line width=0.5mm, blue] (2,1) circle (.3);
            \node at (2, 1) {$2$};
            \draw[line width=0.5mm, blue] (1,0) circle (.3);
            \node at (1, 0) {$2$};
            \draw[line width=0.5mm, blue] (0,-1) circle (.3);
            \node at (0, -1) {$2$};
            % EMPTY INNER EDGES:
            \foreach \x in {1, 3, 5} {
                \node at (\x, -2) {$+$};
            }
            \foreach \x in {2, 4, 6} {
                \node at (\x, -1) {$+$};
            }
            \foreach \x in {3,5} {
                \node at (\x, 0) {$+$};
            }
            \foreach \x in {4,6} {
                \node at (\x, 1) {$+$};
            }
            % ROW NUMBERS
            \node at (-2,2) {$1$};
            \node at (-2,0) {$2$};
            \node at (-2,-2) {$3$};
            % COLUMN NUMBERS
            \node at (0, 4) {$3$};
            \node at (2, 4) {$2$};
            \node at (4, 4) {$1$};
            \node at (6, 4) {$0$};
            % VERTEX BOLTZMANN WEIGHTS
            \foreach \x in {0, 2, 4, 6} {
                \foreach \y in {-2, 0, 2} {
                    \path[fill=white] (\x,\y) circle (.3);
                }
            }
            \foreach \x in {0, 2, 4} {
                \node at (\x, 2) {\small $x_1$};
            }
            \node at (0,0) {\LARGE $\star$};
            \node at (6,2) {\small $1$};
            \foreach \x in {2, 4, 6} {
                \node at (\x, 0) {\small $1$};
            }
            \foreach \x in {0, 2, 4, 6} {
                \node at (\x, -2) {\small $1$};
            }
        \end{tikzpicture}
        \hspace{1cm}
        % State 2
        \begin{tikzpicture}[scale=0.72]
            % GRID:
            \foreach \x in {0, 2, 4, 6} {
                \draw (\x, -3) to (\x, 3);
            }
            \foreach \y in {-2, 0, 2} {
                \draw (0,\y) to (7,\y);
                }
            % PATHS:
            \draw[line width=0.5mm, magenta] (0,2) to (6,2);
            \draw[line width=0.5mm, green] (0,2) to (0,0);
            \draw[line width=0.5mm, blue] (2,2) to (2,-2) to (0,-2);
            % EMPTY EDGES:
            \foreach \x in {1,3,5} {
                \foreach \y in {-2, 0, 2} {
                \draw[fill=white] (\x,\y) circle (.3);
                }
            }
            \foreach \x in {0,2,...,6} {
                \foreach \y in {-1, 1} {
                \draw[fill=white] (\x,\y) circle (.3);
                }
            }
            % BOUNDARY EDGES
            \foreach \x in {0,2,...,6} {
                \foreach \y in {-3, 3} {
                \draw[fill=white] (\x,\y) circle (.3);
                }
                \node at (\x, -3) {$+$};
            }
            \foreach \y in {-2, 0, 2} {
                \foreach \x in {7} {
                    \draw[fill=white] (\x,\y) circle (.3);
                }
                \node at (7, \y) {$+$};
            }
            \foreach \x in {4} {
                \node at (\x,3) {$+$};
            }
            \draw[line width=0.5mm, green] (0,3) to (0,2);
            \draw[line width=0.5mm, green, fill=white] (0,3) circle (.3);
            \node at (0,3) {$1$};
            \draw[line width=0.5mm, blue] (2,3) to (2,2);
            \draw[line width=0.5mm, blue, fill=white] (2,3) circle (.3);
            \node at (2,3) {$2$};
            \draw[line width = .5mm, magenta] (6,3) to (6,2);
            \draw[line width=0.5mm, magenta, fill=white] (6,3) circle (.3);
            \node at (6,3) {$3$};
            % Left side
            \draw[line width=0.5mm, magenta] (-1, 2) to (0,2);
            \draw[line width=0.5mm, magenta, fill=white] (-1,2) circle (.3);
            \node at (-1,2) {$3$};
            \draw[line width=0.5mm, green] (-1,0) to (0,0);
            \draw[line width=0.5mm, green, fill=white] (-1,0) circle (.3);
            \node at (-1,0) {$1$};
            \draw[line width=0.5mm, blue] (-1, -2) to (0,-2);
            \draw[line width=0.5mm, blue, fill=white] (-1,-2) circle (.3);
            \node at (-1,-2) {$2$};
            % PATH EDGE NUMBERS
            \foreach \x in {1, 3, 5} {
                \draw[line width=0.5mm, magenta] (\x,2) circle (.3);
                \node at (\x, 2) {$3$};
            }
            \draw[line width=0.5mm, green] (0,1) circle (.3);
            \node at (0, 1) {$1$};
            \draw[line width=0.5mm, blue] (2,1) circle (.3);
            \node at (2, 1) {$2$};
            \draw[line width=0.5mm, blue] (2,-1) circle (.3);
            \node at (2, -1) {$2$};
            \draw[line width=0.5mm, blue] (1,-2) circle (.3);
            \node at (1, -2) {$2$};
            % EMPTY INNER EDGES:
            \foreach \x in {3, 5} {
                \node at (\x, -2) {$+$};
            }
            \foreach \x in {0, 4, 6} {
                \node at (\x, -1) {$+$};
            }
            \foreach \x in {1,3,5} {
                \node at (\x, 0) {$+$};
            }
            \foreach \x in {4,6} {
                \node at (\x, 1) {$+$};
            }
            % ROW NUMBERS
            \node at (-2,2) {$1$};
            \node at (-2,0) {$2$};
            \node at (-2,-2) {$3$};
            % COLUMN NUMBERS
            \node at (0, 4) {$3$};
            \node at (2, 4) {$2$};
            \node at (4, 4) {$1$};
            \node at (6, 4) {$0$};
            % VERTEX BOLTZMANN WEIGHTS
            \foreach \x in {0, 2, 4, 6} {
                \foreach \y in {-2, 0, 2} {
                    \path[fill=white] (\x,\y) circle (.3);
                }
            }
            \foreach \x in {0, 2, 4} {
                \node at (\x, 2) {\small $x_1$};
            }
            \node at (2,0) {\LARGE $\ast$};
            \node at (6,2) {\small $1$};
            \node at (0, -2) {\small $x_3$};
            \foreach \x in {0, 4, 6} {
                \node at (\x, 0) {\small $1$};
            }
            \foreach \x in {2, 4, 6} {
                \node at (\x, -2) {\small $1$};
            }
        \end{tikzpicture}
\caption{Admissible states of the system $\mathcal{S}_w^{\lambda+\rho}$ for $\lambda = (1, 1, 0)$ and $w = s_2 \in S_3$.
Here $\star = 1 + (\alpha + \gamma)x_2$ and $\ast = \gamma + (\alpha + \gamma)(\beta + \gamma)x_2$, and Boltzmann weights are overlaid on their respective vertices.}
\label{w0bdry}
\end{figure}

\noindent
Then $Z(\mathcal{S}_w^{\lambda+\rho}) = x_1^3 + (\alpha + \gamma)x_1^3x_2 + \gamma x_1^3 x_3 + (\alpha + \gamma)(\beta + \gamma)x_1^3 x_2 x_3 = T_2^{(\alpha, \beta, \gamma)}(x_1^3 x_2) = \mathcal{KN}_{s_2}^{(\alpha, \beta, \gamma)}(\boldsymbol{x}; \lambda)$.

\section{Solvability of the lattice model}

To prove Theorem~\ref{rephrase}, we will show that our lattice model is ``solvable,'' by which we mean that the Boltzmann weights of the lattice model satisfy certain Yang--Baxter equations leading to an exact solution for the partition function.\footnote{The adjective ``solvable'' is used here in place of ``exactly solved models'' as in the title of Baxter's influential monograph~\cite{Baxter}, though we mean it even more narrowly as models solved with the use of Yang--Baxter equations as opposed to other methods.
In Definition~7.5.2 of \cite{Chari-Pressley}, such models are referred to as ``integrable'' owing to connections with integrable systems whose Hamiltonians are represented by commuting transfer matrices for the lattice model.}
In this section, we recall the precise definition of these Yang--Baxter equations and then prove that our lattice model weights admit such a solution.
We offer two equivalent definitions of the solutions to the Yang--Baxter equation, one algebraic and one diagrammatic.

Beginning with the algebraic interpretation, we first recall the correspondence between Boltzmann weights and structure constants for vector space endomorphisms, adapted to our setup.
For each row $j$, we define the free vector space $V_j$ with distinguished basis elements $v_c$ indexed by the set of horizontal edge labels $c \in \{+, 1, \ldots, n\}$, and analogously for any column $k$, define the free vector space $W_k$ with basis elements $w_\Sigma$ indexed by the set of vertical edge labels $\Sigma \subseteq \{1, 2, \ldots, n \}$.
Define the vector space endomorphism\footnote{In analogy with the convention set in \cite{FRT1989}, the notation $T$ is used in place of $L$ in other closely-related literature on solvable lattice models, including \cite{BBBGIwahori, BBBG2024, metahori, brubaker-bump-friedberg2011, frozen-pipes}.} $L_{jk} \in \text{End}(V_j \otimes W_k)$ according to
$$
L_{jk} \coloneqq v_c \otimes w_\Sigma \mapsto \sum_{\substack{c' \in \{+, 1, \ldots, n\} \\
\Sigma' \subset \{1, 2, \ldots, n\}}} L^{c' \!\! , \, \Sigma'}_{\, c, \,\Sigma} v_{c'} \otimes w_{\Sigma'}.
$$
Here $L^{c' \!\! , \, \Sigma'}_{\, c, \,\Sigma}$ denotes the Boltzmann weight of the labeled vertex in row $j$ and column $k$
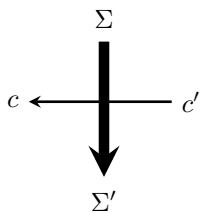
\begin{center}
\begin{tikzpicture}[anchor=base, baseline]
  \draw[line width=1.0pt,<-, >=stealth] (-1,0) node[left]{$c$} -- (.9,0) node[right]{$c'$};
  \draw[line width=4pt, ->, >=stealth] (0,.8) node[above]{$\Sigma$} -- (0,-1) node[below]{$\Sigma'$};
\end{tikzpicture}
\end{center}
with the Boltzmann weights given according to Figure \ref{coloredalabw}.

\begin{definition}
    A \textit{Yang--Baxter equation} on $V_i \otimes V_j \otimes W_k$ is an identity in $\text{End}(V_i \otimes V_j \otimes W_k)$, which for $L_{ik}$ and $L_{jk}$ as above has the form
    \begin{equation}
    \label{YBEeq}
        (R_{ij})_{12} (L_{ik})_{13} (L_{jk})_{23} = (L_{jk})_{23} (L_{ik})_{13} (R_{ij})_{12},
    \end{equation}
    where $(R_{ij})_{12}$ acts as some $R_{ij} \in \text{End}(V_i \otimes V_j)$ on the first two factors and as the identity on the third, and the remaining notation is interpreted analogously.
    If such a non-zero $R_{ij}$ exists, it is said to be a \textit{solution} of the Yang--Baxter equation.
\end{definition}
As is standard in the literature, we refer to the endomorphism on $V_i \otimes V_j$ with respect to our chosen basis as an \textit{$R$-matrix}, while either of the endomorphisms on $V_i \otimes W_k$ or $V_j \otimes W_k$ is referred to as an \textit{$L$-matrix}.
Then the Yang--Baxter equation as presented in~(\ref{YBEeq}) is sometimes referred to as the RLL relation, owing to the constituent matrices in the relation.
This algebraic formulation makes clear the connection between such matrices from quantum group module endomorphisms and solutions to the Yang--Baxter equation.
Namely, the formalism of quasi-triangular Hopf algebras guarantees that if $V_i$, $V_j$, and $W_k$ are modules for a quasi-triangular Hopf algebra and the associated $R$- and $L$-matrices $R_{ij}$, $L_{ik}$, and $L_{jk}$ for the respective modules here are induced by the universal $R$-matrix for the quasi-triangular Hopf algebra, then they satisfy~(\ref{YBEeq}); see for example Section 4.2 of~\cite{Chari-Pressley}.
Our solution to the Yang--Baxter equation falls outside this paradigm, so it is an interesting open question whether some multi-parameter generalization of a quasi-triangular Hopf algebra would produce such distinguished endomorphisms.

We now present a second equivalent definition of the Yang--Baxter equation.
In the context of our lattice model, an $R$-matrix can be described diagrammatically by introducing vertices whose four adjacent edges take labels in the set $\{+, 1, \ldots, n\}$ and whose Boltzmann weights are the matrix coefficients of $R_{ij}$ with respect to the tensor basis of $V_i \otimes V_j$.
More explicitly, if
\begin{equation}
    R_{ij}(v_a \otimes v_b) = \sum_{c, d \in \{+, 1, \ldots, n \}} R_{a,b}^{c,d} v_c \otimes v_d \label{eq:R-endomorphism}
\end{equation}
for some coefficients $R_{a,b}^{c,d}$, we may assign Boltzmann weights to a new type of vertex pictured in the identity
\begin{equation}
\label{R-vertex}
    \textrm{wt} \left(
        \begin{tikzpicture}[scale=0.7,baseline=5mm]
            \draw[line width = .5mm,red] (0,2) to [out = 0, in = 120] (1,1);
            \draw[line width = .5mm, violet] (2,2) to [out = 180, in=60] (1,1);
            \draw[line width = .5mm, blue] (0,0) to [out = 0, in = -120] (1,1);
            \draw[line width = .5mm, orange] (2,0) to [out = 180, in = -60] (1,1);
            \draw[line width=0.5mm, blue, fill=white] (0,0) circle (.35);
            \draw[line width=0.5mm, red, fill=white] (0,2) circle (.35);
            \draw[line width=0.5mm, violet, fill=white] (2,2) circle (.35);
            \draw[line width=0.5mm, orange, fill=white] (2,0) circle (.35);
            \node at (0,0) {$a$};
            \node at (0,2) {$b$};
            \node at (2,2) {$c$};
            \node at (2,0) {$d$};
        \end{tikzpicture}
    \right) = R_{a,b}^{c,d},
\end{equation}
where $R_{a,b}^{c,d}$ is defined according to $R_{ij}$ as in~\eqref{eq:R-endomorphism}.
This vertex is the result of intersecting rows $i$ and $j$ of the lattice model, which are occupied by the southwestern and northwestern edges of the vertex, respectively.
We refer to this new type of vertex as an \textit{$R$-vertex}.
The rendering of the $R$-vertex as a $45^\circ$ counter-clockwise rotation of the prior diagrams emphasizes that the Boltzmann weights depend on rows, indicates that they differ from the uniform way of defining Boltzmann weights for the vertices labeled $L_{ik}$ from a row and a column, and lastly allows us to tidily reformulate the Yang--Baxter equation as an identity among partition functions of certain three-vertex diagrams.
Indeed, in the language of lattice models, there is a solution to the Yang--Baxter equation (\ref{YBEeq}) if there exist Boltzmann weights for the $R$-vertices, not all of which are $0$, such that the partition functions of the lattice models below are equal for every choice of boundary condition $(a, b, \Sigma, c, d,\Sigma')$:

\begin{figure}[H]
    \centering
        \begin{tikzpicture}
        % Boundary condition coordinates
        \coordinate(a) at (-1,-1);
        \coordinate(b) at (-1,1);
        \coordinate(c) at (2,2);
        \coordinate(d) at (3,1);
        \coordinate(e) at (3,-1);
        \coordinate(f) at (2,-2);
        
        % Empty edge nodes
        \node(1) at (2,0) {};
        \node(2) at (1,1) {};
        \node(3) at (1,-1) {};
        
        % Vertices
        \node(R) at (0,0) {$R_{ij}$};
        \node(S) at (2,1) {$L_{ik}$};
        \node(T) at (2,-1) {$L_{jk}$};
        
        % Draw edges
        \draw (b) to [out = 0, in = 120] (R);
        \draw (a) to [out = 0, in = -120] (R);
        \draw (2) to [out = 180, in=60] (R);
        \draw (3) to [out = 180, in = -60] (R);
        \draw (c)--(S);
        \draw (2)--(S);
        \draw (1)--(S);
        \draw (d)--(S);
        \draw (3)--(T);
        \draw (1)--(T);
        \draw (f)--(T);
        \draw (e)--(T);
        
        % Circles on boundary condition
        \draw[fill=white] (b) circle (.3);
        \draw[fill=white] (a) circle (.3);
        \draw[fill=white] (c) circle (.3);
        \draw[fill=white] (f) circle (.3);
        \draw[fill=white] (d) circle (.3);
        \draw[fill=white] (e) circle (.3);
        
        % Boundary condition nodes
        \node(a) at (-1,-1) {$a$};
        \node(b) at (-1,1) {$b$};
        \node(c) at (2,2) {$\Sigma$};
        \node(d) at (3,1) {$c$};
        \node(e) at (3,-1) {$d$};
        \node(f) at (2,-2) {$\Sigma'$};
        
        % Empty edges
        \draw[fill=white] (2,0) circle (.3);
        \draw[fill=white] (1,1) circle (.3);
        \draw[fill=white] (1,-1) circle (.3);
        
        \end{tikzpicture}
        \qquad \quad \qquad \quad
        \begin{tikzpicture}
        % Boundary condition coordinates
        \coordinate(a) at (1,-1);
        \coordinate(b) at (1,1);
        \coordinate(c) at (2,2);
        \coordinate(d) at (5,1);
        \coordinate(e) at (5,-1);
        \coordinate(f) at (2,-2);
        
        % Empty edge nodes
        \node(1) at (2,0) {};
        \node(2) at (3,1) {};
        \node(3) at (3,-1) {};
        
        % Vertices
        \node(R) at (4,0) {$R_{ij}$};
        \node(T) at (2,1) {$L_{jk}$};
        \node(S) at (2,-1) {$L_{ik}$};
        
        % Draw edges
        \draw (b)--(T);
        \draw (a)--(S);
        \draw (2)--(T);
        \draw (3) to [out = 0, in = -120] (R);
        \draw (c)--(T);
        \draw (2) to [out = 0, in = 120] (R);
        \draw (1)--(S);
        \draw (d) to [out = 180, in=60] (R);
        \draw (3)--(S);
        \draw (1)--(T);
        \draw (f)--(S);
        \draw (e) to [out = 180, in = -60] (R);
        
        % Circles on boundary condition
        \draw[fill=white] (b) circle (.3);
        \draw[fill=white] (a) circle (.3);
        \draw[fill=white] (c) circle (.3);
        \draw[fill=white] (f) circle (.3);
        \draw[fill=white] (d) circle (.3);
        \draw[fill=white] (e) circle (.3);
        
        % Boundary condition nodes
        \node(a) at (1,-1) {$a$};
        \node(b) at (1,1) {$b$};
        \node(c) at (2,2) {$\Sigma$};
        \node(d) at (5,1) {$c$};
        \node(e) at (5,-1) {$d$};
        \node(f) at (2,-2) {$\Sigma'$};
        
        % Empty edges
        \draw[fill=white] (2,0) circle (.3);
        \draw[fill=white] (3,1) circle (.3);
        \draw[fill=white] (3,-1) circle (.3);
        
        \end{tikzpicture}
\caption{The Yang--Baxter equation as an equality of partition functions.
The Boltzmann weights of the vertices labeled $L_{ik}$ and $L_{jk}$ take values in Figure \ref{coloredalabw} if they are non-zero, and they are functions of the spectral parameters $x_i$ and $x_j$, respectively.
The Boltzmann weights of the vertices labeled $R_{ij}$ are functions of the spectral parameters $x_i$ and $x_j$.}
\label{YBELM}
\end{figure}

\noindent
In this case, a lattice model whose vertices have Boltzmann weights given by the matrix coefficients of the endomorphisms $L_{ik}$ is said to be \textit{solvable}.
This diagrammatic viewpoint is derived from Baxter~\cite{Baxter}, where it is referred to as the ``star-triangle'' relation for vertex models in Figure 9.3 of {\it op~cit}.
For a discussion matching our setting of colored lattice models, see for example Chapter 2 of~\cite{borodin-wheeler}.

\subsection{Solution to the Yang--Baxter equation}

A solution to the Yang--Baxter equation (\ref{YBEeq}) is expressed in terms of vertices in Figure \ref{R-weightsKirillov} below.
All other labelings are assigned a Boltzmann weight of $0$.

\begin{figure}[H]
    \centering
    \scalebox{.8}{$\begin{array}{c@{\hspace{10pt}}|c@{\hspace{10pt}}}
        \toprule
        \tt{A_1} & \tt{A_2} \\
        \midrule
        \begin{tikzpicture}[scale=0.7]
        \draw[line width = .5mm, violet] (0,0) to [out = 0, in = 180] (2,2);
        \draw[line width = .5mm, violet] (0,2) to [out = 0, in = 180] (2,0);
        \draw[line width=0.5mm, violet, fill=white] (0,0) circle (.35);
        \draw[line width=0.5mm, violet, fill=white] (0,2) circle (.35);
        \draw[line width=0.5mm, violet, fill=white] (2,2) circle (.35);
        \draw[line width=0.5mm, violet, fill=white] (2,0) circle (.35);
        \node at (0,0) {$c$};
        \node at (0,2) {$c$};
        \node at (2,2) {$c$};
        \node at (2,0) {$c$};
        \end{tikzpicture}
        &
    \begin{tikzpicture}[scale=0.7]
        \draw[line width = .5mm, black] (0,0) to [out = 0, in = 180] (2,2);
        \draw[line width = .5mm, black] (0,2) to [out = 0, in = 180] (2,0);
        \draw[line width=0.5mm, black, fill=white] (0,0) circle (.35);
        \draw[line width=0.5mm, black, fill=white] (0,2) circle (.35);
        \draw[line width=0.5mm, black, fill=white] (2,2) circle (.35);
        \draw[line width=0.5mm, black, fill=white] (2,0) circle (.35);
        \node at (0,0) {$+$};
        \node at (0,2) {$+$};
        \node at (2,2) {$+$};
        \node at (2,0) {$+$};
        \end{tikzpicture}
        \\
        \midrule
        (\alpha+\beta+\gamma)x_j + \gamma x_i + 1 + (\beta + \gamma)(\alpha + \gamma) x_i x_j &
        (\alpha+\beta+\gamma)x_i + \gamma x_j + 1 + (\beta + \gamma)(\alpha + \gamma) x_i x_j \\
        \bottomrule
        \end{array}$}
        \vspace{15pt}

    \scalebox{.8}{$\begin{array}{c@{\hspace{10pt}}|c@{\hspace{10pt}}}
        \toprule
        \tt{B} & \tt{C} \\
        \midrule
        \begin{tikzpicture}[scale=0.7]
        \draw[line width = .5mm, red] (0,0) to [out = 0, in = 180] (2,2);
        \draw[line width = .5mm, blue] (0,2) to [out = 0, in = 180] (2,0);
        \draw[line width=0.5mm, red, fill=white] (0,0) circle (.35);
        \draw[line width=0.5mm, blue, fill=white] (0,2) circle (.35);
        \draw[line width=0.5mm, red, fill=white] (2,2) circle (.35);
        \draw[line width=0.5mm, blue, fill=white] (2,0) circle (.35);
        \node at (0,0) {$a$};
        \node at (0,2) {$b$};
        \node at (2,2) {$a$};
        \node at (2,0) {$b$};
        \end{tikzpicture} \qquad
        \begin{tikzpicture}[scale=0.7]
        \draw[line width = .5mm, black] (0,0) to [out = 0, in = 180] (2,2);
        \draw[line width = .5mm, violet] (0,2) to [out = 0, in = 180] (2,0);
        \draw[line width=0.5mm, black, fill=white] (0,0) circle (.35);
        \draw[line width=0.5mm, violet, fill=white] (0,2) circle (.35);
        \draw[line width=0.5mm, black, fill=white] (2,2) circle (.35);
        \draw[line width=0.5mm, violet, fill=white] (2,0) circle (.35);
        \node at (0,0) {$+$};
        \node at (0,2) {$c$};
        \node at (2,2) {$+$};
        \node at (2,0) {$c$};
        \end{tikzpicture}
        &
        \begin{tikzpicture}[scale=0.7]
        \draw[line width = .5mm, blue] (0,0) to [out = 0, in = 180] (2,2);
        \draw[line width = .5mm, red] (0,2) to [out = 0, in = 180] (2,0);
        \draw[line width=0.5mm, blue, fill=white] (0,0) circle (.35);
        \draw[line width=0.5mm, red, fill=white] (0,2) circle (.35);
        \draw[line width=0.5mm, blue, fill=white] (2,2) circle (.35);
        \draw[line width=0.5mm, red, fill=white] (2,0) circle (.35);
        \node at (0,0) {$b$};
        \node at (0,2) {$a$};
        \node at (2,2) {$b$};
        \node at (2,0) {$a$};
        \end{tikzpicture} \qquad
        \begin{tikzpicture}[scale=0.7]
        \draw[line width = .5mm, violet] (0,0) to [out = 0, in = 180] (2,2);
        \draw[line width = .5mm, black] (0,2) to [out = 0, in = 180] (2,0);
        \draw[line width=0.5mm, violet, fill=white] (0,0) circle (.35);
        \draw[line width=0.5mm, black, fill=white] (0,2) circle (.35);
        \draw[line width=0.5mm, violet, fill=white] (2,2) circle (.35);
        \draw[line width=0.5mm, black, fill=white] (2,0) circle (.35);
        \node at (0,0) {$c$};
        \node at (0,2) {$+$};
        \node at (2,2) {$c$};
        \node at (2,0) {$+$};
        \end{tikzpicture}
        \\
        \midrule
        x_j-x_i & \alpha \beta (x_j -x_i) \\
        \bottomrule
        \end{array}$}
        \vspace{15pt}

    \scalebox{.8}{$\begin{array}{c@{\hspace{10pt}}|c@{\hspace{10pt}}|c@{\hspace{10pt}}|c@{\hspace{10pt}}}
        \toprule
        \tt{D_1} & \tt{E_1} & \tt{D_2} & \tt{E_2} \\
        \midrule
        \begin{tikzpicture}[scale=0.7]
        \draw[line width = .5mm,blue] (0,2) to [out = 0, in = 120] (1,1);
        \draw[line width = .5mm, blue] (2,2) to [out = 180, in=60] (1,1);
        \draw[line width = .5mm, red] (0,0) to [out = 0, in = -120] (1,1);
        \draw[line width = .5mm, red] (2,0) to [out = 180, in = -60] (1,1);
        \draw[line width=0.5mm, red, fill=white] (0,0) circle (.35);
        \draw[line width=0.5mm, blue, fill=white] (0,2) circle (.35);
        \draw[line width=0.5mm, blue, fill=white] (2,2) circle (.35);
        \draw[line width=0.5mm, red, fill=white] (2,0) circle (.35);
        \node at (0,0) {$a$};
        \node at (0,2) {$b$};
        \node at (2,2) {$b$};
        \node at (2,0) {$a$};
        \end{tikzpicture}
        &
        \begin{tikzpicture}[scale=0.7]
        \draw[line width = .5mm,red] (0,2) to [out = 0, in = 120] (1,1);
        \draw[line width = .5mm, red] (2,2) to [out = 180, in=60] (1,1);
        \draw[line width = .5mm, blue] (0,0) to [out = 0, in = -120] (1,1);
        \draw[line width = .5mm, blue] (2,0) to [out = 180, in = -60] (1,1);
        \draw[line width=0.5mm, blue, fill=white] (0,0) circle (.35);
        \draw[line width=0.5mm, red, fill=white] (0,2) circle (.35);
        \draw[line width=0.5mm, red, fill=white] (2,2) circle (.35);
        \draw[line width=0.5mm, blue, fill=white] (2,0) circle (.35);
        \node at (0,0) {$b$};
        \node at (0,2) {$a$};
        \node at (2,2) {$a$};
        \node at (2,0) {$b$};
        \end{tikzpicture}
        &
        \begin{tikzpicture}[scale=0.7]
        \draw[line width = .5mm,violet] (0,2) to [out = 0, in = 120] (1,1);
        \draw[line width = .5mm, violet] (2,2) to [out = 180, in=60] (1,1);
        \draw[line width = .5mm, black] (0,0) to [out = 0, in = -120] (1,1);
        \draw[line width = .5mm, black] (2,0) to [out = 180, in = -60] (1,1);
        \draw[line width=0.5mm, black, fill=white] (0,0) circle (.35);
        \draw[line width=0.5mm, violet, fill=white] (0,2) circle (.35);
        \draw[line width=0.5mm, violet, fill=white] (2,2) circle (.35);
        \draw[line width=0.5mm, black, fill=white] (2,0) circle (.35);
        \node at (0,0) {$+$};
        \node at (0,2) {$c$};
        \node at (2,2) {$c$};
        \node at (2,0) {$+$};
        \end{tikzpicture}
        &
        \begin{tikzpicture}[scale=0.7]
        \draw[line width = .5mm,black] (0,2) to [out = 0, in = 120] (1,1);
        \draw[line width = .5mm, black] (2,2) to [out = 180, in=60] (1,1);
        \draw[line width = .5mm, violet] (0,0) to [out = 0, in = -120] (1,1);
        \draw[line width = .5mm, violet] (2,0) to [out = 180, in = -60] (1,1);
        \draw[line width=0.5mm, violet, fill=white] (0,0) circle (.35);
        \draw[line width=0.5mm, black, fill=white] (0,2) circle (.35);
        \draw[line width=0.5mm, black, fill=white] (2,2) circle (.35);
        \draw[line width=0.5mm, violet, fill=white] (2,0) circle (.35);
        \node at (0,0) {$c$};
        \node at (0,2) {$+$};
        \node at (2,2) {$+$};
        \node at (2,0) {$c$};
        \end{tikzpicture}
        \\
        \midrule
        (1+(\beta+\gamma) x_j)(1+(\alpha+\gamma) x_i) & (1 + (\beta+\gamma) x_i)(1 + (\alpha+\gamma) x_j) & (1+(\beta+\gamma) x_i)(1+(\alpha+\gamma) x_i) & (1 + (\beta+\gamma) x_j)(1 + (\alpha+\gamma) x_j)\\
        \bottomrule
        \end{array}$}
        \vspace{.5\baselineskip}
        \caption{Solution to the Yang--Baxter equation (\ref{YBEeq}).
       Here $a < b$ and $c$ is any color.}
        \label{R-weightsKirillov}
\end{figure}

\noindent
Before describing the proof that these Boltzmann weights satisfy the RLL relation (\ref{YBEeq}), we prove another parametrized Yang--Baxter equation also known as the RRR relation.
As noted in the previous section, either of these two types of relations would follow from identifying our Boltzmann weights with the matrix coefficients of a module endomorphism for a quasi-triangular Hopf algebra.
Since we are without such a connection, we must resort to other methods.
While we do not directly use the RRR relation in the determination of the partition function or its subsequent applications, we record it here because its existence is further evidence of a sought-after link between our lattice model and quasi-triangular Hopf algebras.

\begin{proposition}
    Let $V_i$ denote the $(n+1)$-dimensional complex vector space with basis indexed by $\{+, 1, \ldots n \}$, where $+$ represents the distinguished ``uncolor.''
    Let $R(x_i, x_j)$ denote the endomorphism of $V_i \otimes V_j$ whose entries are given by the Boltzmann weights in Figure~\ref{R-weightsKirillov}.
    Then the matrices $R(x_i, x_j)$ satisfy the parametrized Yang--Baxter equation on $V_{1} \otimes V_{2} \otimes V_{3}$:
    \begin{equation}
        R_{12}(x_1, x_2) R_{13}(x_1, x_3) R_{23}(x_2, x_3) = R_{23}(x_2, x_3) R_{13}(x_1, x_3) R_{12}(x_1, x_2).
    \label{rrrequation}
    \end{equation}
    Here $R_{ij}(x_i, x_j)$ denotes the endomorphism of $V_{1} \otimes V_{2} \otimes V_{3}$ which acts as $R(x_i, x_j)$ on the $i$-th and $j$-th tensor factors and as the identity on the remaining factor.
\end{proposition}

\begin{proof}[Proof.]
    The general case can be reduced to the case $n = 3$, as follows.
    Every vertex in Figure~\ref{R-weightsKirillov} has a color-conservation property in the sense that at each vertex, the set of labels on the pair of left-hand edges is the set of labels on the pair of right-hand edges.
    The relation~(\ref{rrrequation}) may be viewed diagrammatically as another identity of partition functions of three-vertex lattice models over all choices of boundary condition, depicted in Figure \ref{RRRLM}.
    These boundary edges may be separated into three input edges on the left and three output edges on the right as shown below:

    \begin{figure}[H]
        \centering
            \begin{tikzpicture}[baseline={(0,-0.8)}, scale=0.8, round/.style={draw, circle, fill=white, inner sep=0pt, minimum size=14pt}, dot/.style={circle, fill, inner sep=0pt, minimum size=3pt}]
            \draw (-1,1) node [round] {$c$} to [out = 0, in = 180] (1,-1) node[round] {} to [out = 0, in = 180] (3,-3) node[round] {$f$};
            \draw (1,1) node [round] {$d$} to [out = 180, in = 0] (-1,-1) node[round]{} to [out = 180, in = 0] (-3,-3) node[round] {$a$};
            \draw (-3,-1) node [round] {$b$} to [out = 0, in = 180] (-1,-3) -- (1,-3) node[midway, round] {} to [out = 0, in = 180] (3,-1) node[round] {$e$};
            \path[fill=white] (2,-2) circle (.3);
            \path[fill=white] (-2,-2) circle (.3);
            \path[fill=white] (0,0) circle (.3);
            \node at (-2, -2) {$R_{12}$};
            \node at (2,-2) {$R_{23}$};
            \node at (0,0) {$R_{13}$};
            \end{tikzpicture}
            \qquad \qquad \qquad \qquad
            \begin{tikzpicture}[scale=0.8, yscale=-1, baseline={(0,.8)}, round/.style={draw, circle, fill=white, inner sep=0pt, minimum size=14pt}, dot/.style={circle, fill, inner sep=0pt, minimum size=3pt}]
            \draw (-1,1) node [round] {$a$} to [out = 0, in = 180] (1,-1) node[round] {} to [out = 0, in = 180] (3,-3) node[round] {$d$};
            \draw (1,1) node [round] {$f$} to [out = 180, in = 0] (-1,-1) node[round]{} to [out = 180, in = 0] (-3,-3) node[round] {$c$};
            \draw (-3,-1) node [round] {$b$} to [out = 0, in = 180] (-1,-3) -- (1,-3) node[midway, round] {} to [out = 0, in = 180] (3,-1) node[round] {$e$};
            \path[fill=white] (2,-2) circle (.3);
            \path[fill=white] (-2,-2) circle (.3);
            \path[fill=white] (0,0) circle (.3);
            \node at (2, -2) {$R_{12}$};
            \node at (-2,-2) {$R_{23}$};
            \node at (0,0) {$R_{13}$};
            \end{tikzpicture}
            \vspace{1\baselineskip}
            \caption{The RRR relation is represented by an equality of the partition functions of the above lattice models for every choice of boundary condition $a, b, c, d, e, f$.}
    \label{RRRLM}
    \end{figure}

    \noindent
    Thus, by the color-conservation property, the three input edge labels on the left boundary in any admissible state must coincide with the three output edge labels on the right boundary.
    A further examination of the Boltzmann weights in Figure~\ref{R-weightsKirillov} shows that they do not depend upon the individual edge labels themselves, but only upon whether they are $+$ or colors, and the relative ordering of the colors.
    Thus we need only consider an arbitrary multiset of three edge labels taken from a set of three colors and the ``uncolor'' $+$, rather than all possible multisets of three edge labels.
    For any choice of three colors, (\ref{rrrequation}) reduces to an identity in the triple tensor product of four-dimensional vector spaces $V_i$, or equivalently, an identity of products of three $64 \times 64$ matrices each made from embedded $16 \times 16$ matrices with entries from Figure~\ref{R-weightsKirillov}.
    This may be readily checked using computer software.
\end{proof}

Alternatively, one may prove the previous result using Theorem 5.4 of \cite{Polymath22}.
That theorem characterizes the solutions to the colored Yang--Baxter equation according to a set of algebraic conditions on the Boltzmann weights.
For example, the authors in \cite{Polymath22} define invariants $\Delta_{i,j}(x,y)$ for all $i,j \in \{0, \ldots, n\}$ which are generalizations of Baxter's invariant $\Delta$, a quotient of two quadratic expressions in the Boltzmann weights that governs the solvability of the six-vertex model.
One necessary condition for the solvability is the equality of $\Delta_{i,j}(x_1, x_3) = \Delta_{i,j}(x_2, x_3)$ for all $i,j \in \{0, \ldots, n\}$.
In particular, these invariants should be independent of the transcendental parameters $x_i$.
The uniformity of the Boltzmann weights allows us to compute that, with weights as in Figure~\ref{R-weightsKirillov}, the invariant $\Delta_{i,j}(x, y)$ is either $-(\alpha + \beta)$ or $-(\alpha + \beta) / \alpha \beta$, depending on whether $i < j$ or $i > j$, respectively.
Hence the required identity of $\Delta_{i,j}$'s is immediately satisfied since they are all independent of spectral parameters and agree for any choice of $i, j$.
The remaining conditions for solvability and the resulting parametrized solution may be checked similarly.

We now come to the primary tool, the so-called RLL relation asserting that the Boltzmann weights in Figure \ref{R-weightsKirillov} satisfy the Yang--Baxter equation (\ref{YBEeq}).
This result is far more challenging to prove than the RRR relation owing to the fact that the vertical edges are allowed to carry labels made with an unbounded number of distinct colors.

\begin{theorem}
\label{solvable}
    A lattice model whose vertices have Boltzmann weights as in Figure \ref{coloredalabw} is solvable.
    That is, there exists a non-zero Boltzmann weight function for the $R$-vertices such that the partition functions of the diagrams in Figure \ref{YBELM} are equal for all $a, b, c, d \in \{+, 1, \ldots, n\}$ and $\Sigma, \Sigma' \subset \{1, 2, \ldots, n\}$.
\end{theorem}

In the literature, similar Yang--Baxter equations are verified by reducing the argument to one involving only vertices whose vertical edges are labeled by at most one element of the set $\{+, 1, \ldots, n\}$, such as in \cite{ABW2023, borodin-wheeler, BBBGIwahori, bump2024colored}.
The methods in these references are collectively known as ``fusion,'' which, roughly speaking, allows for complicated Boltzmann weights at each vertex to be seen as arising from the collapsing of a partition function for another lattice model with more rows or columns, but with simpler Boltzmann weights.
Algebraically this fusion corresponds to the tensoring of quantum group modules, and is sometimes composed with a projection operator onto an irreducible subquotient.
Even in the absence of a quantum group module interpretation for the Boltzmann weights, one can use formal versions of this fusion process, especially when the Boltzmann weights arise simply from the collapse of sets of $n$ adjacent columns with no subsequent projection operator, as was done in \cite{BBBGIwahori, bump2024colored}.
However, these latter formal methods of proof are not applicable to the vertices in Figure \ref{coloredalabw} because the Boltzmann weights do not factor appropriately, and so do not admit a simple column-based fusion.
It is a challenging problem to determine whether our $L$-matrix Boltzmann weights in Figure~\ref{coloredalabw} arise from a more general fusion process.
Instead, a proof of Theorem \ref{solvable} can be given by a rather daunting computation that verifies an equality of partition functions of the lattice models in Figure \ref{YBELM} for every choice of boundary condition.

Although $n$ is unbounded, we show that the calculations proceed independently of $n$ as they depend only upon the relative ordering of the labels on the horizontal edges and whether they are contained in the labels on the vertical edges.
This results in a finite list of cases handled uniformly for any choice of $n$ and any set of colors on the vertical boundary edges.
It is possible to treat all such cases by hand computation, and indeed we essentially did this in honing and verifying our presentation for the Boltzmann weights of the $R$-vertices.
In the end, we find it more satisfactory and easily verifiable to write computer code to check the solution.
To give the reader a feel for the computation required, we show how to check one choice of boundary condition in detail.

\begin{proof}[Proof sketch.]
    Let $a$, $b$, $\Sigma$, $c$, $d$, $\Sigma'$ denote the labels of the boundary edges as in Figure \ref{YBELM}.
    It suffices to consider only boundary conditions for which an admissible state exists, since otherwise the Yang--Baxter equation is trivially satisfied.
    If a boundary condition is chosen so that there exists an admissible state, then paths travel only downward and leftward, and thus $\Sigma$ cannot contain elements of $\{c, d\} \setminus \{a, b\}$ but may contain elements of $\{a, b\}$, and $\Sigma'$ cannot contain elements of $\{a, b\} \setminus \{c, d\}$ but may contain elements of $\{c, d\}$.
    Thus, to verify the Yang--Baxter equation, it suffices to consider all choices of $a$, $b$, $\Sigma$, $c$, $d$, $\Sigma'$ such that $\Sigma$ contains a subset of $\{a, b\} \setminus \{+\}$ and no elements of $\{c, d\} \setminus \{a, b\}$ and $\Sigma'$ contains a subset of $\{c, d\} \setminus \{+\}$ and no elements of $\{a, b\} \setminus \{c, d\}$.

    The Boltzmann weights in Figure \ref{coloredalabw} are determined by the cardinality of the northern edge label, the cardinalities of the subsets of the northern edge label consisting of elements greater than the horizontal edge labels, the relative ordering of the horizontal edge labels, and whether the horizontal edge labels are contained in the northern edge label.
    In our computations, we thus reduce a choice of boundary condition to the following data.
    A choice of horizontal boundary edge labels is a choice of which of $a$, $b$, $c$, $d$ is equal to $+$ and an ordering on those which are not equal to $+$.
    A choice of $\Sigma$ is a subset of $\{a, b\} \setminus \{+\}$ and a collection of variables representing $|\Sigma|$ and $|\Sigma_{[l + 1, n]}|$ for $l \in \{a, b, c, d\} \setminus \{+\}$.
    Analogously, $\Sigma'$ is represented by a subset of $\{c, d\} \setminus \{+\}$ and a collection of variables representing cardinalities.
    We observe that in an admissible state, $\Sigma'$ may be obtained from $\Sigma$ by adding or removing labels in $\{a, b, c, d\} \setminus \{+\}$, and so the variables that define $\Sigma'$ may be expressed in terms of the variables that define $\Sigma$.
    Likewise, the label of an internal vertical edge in an admissible state may be represented in terms of the variables defining $\Sigma$.

    It is seen from Figure \ref{coloredalabw} that the labels of the western, northern, and eastern edges of a vertex uniquely determine the label of the southern edge such that the Boltzmann weight of the vertex is non-zero, if it exists.
    Again using the path dynamics, we see that for each diagram in Figure \ref{YBELM}, there are at most two choices of label for the internal vertical edge, and they are determined by the respective choices of labelings for the $R$-vertex in the diagram.
    Thus, verifying the Yang--Baxter equation essentially amounts to computing the Boltzmann weights of at most four states and testing two sums for equality.

    To handle these computations symbolically, we also treat the complete homogeneous symmetric functions $h_m(\alpha,\beta)$ as a set of variables.
    We observe that only $h_{|\Sigma| - i}(\alpha, \beta)$ for $i = 0, 1, 2, 3, 4$ may arise in the Boltzmann weights of vertices in the diagrams in Figure \ref{YBELM} by noting that the cardinality of a northern edge label of a vertex must have cardinality at least $|\Sigma| - 1$ and at most $|\Sigma| + 1$ in an admissible state, and then calculating the expressions ($\dagger$) and ($\ddagger$).
    We use the usual relation $h_m(\alpha, \beta) = \alpha h_{m -1}(\alpha, \beta) + \beta^m$ to establish a dependence among the variables, which may all be expressed in terms of $h_{|\Sigma| - 4}(\alpha, \beta)$.

    In this way, the Boltzmann weight of each state may be expressed symbolically in terms of $|\Sigma|$, $|\Sigma_{[l + 1, n]}|$ for $l \in \{a, b, c, d\} \setminus \{+\}$, and $h_{|\Sigma| - 4}(\alpha, \beta)$, and confirming the Yang--Baxter equation holds is reduced to a matter of testing two multivariable expressions for equality for each choice of boundary condition.
    For example, let $a, d \in \{1, 2, \ldots, n\}$ such that $a < d$, let $\Sigma \subset \{1, 2, \ldots, n\}$ be a set containing $a$ but not $d$, and recall that $\Sigma^{+-}_{da} \coloneqq (\Sigma \cup \{d \}) \setminus \{a\}$.
    Then one case of the Yang--Baxter equation is verified by proving equality of the partition functions of the diagrams below.
    
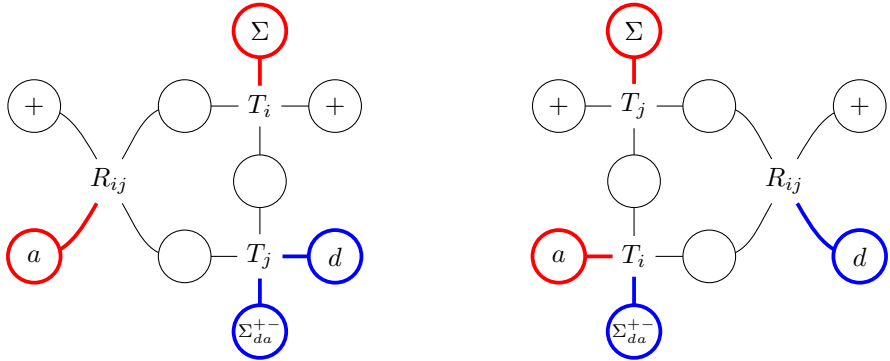
\begin{figure}[H]
        \centering
            \begin{tikzpicture}
            % Boundary condition coordinates
            \coordinate(a) at (-1,-1);
            \coordinate(b) at (-1,1);
            \coordinate(c) at (2,2);
            \coordinate(d) at (3,1);
            \coordinate(e) at (3,-1);
            \coordinate(f) at (2,-2);
            
            % Empty edge nodes
            \node(1) at (2,0) {};
            \node(2) at (1,1) {};
            \node(3) at (1,-1) {};
            
            % Vertices
            \node(R) at (0,0) {$R_{ij}$};
            \node(S) at (2,1) {$L_{ik}$};
            \node(T) at (2,-1) {$L_{jk}$};
            
            % Draw edges
            \draw (b) to [out = 0, in = 120] (R);
            \draw[line width = .5mm,red] (a) to [out = 0, in = -120] (R);
            \draw (2) to [out = 180, in=60] (R);
            \draw (3) to [out = 180, in = -60] (R);
            \draw[line width = .5mm, red] (c)--(S);
            \draw (2)--(S);
            \draw (1)--(S);
            \draw (d)--(S);
            \draw (3)--(T);
            \draw (1)--(T);
            \draw[line width = .5mm,blue] (f)--(T);
            \draw[line width = .5mm, blue] (e)--(T);
            
            % Circles on boundary condition
            \draw[fill=white] (b) circle (.35);
            \draw[line width = .5mm, red, fill=white] (a) circle (.35);
            \draw[line width = .5mm, red, fill=white] (c) circle (.35);
            \draw[line width = .5mm, blue,fill=white] (f) circle (.35);
            \draw[fill=white] (d) circle (.35);
            \draw[line width = .5mm, blue, fill=white] (e) circle (.35);
            
            % Boundary condition nodes
            \node(a) at (-1,-1) {$a$};
            \node(b) at (-1,1) {$+$};
            \node(c) at (2,2) {$\Sigma$};
            \node(d) at (3,1) {$+$};
            \node(e) at (3,-1) {$d$};
            \node(f) at (2,-2) {\scriptsize $\Sigma_{da}^{+-}$};
            
            % Empty edges
            \draw[fill=white] (2,0) circle (.35);
            \draw[fill=white] (1,1) circle (.35);
            \draw[fill=white] (1,-1) circle (.35);
                        
            \end{tikzpicture}
            \qquad \quad \qquad \quad
            \begin{tikzpicture}
            % Boundary condition coordinates
            \coordinate(a) at (1,-1);
            \coordinate(b) at (1,1);
            \coordinate(c) at (2,2);
            \coordinate(d) at (5,1);
            \coordinate(e) at (5,-1);
            \coordinate(f) at (2,-2);
            
            % Empty edge nodes
            \node(1) at (2,0) {};
            \node(2) at (3,1) {};
            \node(3) at (3,-1) {};
            
            % Vertices
            \node(R) at (4,0) {$R_{ij}$};
            \node(T) at (2,1) {$L_{jk}$};
            \node(S) at (2,-1) {$L_{ik}$};
            
            % Draw edges
            \draw (b)--(T);
            \draw[line width = .5mm, red] (a)--(S);
            \draw (2)--(T);
            \draw (3) to [out = 0, in = -120] (R);
            \draw[line width = .5mm,red] (c)--(T);
            \draw (2) to [out = 0, in = 120] (R);
            \draw (1)--(S);
            \draw (d) to [out = 180, in=60] (R);
            \draw (3)--(S);
            \draw (1)--(T);
            \draw[line width = .5mm, blue] (f)--(S);
            \draw[line width = .5mm, blue] (e) to [out = 180, in = -60] (R);
            
            % Circles on boundary condition
            \draw[fill=white] (b) circle (.35);
            \draw[line width = .5mm, red, fill=white] (a) circle (.35);
            \draw[line width = .5mm, red, fill=white] (c) circle (.35);
            \draw[line width = .5mm, blue, fill=white] (f) circle (.35);
            \draw[fill=white] (d) circle (.35);
            \draw[line width = .5mm, blue, fill=white] (e) circle (.35);
            
            % Boundary condition nodes
            \node(a) at (1,-1) {$a$};
            \node(b) at (1,1) {$+$};
            \node(c) at (2,2) {$\Sigma$};
            \node(d) at (5,1) {$+$};
            \node(e) at (5,-1) {$d$};
            \node(f) at (2,-2) {\scriptsize$\Sigma_{da}^{+-}$};
            
            % Empty edges
            \draw[fill=white] (2,0) circle (.35);
            \draw[fill=white] (3,1) circle (.35);
            \draw[fill=white] (3,-1) circle (.35);
                        
            \end{tikzpicture}
    \end{figure}

    \noindent
    The Boltzmann weights of the vertices labeled $R_{ij}$ are functions of the spectral parameters $x_i$ and $x_j$ and take values given in Figure \ref{R-weightsKirillov} if they are non-zero.
    Similarly, Boltzmann weights of the vertices labeled $L_{ik}$ and $L_{jk}$ take values given in Figure \ref{coloredalabw} if they are non-zero and are functions of the spectral parameters $x_i$ and $x_j$, respectively.

    The two admissible states and respective Boltzmann weights of the left-hand diagram are:

    \begin{figure}[H]
        \centering
        \begin{tikzpicture}
            % Boundary condition coordinates
            \coordinate(a) at (-1,-1);
            \coordinate(b) at (-1,1);
            \coordinate(c) at (2,2);
            \coordinate(d) at (3,1);
            \coordinate(e) at (3,-1);
            \coordinate(f) at (2,-2);
            
            % Empty edge nodes
            \node(1) at (2,0) {};
            \node(2) at (1,1) {};
            \node(3) at (1,-1) {};
            
            % Vertices
            \node(R) at (0,0) {$R_{ij}$};
            \node(S) at (2,1) {$L_{ik}$};
            \node(T) at (2,-1) {$L_{jk}$};
            
            % Draw edges
            \draw (b) to [out = 0, in = 120] (R);
            \draw[line width = .5mm,red] (a) to [out = 0, in = -120] (R);
            \draw[line width = .5mm,red] (2) to [out = 180, in=60] (R);
            \draw (3) to [out = 180, in = -60] (R);
            \draw[line width = .5mm,red] (c)--(S);
            \draw[line width = .5mm,red] (2)--(S);
            \draw (1)--(S);
            \draw (d)--(S);
            \draw (3)--(T);
            \draw (1)--(T);
            \draw[line width = .5mm,blue] (f)--(T);
            \draw[line width = .5mm,blue] (e)--(T);
            
            % Circles on boundary condition
            \draw[fill=white] (b) circle (.35);
            \draw[line width = .5mm, red, fill=white] (a) circle (.35);
            \draw[line width = .5mm, red, fill=white] (c) circle (.35);
            \draw[line width = .5mm, blue, fill=white] (f) circle (.35);
            \draw[fill=white] (d) circle (.35);
            \draw[line width = .5mm, blue, fill=white] (e) circle (.35);
            
            % Boundary condition nodes
            \node(a) at (-1,-1) {$a$};
            \node(b) at (-1,1) {$+$};
            \node(c) at (2,2) {$\Sigma$};
            \node(d) at (3,1) {$+$};
            \node(e) at (3,-1) {$d$};
            \node(f) at (2,-2) {\scriptsize $\Sigma_{da}^{+-}$};
            
            % Empty edges
            \draw[fill=white] (2,0) circle (.35);
            \draw[line width = .5mm, red, fill=white] (1,1) circle (.35);
            \draw[fill=white] (1,-1) circle (.35);

            \node at (2,0)  {$\Sigma_a^-$};
            \node at (1,1)  {$a$};
            \node at (1,-1) {$+$};
            
            \end{tikzpicture}

            \,\

            \scriptsize
            $\underbrace{\alpha \beta (x_j - x_i)}_{R_{ij}} \cdot \underbrace{(-1)^{|\Sigma|}(-\beta)^{|\Sigma_{[a+1, n]}|}(\alpha \beta \cdot h_{|\Sigma| - 3} + \gamma \cdot h_{|\Sigma| - 2})}_{L_{ik}} \cdot \underbrace{(-\alpha)^{|(\Sigma_a^-)_{[d+1, n]}|}(1 + (\alpha + \gamma)x_j) (1 + (\beta + \gamma)x_j)}_{L_{jk}}$
            \end{figure}

            \begin{figure}[H]
            \centering
            \begin{tikzpicture}
            % Boundary condition coordinates
            \coordinate(a) at (-1,-1);
            \coordinate(b) at (-1,1);
            \coordinate(c) at (2,2);
            \coordinate(d) at (3,1);
            \coordinate(e) at (3,-1);
            \coordinate(f) at (2,-2);
            
            % Empty edge nodes
            \node(1) at (2,0) {};
            \node(2) at (1,1) {};
            \node(3) at (1,-1) {};
            
            % Vertices
            \node(R) at (0,0) {$R_{ij}$};
            \node(S) at (2,1) {$L_{ik}$};
            \node(T) at (2,-1) {$L_{jk}$};
            
            % Draw edges
            \draw (b) to [out = 0, in = 120] (R);
            \draw[line width = .5mm, red] (a) to [out = 0, in = -120] (R);
            \draw (2) to [out = 180, in=60] (R);
            \draw[line width = .5mm, red] (3) to [out = 180, in = -60] (R);
            \draw[line width = .5mm, red] (c)--(S);
            \draw (2)--(S);
            \draw[line width = .5mm, red] (1)--(S);
            \draw (d)--(S);
            \draw[line width = .5mm, red] (3)--(T);
            \draw[line width = .5mm, red] (1)--(T);
            \draw[line width = .5mm, blue] (f)--(T);
            \draw[line width = .5mm, blue] (e)--(T);
            
            % Circles on boundary condition
            \draw[fill=white] (b) circle (.35);
            \draw[line width = .5mm, red, fill=white] (a) circle (.35);
            \draw[line width = .5mm, red, fill=white] (c) circle (.35);
            \draw[line width = .5mm, blue, fill=white] (f) circle (.35);
            \draw[fill=white] (d) circle (.35);
            \draw[line width = .5mm, blue, fill=white] (e) circle (.35);
            
            % Boundary condition nodes
            \node(a) at (-1,-1) {$a$};
            \node(b) at (-1,1) {$+$};
            \node(c) at (2,2) {$\Sigma$};
            \node(d) at (3,1) {$+$};
            \node(e) at (3,-1) {$d$};
            \node(f) at (2,-2) {\scriptsize $\Sigma_{da}^{+-}$};
            
            % Empty edges
            \draw[line width = .5mm, red, fill=white] (2,0) circle (.35);
            \draw[fill=white] (1,1) circle (.35);
            \draw[line width = .5mm, red, fill=white] (1,-1) circle (.35);
            
            \node at (2,0)  {$\Sigma$};
            \node at (1,1)  {$+$};
            \node at (1,-1) {$a$};
            
            \end{tikzpicture}
            
            \,\
            
            \scriptsize
            $\underbrace{(1 + (\beta + \gamma)x_j)(1 + (\alpha + \gamma)x_j)}_{R_{ij}} \cdot \underbrace{(-1)^{|\Sigma| + 1} (((\alpha + \gamma)(\beta + \gamma)x_i + \gamma) \cdot h_{|\Sigma| - 1} + \alpha \beta \cdot h_{|\Sigma| - 2})}_{L_{ik}} \cdot \underbrace{(-\alpha)^{|\Sigma_{[d+1, n]}|}(-\beta)^{|\Sigma_{[a + 1, n]}|} (1 + (\alpha + \gamma)x_j)}_{L_{jk}}$
    \end{figure}

    \noindent
    Similarly, the admissible states and Boltzmann weights of the right-hand diagram are:

    
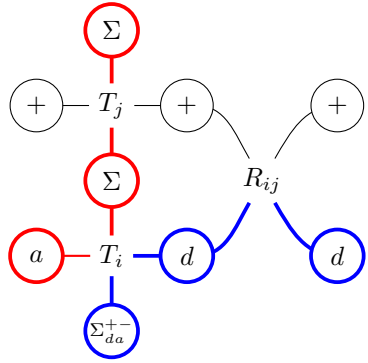
\begin{figure}[H]
        \centering
            \begin{tikzpicture}
            % Boundary condition coordinates
            \coordinate(a) at (1,-1);
            \coordinate(b) at (1,1);
            \coordinate(c) at (2,2);
            \coordinate(d) at (5,1);
            \coordinate(e) at (5,-1);
            \coordinate(f) at (2,-2);
            
            % Empty edge nodes
            \node(1) at (2,0) {};
            \node(2) at (3,1) {};
            \node(3) at (3,-1) {};
            
            % Vertices
            \node(R) at (4,0) {$R_{ij}$};
            \node(T) at (2,1) {$L_{jk}$};
            \node(S) at (2,-1) {$L_{ik}$};
            
            % Draw edges
            \draw (b)--(T);
            \draw[thick, red] (a)--(S);
            \draw (2)--(T);
            \draw[line width = .5mm, blue] (3) to [out = 0, in = -120] (R);
            \draw[line width = .5mm, red] (c)--(T);
            \draw (2) to [out = 0, in = 120] (R);
            \draw[line width = .5mm, red] (1)--(S);
            \draw (d) to [out = 180, in=60] (R);
            \draw[line width = .5mm, blue] (3)--(S);
            \draw[line width = .5mm, red] (1)--(T);
            \draw[line width = .5mm, blue] (f)--(S);
            \draw[line width = .5mm, blue] (e) to [out = 180, in = -60] (R);
            
            % Circles on boundary condition
            \draw[fill=white] (b) circle (.35);
            \draw[line width = .5mm, red, fill=white] (a) circle (.35);
            \draw[line width = .5mm, red, fill=white] (c) circle (.35);
            \draw[line width = .5mm, blue, fill=white] (f) circle (.35);
            \draw[fill=white] (d) circle (.35);
            \draw[line width = .5mm, blue, fill=white] (e) circle (.35);
            
            % Boundary condition nodes
            \node(a) at (1,-1) {$a$};
            \node(b) at (1,1) {$+$};
            \node(c) at (2,2) {$\Sigma$};
            \node(d) at (5,1) {$+$};
            \node(e) at (5,-1) {$d$};
            \node(f) at (2,-2) {\scriptsize$\Sigma_{da}^{+-}$};
            
            % Empty edges
            \draw[line width = .5mm, red, fill=white] (2,0) circle (.35);
            \draw[fill=white] (3,1) circle (.35);
            \draw[line width = .5mm, blue, fill=white] (3,-1) circle (.35);
            
            % Empty edges
            \node at (2,0) {$\Sigma$};
            \node at (3,1) {$+$};
            \node at (3,-1) {$d$};

            \end{tikzpicture}

            \,\

            \scriptsize
            $\underbrace{(1 + (\beta + \gamma)x_j)(1 + (\alpha + \gamma)x_j)}_{R_{ij}} \cdot \underbrace{(-1)^{|\Sigma| + 1} (((\alpha + \gamma)(\beta + \gamma)x_j + \gamma)\cdot h_{|\Sigma| - 1} + \alpha \beta \cdot h_{|\Sigma| - 2})}_{L_{jk}} \cdot \underbrace{(-\alpha)^{|\Sigma_{[d + 1, n]}|}(-\beta)^{|\Sigma_{[a + 1, n]}|}(1 + (\alpha + \gamma)x_i)}_{L_{ik}}$     

            \end{figure}

            \begin{figure}[H]
            \centering
            \begin{tikzpicture}
            % Boundary condition coordinates
            \coordinate(a) at (1,-1);
            \coordinate(b) at (1,1);
            \coordinate(c) at (2,2);
            \coordinate(d) at (5,1);
            \coordinate(e) at (5,-1);
            \coordinate(f) at (2,-2);
            
            % Empty edge nodes
            \node(1) at (2,0) {};
            \node(2) at (3,1) {};
            \node(3) at (3,-1) {};
            
            % Vertices
            \node(R) at (4,0) {$R_{ij}$};
            \node(T) at (2,1) {$L_{jk}$};
            \node(S) at (2,-1) {$L_{ik}$};
            
            % Draw edges
            \draw (b)--(T);
            \draw[line width = .5mm, red] (a)--(S);
            \draw[line width = .5mm, blue] (2)--(T);
            \draw (3) to [out = 0, in = -120] (R);
            \draw[line width = .5mm, red] (c)--(T);
            \draw[line width = .5mm, blue] (2) to [out = 0, in = 120] (R);
            \draw[transform canvas={xshift=1.1pt}, line width = .5mm, blue] (1)--(S);
            \draw[transform canvas={xshift=-1.1pt}, line width = .5mm, red] (1)--(S);
            \draw (d) to [out = 180, in=60] (R);
            \draw (3)--(S);
            \draw[transform canvas={xshift=1.1pt}, line width = .5mm, blue] (1)--(T);
            \draw[transform canvas={xshift=-1.1pt}, line width = .5mm, red] (1)--(T);
            \draw[line width = .5mm, blue] (f)--(S);
            \draw[line width = .5mm, blue] (e) to [out = 180, in = -60] (R);
            
            % Circles on boundary condition
            \draw[fill=white] (b) circle (.35);
            \draw[line width = .5mm, red, fill=white] (a) circle (.35);
            \draw[line width = .5mm, red, fill=white] (c) circle (.35);
            \draw[line width = .5mm, blue, fill=white] (f) circle (.35);
            \draw[fill=white] (d) circle (.35);
            \draw[line width = .5mm, blue, fill=white] (e) circle (.35);
            
            % Boundary condition nodes
            \node(a) at (1,-1) {$a$};
            \node(b) at (1,1) {$+$};
            \node(c) at (2,2) {$\Sigma$};
            \node(d) at (5,1) {$+$};
            \node(e) at (5,-1) {$d$};
            \node(f) at (2,-2) {\scriptsize$\Sigma_{da}^{+-}$};
            
            % Empty edges
            \draw[line width = .8mm, blue, fill=white] (2,0) circle (.38);
            \draw[line width = .5mm, red, fill=white] (2,0) circle (.35);
            \draw[line width = .5mm, blue, fill=white] (3,1) circle (.35);
            \draw[fill=white] (3,-1) circle (.35);
            
            % Empty edges
            \node at (2,0) {$\Sigma_d^+$};
            \node at (3,1) {$d$};
            \node at (3,-1) {$+$};

            \end{tikzpicture}
            
            \,\

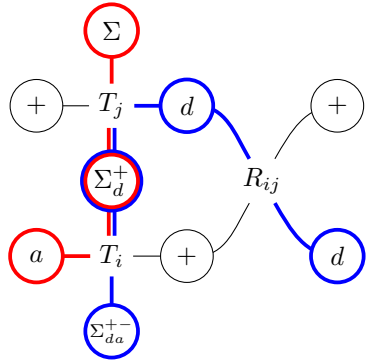
            \scriptsize
            $\underbrace{(x_j - x_i)}_{R_{ij}} \cdot \underbrace{(-\alpha)^{|\Sigma_{[d + 1, n]}|}(1 + (\alpha + \gamma)x_j)(1 + (\beta + \gamma)x_j)}_{L_{jk}} \cdot \underbrace{(-1)^{|\Sigma_d^+|}(-\beta)^{|(\Sigma_d^+)_{[a + 1, n]}|}(\alpha \beta \cdot h_{|\Sigma_d^+| - 3} + \gamma \cdot h_{|\Sigma_d^+|-2})}_{L_{ik}}$ 
    \end{figure}
    
    \noindent
    A number of substitutions can be made so that the partition functions of the two diagrams are comparable.
    Since $a < d$, it follows that $|(\Sigma_a^-)_{[d+1, n]}| = |\Sigma_{[d+1,n]}|$.
    The partition function of the left-hand diagram can thus be rewritten as
    \begin{align*}
        Z(\text{LHS}) &= \alpha \beta (x_j - x_i)  \cdot (-1)^{|\Sigma|} (-\beta)^{|\Sigma_{[a+1, n]}|}(\alpha \beta \cdot h_{|\Sigma| - 3} + \gamma \cdot h_{|\Sigma| - 2}) \\
        &\hspace{6cm} \cdot (-\alpha)^{|\Sigma_{[d+1, n]}|}(1 + (\alpha + \gamma)x_j) (1 + (\beta + \gamma)x_j) \\
        &+ (1 + (\beta + \gamma)x_j)(1 + (\alpha + \gamma)x_j) \cdot (-1)^{|\Sigma| + 1} (((\alpha + \gamma)(\beta + \gamma)x_i + \gamma) \cdot h_{|\Sigma| - 1} + \alpha \beta \cdot h_{|\Sigma| - 2}) \\
        & \hspace{6cm}\cdot (-\alpha)^{|\Sigma_{[d+1, n]}|}(-\beta)^{|\Sigma_{[a + 1, n]}|} (1 + (\alpha + \gamma)x_j).
    \end{align*}
    Note also that $|\Sigma_d^+| = |\Sigma| + 1$, and $|(\Sigma_d^+)_{[a + 1, n]}| = |\Sigma_{[a + 1, n]}| + 1$ since $a < d$, so the partition function of the right-hand diagram is
    \begin{align*}
        Z(\text{RHS}) &= (1 + (\beta + \gamma)x_j)(1 + (\alpha + \gamma)x_j) \cdot (-1)^{|\Sigma| + 1} (((\alpha + \gamma)(\beta + \gamma)x_j + \gamma)\cdot h_{|\Sigma| - 1} + \alpha \beta \cdot h_{|\Sigma| - 2}) \\
        & \hspace{6cm} \cdot (-\alpha)^{|\Sigma_{[d + 1, n]}|}(-\beta)^{|\Sigma_{[a + 1, n]}|}(1 + (\alpha + \gamma)x_i) \\
        &+ (x_j - x_i)\cdot (-\alpha)^{|\Sigma_{[d + 1, n]}|}(1 + (\alpha + \gamma)x_j)(1 + (\beta + \gamma)x_j) \\
        & \hspace{6cm} \cdot (-1)^{|\Sigma| + 1} (-\beta)^{|(\Sigma_{[a + 1, n]}|+1}(\alpha \beta \cdot h_{|\Sigma| - 2} + \gamma \cdot h_{|\Sigma|-1}).
    \end{align*}
    Using the recursion $h_m(\alpha,\beta) = \alpha \cdot h_{m-1}(\alpha, \beta) + \beta^m$ on the complete homogeneous symmetric functions, the substitution
    $$
    \alpha \beta(x_j - x_i)(\alpha \beta \cdot h_{|\Sigma|-3} + \gamma \cdot h_{|\Sigma|-2}) = \beta(x_j - x_i)(\alpha \beta \cdot h_{|\Sigma|-2} + \gamma \cdot h_{|\Sigma| - 1}) - \beta^{|\Sigma|}(x_j - x_i)(\alpha + \gamma)
    $$
    may be made in $Z(\text{LHS})$.
    Similarly, the substitution
    \begin{align*}
        (1 + (\alpha + \gamma)x_j)(((\alpha + \gamma)(\beta+\gamma)x_i + \gamma) &\cdot h_{|\Sigma|-1} + \alpha \beta \cdot h_{|\Sigma| - 2}) = (\alpha \beta \cdot h_{|\Sigma| - 2} + \gamma \cdot h_{|\Sigma|-1}) \\
        &+ (\alpha + \gamma)x_j(\alpha \beta \cdot h_{|\Sigma| - 2} + \gamma \cdot h_{|\Sigma|-1}) + (\alpha + \gamma)(\beta + \gamma)x_i h_{|\Sigma| - 1} \\
        &+ (\alpha + \gamma)^2(\beta+\gamma)x_ix_j \cdot h_{|\Sigma|-1}
    \end{align*}
    may be made in $Z(\text{LHS})$, and the analogous substitution with $i$ and $j$ exchanged may be made in $Z(\text{RHS})$.
    Upon making these substitutions and again using the identity $h_m(\alpha,\beta) = \alpha \cdot h_{m-1}(\alpha, \beta) + \beta^m$, further calculations confirm the equality $Z(\text{LHS}) = Z(\text{RHS})$.
    This computation illustrates how one may indeed handle cases of the RLL relation uniformly, based on the relative ordering of colors of horizontal edges, and hence reduce to a finite list of possibilities.
    In this manner, an exhaustive proof is given by the SageMath script in Appendix \ref{sage_code}.
\end{proof}

\subsection{Connections to known $R$-matrices}

It is natural to ask about relations between these $R$-matrix Boltzmann weights and previously studied colored vertex model weights arising from standard modules of quantum groups.
We consider two degenerations which recover $R$-matrices for standard modules of the quantum deformation of affine Lie superalgebras $\hat{\mathfrak{sl}}(n|1)$ and $\hat{\mathfrak{sl}}(1|n)$.
Such comparisons are naturally somewhat fussy to make, as we are comparing matrices entry by entry and so need to be careful in explaining our conventions.

First let $\alpha = \gamma = 0$ in the weights of Figure~\ref{R-weightsKirillov}, leaving only the parameter $\beta$.
In this degeneration, notice that the Boltzmann weights labeled $\tt{D}_2$ and $\tt{E}_1$ are equal, as are the Boltzmann weights labeled $\tt{D}_1$ and $\tt{E}_2$.
This results in the following Boltzmann weights:

\begin{figure}[H]
    \centering
    \scalebox{.8}{$\begin{array}{c@{\hspace{10pt}}|c@{\hspace{10pt}}|c@{\hspace{10pt}}|c@{\hspace{10pt}}}
        \toprule
        \tt{A}_1 & \tt{A}_2 & \tt{B} & \tt{C} \\
        \midrule
        \begin{tikzpicture}[scale=0.7]
        \draw[line width = .5mm, violet] (0,0) to [out = 0, in = 180] (2,2);
        \draw[line width = .5mm, violet] (0,2) to [out = 0, in = 180] (2,0);
        \draw[line width=0.5mm, violet, fill=white] (0,0) circle (.35);
        \draw[line width=0.5mm, violet, fill=white] (0,2) circle (.35);
        \draw[line width=0.5mm, violet, fill=white] (2,2) circle (.35);
        \draw[line width=0.5mm, violet, fill=white] (2,0) circle (.35);
        \node at (0,0) {$c$};
        \node at (0,2) {$c$};
        \node at (2,2) {$c$};
        \node at (2,0) {$c$};
        \end{tikzpicture} &
        \begin{tikzpicture}[scale=0.7]
        \draw[line width = .5mm, black] (0,0) to [out = 0, in = 180] (2,2);
        \draw[line width = .5mm, black] (0,2) to [out = 0, in = 180] (2,0);
        \draw[line width=0.5mm, black, fill=white] (0,0) circle (.35);
        \draw[line width=0.5mm, black, fill=white] (0,2) circle (.35);
        \draw[line width=0.5mm, black, fill=white] (2,2) circle (.35);
        \draw[line width=0.5mm, black, fill=white] (2,0) circle (.35);
        \node at (0,0) {$+$};
        \node at (0,2) {$+$};
        \node at (2,2) {$+$};
        \node at (2,0) {$+$};
        \end{tikzpicture}
        &
        \begin{tikzpicture}[scale=0.7]
        \draw[line width = .5mm, red] (0,0) to [out = 0, in = 180] (2,2);
        \draw[line width = .5mm, blue] (0,2) to [out = 0, in = 180] (2,0);
        \draw[line width=0.5mm, red, fill=white] (0,0) circle (.35);
        \draw[line width=0.5mm, blue, fill=white] (0,2) circle (.35);
        \draw[line width=0.5mm, red, fill=white] (2,2) circle (.35);
        \draw[line width=0.5mm, blue, fill=white] (2,0) circle (.35);
        \node at (0,0) {$a$};
        \node at (0,2) {$b$};
        \node at (2,2) {$a$};
        \node at (2,0) {$b$};
        \end{tikzpicture} \qquad
    \begin{tikzpicture}[scale=0.7]
        \draw[line width = .5mm, black] (0,0) to [out = 0, in = 180] (2,2);
        \draw[line width = .5mm, violet] (0,2) to [out = 0, in = 180] (2,0);
        \draw[line width=0.5mm, black, fill=white] (0,0) circle (.35);
        \draw[line width=0.5mm, violet, fill=white] (0,2) circle (.35);
        \draw[line width=0.5mm, black, fill=white] (2,2) circle (.35);
        \draw[line width=0.5mm, violet, fill=white] (2,0) circle (.35);
        \node at (0,0) {$+$};
        \node at (0,2) {$c$};
        \node at (2,2) {$+$};
        \node at (2,0) {$c$};
        \end{tikzpicture}
        &
        \begin{tikzpicture}[scale=0.7]
        \draw[line width = .5mm, blue] (0,0) to [out = 0, in = 180] (2,2);
        \draw[line width = .5mm, red] (0,2) to [out = 0, in = 180] (2,0);
        \draw[line width=0.5mm, blue, fill=white] (0,0) circle (.35);
        \draw[line width=0.5mm, red, fill=white] (0,2) circle (.35);
        \draw[line width=0.5mm, blue, fill=white] (2,2) circle (.35);
        \draw[line width=0.5mm, red, fill=white] (2,0) circle (.35);
        \node at (0,0) {$b$};
        \node at (0,2) {$a$};
        \node at (2,2) {$b$};
        \node at (2,0) {$a$};
        \end{tikzpicture} \qquad
    \begin{tikzpicture}[scale=0.7]
        \draw[line width = .5mm, violet] (0,0) to [out = 0, in = 180] (2,2);
        \draw[line width = .5mm, black] (0,2) to [out = 0, in = 180] (2,0);
        \draw[line width=0.5mm, violet, fill=white] (0,0) circle (.35);
        \draw[line width=0.5mm, black, fill=white] (0,2) circle (.35);
        \draw[line width=0.5mm, violet, fill=white] (2,2) circle (.35);
        \draw[line width=0.5mm, black, fill=white] (2,0) circle (.35);
        \node at (0,0) {$c$};
        \node at (0,2) {$+$};
        \node at (2,2) {$c$};
        \node at (2,0) {$+$};
        \end{tikzpicture}
        
        \\
        \midrule
        1+ \beta x_j  & 1+\beta x_i & x_j-x_i & 0 \\
        \bottomrule
        \end{array}$}

    \vspace{15pt}

    \scalebox{.8}{$\begin{array}{c@{\hspace{10pt}}|c@{\hspace{10pt}}}
        \toprule
        \tt{D}_1 = \tt{E}_2 & \tt{E}_1 = \tt{D}_2 \\
        \midrule
        \begin{tikzpicture}[scale=0.7]
        \draw[line width = .5mm,blue] (0,2) to [out = 0, in = 120] (1,1);
        \draw[line width = .5mm, blue] (2,2) to [out = 180, in=60] (1,1);
        \draw[line width = .5mm, red] (0,0) to [out = 0, in = -120] (1,1);
        \draw[line width = .5mm, red] (2,0) to [out = 180, in = -60] (1,1);
        \draw[line width=0.5mm, red, fill=white] (0,0) circle (.35);
        \draw[line width=0.5mm, blue, fill=white] (0,2) circle (.35);
        \draw[line width=0.5mm, blue, fill=white] (2,2) circle (.35);
        \draw[line width=0.5mm, red, fill=white] (2,0) circle (.35);
        \node at (0,0) {$a$};
        \node at (0,2) {$b$};
        \node at (2,2) {$b$};
        \node at (2,0) {$a$};
        \end{tikzpicture}
        \qquad
        \begin{tikzpicture}[scale=0.7]
        \draw[line width = .5mm,black] (0,2) to [out = 0, in = 120] (1,1);
        \draw[line width = .5mm, black] (2,2) to [out = 180, in=60] (1,1);
        \draw[line width = .5mm, violet] (0,0) to [out = 0, in = -120] (1,1);
        \draw[line width = .5mm, violet] (2,0) to [out = 180, in = -60] (1,1);
        \draw[line width=0.5mm, violet, fill=white] (0,0) circle (.35);
        \draw[line width=0.5mm, black, fill=white] (0,2) circle (.35);
        \draw[line width=0.5mm, black, fill=white] (2,2) circle (.35);
        \draw[line width=0.5mm, violet, fill=white] (2,0) circle (.35);
        \node at (0,0) {$c$};
        \node at (0,2) {$+$};
        \node at (2,2) {$+$};
        \node at (2,0) {$c$};
        \end{tikzpicture}
        &
        \begin{tikzpicture}[scale=0.7]
        \draw[line width = .5mm,red] (0,2) to [out = 0, in = 120] (1,1);
        \draw[line width = .5mm, red] (2,2) to [out = 180, in=60] (1,1);
        \draw[line width = .5mm, blue] (0,0) to [out = 0, in = -120] (1,1);
        \draw[line width = .5mm, blue] (2,0) to [out = 180, in = -60] (1,1);
        \draw[line width=0.5mm, blue, fill=white] (0,0) circle (.35);
        \draw[line width=0.5mm, red, fill=white] (0,2) circle (.35);
        \draw[line width=0.5mm, red, fill=white] (2,2) circle (.35);
        \draw[line width=0.5mm, blue, fill=white] (2,0) circle (.35);
        \node at (0,0) {$b$};
        \node at (0,2) {$a$};
        \node at (2,2) {$a$};
        \node at (2,0) {$b$};
        \end{tikzpicture}
        \qquad
        \begin{tikzpicture}[scale=0.7]
        \draw[line width = .5mm,violet] (0,2) to [out = 0, in = 120] (1,1);
        \draw[line width = .5mm, violet] (2,2) to [out = 180, in=60] (1,1);
        \draw[line width = .5mm, black] (0,0) to [out = 0, in = -120] (1,1);
        \draw[line width = .5mm, black] (2,0) to [out = 180, in = -60] (1,1);
        \draw[line width=0.5mm, black, fill=white] (0,0) circle (.35);
        \draw[line width=0.5mm, violet, fill=white] (0,2) circle (.35);
        \draw[line width=0.5mm, violet, fill=white] (2,2) circle (.35);
        \draw[line width=0.5mm, black, fill=white] (2,0) circle (.35);
        \node at (0,0) {$+$};
        \node at (0,2) {$c$};
        \node at (2,2) {$c$};
        \node at (2,0) {$+$};
        \end{tikzpicture}
        \\
        \midrule
        1+ \beta x_j & 1+ \beta x_i \\
        \bottomrule
        \end{array}$}
        \vspace{0.5\baselineskip}
        \caption{Degeneration for the solution to the Yang--Baxter equation when $\alpha = \gamma = 0$.
        Here again, $a < b$ and $c$ is any color.}
        \label{R-weights-degenerate}
\end{figure}
Recall that we can encode these weights from Figure~\ref{R-weights-degenerate} in a matrix $R(x_i, x_j) \coloneqq R_{ij}$ according to \eqref{eq:R-endomorphism}.
To render them explicitly in a matrix, we order basis vectors $v_i$ corresponding to the edge labels $i$, the colors $1, \ldots, n$ and ``uncolor'' $+$.
We will choose the ordering $\{ v_+ < v_n < v_{n-1} < \cdots <  v_1 \}$.
Then, for example, the rows of the matrix $R$ correspond to pairs of labels, with the first row corresponding to $(v_+,v_+)$, the second to $(v_+,v_n)$, and the last row corresponding to $(v_1,v_1)$.
We similarly order the columns from left to right.
In the simplest case where the number of colors is $n=1$, so that the only labels are just one color and the uncolor $+$, this ordering results in the matrix
\begin{equation*}
  R^{(0, \beta, 0)}(x_i, x_j) =
  \begin{pmatrix}
    1+\beta x_i & 0 & 0 & 0  \\
    0 & x_j-x_i & 1+\beta x_j & 0 \\
    0 & 1+\beta x_i & 0 & 0 \\
    0 & 0 & 0 & 1+\beta x_j
  \end{pmatrix}.
\end{equation*}
We take the convention that a matrix coefficient $R_{a, b}^{c, d}$ as in (\ref{R-vertex}) has subscripts referring to the column indexed by $v_a \otimes v_b$ and superscripts indicating the row indexed by $v_c \otimes v_d$.

\begin{theorem}
    For any positive integer $n$, the matrix $R^{(0, \beta, 0)}(x_i, x_j)$ formed using the Boltzmann weights in Figure~\ref{R-weights-degenerate} with $n$ colors matches a Drinfeld twist of the $R$-matrix for a pair of standard modules for the affine quantum superalgebra $U_q(\hat{\mathfrak{sl}}(1|n))$ in the limit $q \rightarrow 0$.
\label{thm:twistedsuperalgebra}
\end{theorem}

\begin{proof}[Proof.]
    We begin by recalling the Perk--Schultz solution \cite{PS1981} for the standard module of the affine quantum superalgebra $U_q(\hat{\mathfrak{sl}}(1|n))$, as given in \cite[Definition 2.1]{kojima2013}.
    This definition is for entries of the $R$-matrix denoted $\bar{R}(z)$ in \cite{kojima2013} satisfying the graded Yang--Baxter equation, with matrix multiplication defined to respect the grading in the associated tensor product of super vector spaces.
    We want the modified matrix $\bar{R}^{PS}(z)$ with entries 
    $$
    \bar{R}^{PS}(z)_{k_1, k_2}^{l_1, l_2} = (-1)^{[v_{k_1}] [v_{k_2}]} \bar{R}(z)_{k_1, k_2}^{l_1, l_2},
    $$
    where $[v_k] = 0$ if the corresponding basis vector is in the $n$-dimensional even graded piece and $[v_k] = 1$ if it corresponds to the $1$-dimensional odd graded piece for the super vector space of dimension $(1|n)$.
    In Kojima's notation, the initial element in the ordered basis receives an odd grading, and the subsequent $n$ basis elements in the ordering obtain an even grading.
    With an eye toward matching the $R$-matrix $R^{(0, \beta, 0)}(x_i, x_j)$ in the statement of the theorem, we use the index $+$ for the odd-graded basis element $v_+$ and the ordered basis $\{ v_n < \cdots < v_1 \}$ for the even-graded basis elements.
     
    Explicitly, this gives:
    \begin{align*}
        \bar{R}^{PS}(z)_{k,k}^{k,k} &=
        \begin{cases}
            1 & \textrm{if $k = +$,} \\
            \frac{z-q^2}{1-q^2 z} & \textrm{if $k \in [1,n]$,}
        \end{cases} \\
        \bar{R}^{PS}(z)_{k,l}^{k,l} &= \frac{(1-z)q}{1-q^2 z} \quad \text{for $k \neq l$,} \\
        \bar{R}^{PS}(z)_{k,l}^{l,k} &= \frac{(1-q^2)}{1-q^2 z} \cdot
        \begin{cases}
            1 & \textrm{if $v_k < v_l$,} \\
            z & \textrm{if $v_k > v_l$.}
        \end{cases}
    \end{align*}
    The remaining entries of the matrix are $0$.
    Consider the normalized $R$-matrix $(x_i - q^2 x_j) \bar{R}^{PS}(x_j / x_i)$ and make the change of variables $x \mapsto 1 + \beta x$ for both $x = x_i$ and $x = x_j$.
    This results in an $R$-matrix we call 
    $$
    R_\beta(x_i, x_j) \coloneqq (x_i - q^2 x_j) \bar{R}^{PS}(x_j / x_i) |_{x_i \mapsto 1 + \beta x_i, x_j \mapsto 1 + \beta x_j}
    $$ 
    whose non-zero entries are
    \begin{align*}
        {R}_\beta(x_i, x_j)_{k,k}^{k,k} &=
        \begin{cases}
            1 + \beta x_i - q^2 (1 + \beta x_j) & \textrm{if $k = +$,} \\
            1 + \beta x_j-q^2 (1 + \beta x_i) & \textrm{if $k \in [1,n]$,}
        \end{cases} \\
        {R}_\beta(x_i, x_j)_{k,l}^{k,l} &= \beta q (x_i - x_j) \quad \text{for $k \ne l$,} \\
        {R}_\beta(x_i, x_j)_{k,l}^{l,k} &= (1-q^2) \cdot
        \begin{cases}
            1+\beta x_i & \textrm{if $v_k < v_l$,} \\
            1+\beta x_j & \textrm{if $v_l<v_k$.}
        \end{cases}
    \end{align*}

Next we will obtain a new matrix from $R_\beta(x_i,x_j)$ by moving a factor of $(-\beta q)$ from the Boltzmann weights $R_\beta(x_i, x_j)_{+,k}^{+,k}$ to $R_\beta(x_i, x_j)_{k,+}^{k,+}$ for $k \in \{1, 2, \ldots,n\}$, and from $R_\beta(x_i, x_j)_{k,l}^{k,l}$ to $R_\beta(x_i, x_j)_{l,k}^{l,k}$ when $k,l \in \{1, 2, \ldots,n\}$ and $k < l$.
Moreover, we will ensure the resulting matrix also satisfies the Yang--Baxter equation.
We will achieve this by using a technique known as a Drinfeld twist, which more generally allows one to obtain from a quasi-triangular Hopf algebra another with the same algebra structure but which has a new coproduct and universal $R$-matrix, as discussed in  \cite{drinfeld_1989} and \cite[Theorem 1]{ReshetikhinMultiparameter}.
Working at the level of $R$-matrices, we will employ a Drinfeld twist of $R_\beta(x_i,x_j)$ by an appropriate matrix $F$ to obtain a new matrix
$$
R_\beta^F(x_i, x_j) \coloneqq F_{21} R_\beta(x_i, x_j) F^{-1}.
$$
Here $F_{21}$ is the result of swapping the order of the pairs of basis vectors in both rows and columns.
The entries of the twisted matrix will be
$$
{R}^F_\beta(x_i, x_j)_{k,l}^{k,l} =
\begin{cases}
    (x_j - x_i) & \textrm{if $k=+$ or ($k<l$ and $k,l \in [1,n]$)}, \\
    \beta^2 q^2 (x_j - x_i) & \textrm{if $l=+$ or  ($k > l$ and $k,l \in [1,n]$)},
\end{cases}
$$
for $k \ne l$ and will otherwise agree with those of $R_\beta(x_i,x_j)$.

We present the construction of $F$ for $U_q(\widehat{\mathfrak{sl}}(m|n))$ with $m$ and $n$ arbitrary, then specialize to the case $U_q(\widehat{\mathfrak{sl}}(1|n))$.
In analogy with \cite{ReshetikhinMultiparameter}, we define $F$ using Cartan elements of the underlying superalgebra.
To have a more flexible selection of Cartan elements, we extend trivially the standard module $V$ of $U_q(\widehat{\mathfrak{sl}}(m|n))$ to a representation of $U_q(\widehat{\mathfrak{gl}}(m|n))$.
Then $F$ is induced by an element of $U_q(\widehat{\mathfrak{gl}}(m|n)) \otimes U_q(\widehat{\mathfrak{gl}}(m|n))$ and is a matrix in $\text{End}(V_i \otimes V_j)$, where $V_i$ and $V_j$ are isomorphic copies of $V$ with associated spectral parameters $x_i$ and $x_j$, respectively.
Let
$$
F = \text{exp}\Big(\sum_{1 \leq k < l \leq m + n} \varphi_{k,l} (E_{k,k} \otimes E_{l,l} - E_{l,l} \otimes E_{k,k})\Big),
$$
where $\varphi_{k,l}$ is a complex number and $E_{k,k}$ is the matrix with $1$ in the $(k,k)$ coordinate and $0$ elsewhere.

Indexing the rows and columns by pairs of indices increasing lexicographically from left to right and top to bottom, we calculate
$$
\varphi_{k,l} (E_{k,k} \otimes E_{l,l} - E_{l,l} \otimes E_{k,k}) =
\begin{cases}
\varphi_{k,l} & \text{ at } \big((k,l),(k,l)\big), \\
-\varphi_{k,l} & \text{ at } \big((l,k),(l,k)\big), \\
0 & \text{ elsewhere.}
\end{cases}
$$
As diagonal matrices, the summands commute, and thus
$$
F = \prod_{1 \leq k < l \leq m + n} \text{exp}\Big( \varphi_{k,l} (E_{k,k} \otimes E_{l,l} - E_{l,l} \otimes E_{k,k})\Big),
$$
from which we conclude
$$
F =
\begin{cases}
\text{exp}(\varphi_{k,l}) & \text{ at } \big((k,l),(k,l)\big), \\
\text{exp}({-\varphi_{k,l}}) & \text{ at } \big((l,k),(l,k)\big), \\
1 & \text{ at } \big((k,k),(k,k)\big), \\
0 & \text{ elsewhere},
\end{cases}
$$
for $1 \leq k, l \leq m + n$ with $k<l$.

The complex parameters $\varphi_{k,l}$ may be freely chosen, so in particular we may take them to satisfy $\text{exp}(\varphi_{k,l}) = (-\beta q)^\frac{1}{2}$.
Therefore, we calculate
$$
R_\beta^F(x_i,x_j)_{k_1,k_2}^{l_1,l_2} =
\begin{cases}
R_\beta(x_i,x_j)_{k_1,k_2}^{l_1,l_2} \cdot (\beta q)^{-1} & \text{ for } k_1 = l_1, k_2 = l_2, k_1 < k_2, \\
R_\beta(x_i, x_j)_{k_1,k_2}^{l_1,l_2} \cdot \beta q & \text{ for } k_1 = l_1, k_2 = l_2, k_2 > k_1, \\
R_\beta(x_i, x_j)_{k_1,k_2}^{l_1,l_2}  & \text{ otherwise}.
\end{cases}
$$
In the case $m = 1$, the coordinates of $R_\beta^F(x_i,x_j)$ are
  \begin{align*}
    {R}_\beta^F(x_i, x_j)_{k,k}^{k,k} &=
    \begin{cases}
      1 + \beta x_i - q^2 (1 + \beta x_j) & \textrm{if $k = +$}, \\
      1 + \beta x_j-q^2 (1 + \beta x_i) & \textrm{if $k \in [1, n]$},
    \end{cases} \\
    R^F_\beta(x_i, x_j)_{k,l}^{k,l} &=
  \begin{cases}
    x_j - x_i & \textrm{if $k=+$, or $k,l \in [1,n]$ with $k<l$}, \\
    \beta^2 q^2 (x_j - x_i) & \textrm{if $l=+$, or $k,l \in [1, n]$ with $k > l$},
  \end{cases} \\
    R_\beta^F(x_i, x_j)_{k,l}^{l,k} &=
    \begin{cases}
      (1-q^2)(1+\beta x_i) & \textrm{if $v_k < v_l$}, \\
      (1-q^2)(1+\beta x_j) & \textrm{if $v_k > v_l$}, \\
    \end{cases}
  \end{align*}
where we assume $k \neq l$.

Finally, if we let $q \rightarrow 0$, we obtain the following Boltzmann weights for $R^{(0, \beta, 0)}(x_i, x_j) \coloneqq R_\beta^F(x_i, x_j) \mid_{q \rightarrow 0}$:
\begin{align*}
    {R}_\beta^F(x_i, x_j)_{k,k}^{k,k} \mid_{q \rightarrow 0} & =
    \begin{cases}
        1 + \beta x_i & \textrm{if $k = +$,} \\
        1 + \beta x_j & \textrm{if $k \in [1,n]$,}
    \end{cases} \\
    {R}_\beta^F(x_i, x_j)_{k,l}^{k,l} \mid_{q \rightarrow 0} & =
    \begin{cases}
        x_j - x_i & \textrm{if $l \neq k=+$ or ($k<l$ and $k,l \in [1,n]$)}, \\
        0 & \textrm{if $k \neq l=+$ or ($k > l$ and $k,l \in [1,n]$)},
    \end{cases} \\
    {R}_\beta^F(x_i, x_j)_{k,l}^{l,k} \mid_{q \rightarrow 0} & =
    \begin{cases} 1+\beta x_i & \textrm{if $v_k < v_l$,} \\
    1+\beta x_j & \textrm{if $v_l<v_k$.}
    \end{cases}
\end{align*}
\noindent
One may quickly check that the resulting Boltzmann weights exactly match those in Figure~\ref{R-weights-degenerate}.
\end{proof}

As noted in the Introduction, this degeneration of $(\alpha, \beta, \gamma) = (0, \beta, 0)$ results in the so-called $\beta$-Grothendieck polynomials arising in the connective $K$-theory of the flag variety for the general linear group.
A different lattice model for $\beta$-Grothendieck polynomials was presented in \cite{frozen-pipes}.
Even in the specializations $\alpha = \gamma = 0$ for the Boltzmann weights in Figure \ref{coloredalabw} and $q = 0$ for those in \cite[Figure 3]{frozen-pipes}, our lattice model fundamentally differs from the one presented in \cite{frozen-pipes}.
In addition to having different sets of Boltzmann weights, the lattice models have different connections to quantum group modules.
As shown in Section 3 of~\cite{frozen-pipes}, their Boltzmann weights are related to modules for quantum affine $\mathfrak{sl}(n+1)$, as opposed to our modules for quantum affine $\mathfrak{sl}(1|n)$.
It appears that only the superalgebra version of the Boltzmann weights can be generalized to the case of arbitrary $(\alpha, \beta, \gamma)$ presented here.

By a similar calculation for the degeneration $\beta = \gamma = 0$, we recover a second familiar $R$-matrix.

\begin{theorem}
\label{(thma,0,0)}
    For any positive integer $n$, the matrix $R^{(\alpha, 0, 0)}(x_i, x_j)$ formed using the degeneration $\beta = \gamma = 0$ of the Boltzmann weights in Figure~\ref{R-weightsKirillov} with $n$ colors matches a Drinfeld twist of the $R$-matrix for a pair of standard modules for the affine quantum superalgebra $U_q(\hat{\mathfrak{sl}}(n|1))$ in the limit $q \rightarrow 0$.
\end{theorem}

For the general case, one might hope to connect the Boltzmann weights in Figure~\ref{R-weightsKirillov} to those for more general multi-parameter quantum groups in the sense of \cite{ArtinSchelterTate,ReshetikhinMultiparameter, Sudbery, Takeuchi}, but at the present time we see no way of achieving the sort of Boltzmann weights in Figure~\ref{R-weightsKirillov} of type $\tt{A}_1$ or $\tt{A}_2$.
One may also wonder about the quantum group module origin of the $L$-matrices and their accompanying Boltzmann weights in Figure~\ref{coloredalabw}.
This appears still more challenging than matching the $R$-matrices, but we mention some brief speculation.
As a rough heuristic based on previous quantum group module interpretations of colored lattice models (see e.g., \cite{metahori, ABW2023}), one might expect the colors on edges of our lattice model to correspond to basis elements in the standard module of a multi-parameter quantum group generalization of our twisted superalgebra $U_q(\hat{\mathfrak{sl}}(1|n))$, consistent with Theorem~\ref{thm:twistedsuperalgebra} for the degenerate $(0, \beta, 0)$ case.
If this were the case, then the
labels for vertical edges appearing in Figure~\ref{coloredalabw} would correspond naturally with a basis for symmetric powers of this (generalized) standard module.
Recall that in the superalgebra setting, symmetric powers of modules manifest as exterior powers of the odd graded piece and symmetric powers of the even graded piece, and exterior powers of the $n$-dimensional odd graded space would then naturally correspond to the labels on vertical edges which do not permit repeated colors.

\section{Proof of main result}
\label{proofsection}

As above, let $\lambda = (\lambda_1, \lambda_2, \ldots, \lambda_n)$ be an integer partition and let $N = \lambda_1 + n - 1$.
Let $\mathcal{S}_w^{\lambda+\rho}$ continue to denote the system whose grid has $n$ rows and $N+1$ columns and whose boundary condition is determined by the permutations $w$ in $S_n$ and the partition $\lambda + \rho$, where $\rho = (n - 1, n-2, \ldots, 1, 0)$.
The proof of Theorem \ref{maintheorem} and its rephrasing Theorem \ref{rephrase} consists of showing that the partition functions $Z(\mathcal{S}_w^{\lambda+\rho})$ and the twisted Kirillov polynomials $\mathcal{KN}_w^{(\alpha, \beta, \gamma)}(\boldsymbol{x};\lambda)$ satisfy the same recursion under the operators $T_i^{(\alpha, \beta, \gamma)}$ and have a common seed for their recursions, and so must be equal.

\begin{proof}[Proof of Theorem \ref{maintheorem}, \ref{rephrase}.]
    The equality $Z(\mathcal{S}_\text{id}^{\lambda+\rho}) = \mathcal{KN}_\text{id}^{(\alpha, \beta, \gamma)}(\boldsymbol{x};\lambda)$ follows from (\ref{unique}) in Proposition \ref{basecase} by observing that $\prod_{i = 1}^k \prod_{j = 1}^{n_i} (-1)^j(\alpha \beta \cdot h_{j-3}(\alpha,\beta) + \gamma \cdot h_{j-2}(\alpha,\beta)) = 1$ when $\mu = \lambda + \rho$.
    To demonstrate $Z(\mathcal{S}_w^{\lambda+\rho})$ and $\mathcal{KN}_w^{(\alpha, \beta, \gamma)}(\boldsymbol{x}; \lambda)$ satisfy the same recursion, we use the solvability of the lattice model to apply the ``train argument'' or ``railroad argument'' so as to obtain an identity of partition functions in terms of Kirillov's operators $T_i^{(\alpha, \beta, \gamma)}$ defined in (\ref{kirillop}), which we relate to the matrix coefficients of the $R$-matrix given in Figure \ref{R-weightsKirillov}.

    The train argument is a repeated application of Theorem \ref{solvable} to rows $i$ and $i+1$ of the lattice model, and is by now a familiar approach that is well-documented in the literature; see for example Sections 6 and 7 of \cite{metaplectic-ice}, the proof of \cite[Lemma 8]{brubaker-bump-friedberg2011}, or the proof of \cite[Proposition 3.2]{bump2024colored}.
    The train argument implies that the partition functions of the lattice models below are equal:

% WHITE SPACING
\newpage

    
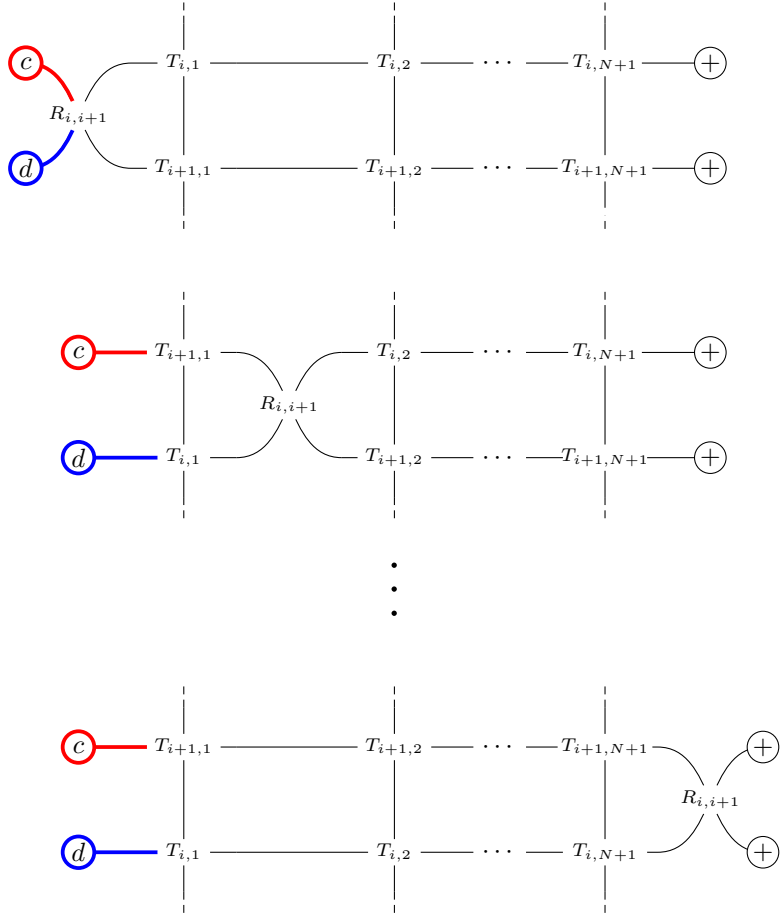
\begin{figure}[H]
    \centering
    \begin{tikzpicture}[scale=0.7]
    \begin{scope}[shift={(-1,0)}]
      \draw (0,1) to [out = 0, in = 180] (2,3) to (4,3);
      \draw (0,3) to [out = 0, in = 180] (2,1) to (4,1);

      \draw[line width=0.5mm,blue] (0,1) to [out = 0, in = 250] (1,2);
      \draw[line width=0.5mm,red] (0,3) to [out = 0, in = 110] (1,2);

      \draw (4,1) to (8.5,1);
      \draw (4,3) to (8.5,3);
      \draw (9.5,1) to (13,1);
      \draw (9.5,3) to (13,3);

      \draw[line width=0.5mm,blue,fill=white] (0,1) circle (.3);
      \draw[line width=0.5mm,red,fill=white] (0,3) circle (.3);
      \draw[fill=white] (13,3) circle (.3);
      \draw[fill=white] (13,1) circle (.3);

      \node at (0,1) {$b$};
      \node at (0,3) {$a$};
      \node at (9,3) {$\cdots$};
      \node at (9,1) {$\cdots$};

      \draw[densely dashed] (3,3.75) to (3,4.25);
      \draw[densely dashed] (3,0.25) to (3,-0.25);
      \draw[densely dashed] (7,3.75) to (7,4.25);
      \draw[densely dashed] (7,0.25) to (7,-0.25);
      \draw[densely dashed] (11,3.75) to (11,4.25);
      \draw[densely dashed] (11,0.25) to (11,-0.25);

      \node at (13,1) {$+$};
      \node at (13,3) {$+$};

      \path[fill=white] (3,3) circle (.5);
      \node at (3,3) {\scriptsize $L_{i,1}$};
      \path[fill=white] (3,1) circle (.7);
      \node at (3,1) {\scriptsize $L_{i+1,1}$};
      \path[fill=white] (7,3) circle (.5);
      \node at (7,3) {\scriptsize$L_{i,2}$};
      \path[fill=white] (7,1) circle (.7);
      \node at (7,1) {\scriptsize$L_{i+1,2}$};
      \path[fill=white] (11,3) circle (.7);
      \node at (11,3) {\scriptsize$L_{i,N+1}$};
      \path[fill=white] (11,1) circle (.9);
      \node at (11,1) {\scriptsize$L_{i+1,N+1}$};
      \path[fill=white] (1,2) circle (.3);
      \node at (1,2) {\scriptsize$R_{i,i+1}$};

      \draw (3,0.25) to (3,.75);
      \draw (3,1.25) to (3,2.75);
      \draw (3,3.25) to (3,3.75);

      \draw (7,0.25) to (7,.75);
      \draw (7,1.25) to (7,2.75);
      \draw (7,3.25) to (7,3.75);

      \draw (11,0.25) to (11,.75);
      \draw (11,1.25) to (11,2.75);
      \draw (11,3.25) to (11,3.75);
    \end{scope}

    \begin{scope}[shift={(1,-5.5)}]
      \draw (-1,1) to (2,1);
      \draw (-1,3) to (2,3);
      \draw (7.5,1) to (11,1);
      \draw (7.5,3) to (11,3);

      \draw[line width=0.5mm,blue] (-1,1) to (1,1);
      \draw[line width=0.5mm,red] (-1,3) to (1,3);
      \draw[line width=0.5mm,blue,fill=white] (-1,1) circle (.3);
      \draw[line width=0.5mm,red,fill=white] (-1,3) circle (.3);
      
      \node at (-1,1) {$b$};
      \node at (-1,3) {$a$};
      \draw[fill=white] (11,3) circle (.3);
      \draw[fill=white] (11,1) circle (.3);
      \node at (11,1) {$+$};
      \node at (11,3) {$+$};

      \node at (7,3) {$\cdots$};
      \node at (7,1) {$\cdots$};
      \draw (2,1) to [out = 0, in = 180] (4,3) to (4,3);
      \draw (2,3) to [out = 0, in = 180] (4,1) to (4,1);
      \path[fill=white] (3,2) circle (.3);
      \node at (3,2) {\scriptsize$R_{i,i+1}$};
      \draw (4,1) to (6.5,1);
      \draw (4,3) to (6.5,3);
      \path[fill=white] (1,3) circle (.7);
      \node at (1,3) {\scriptsize$L_{i+1,1}$};
      \path[fill=white] (1,1) circle (.5);
      \node at (1,1) {\scriptsize$L_{i,1}$};
      \path[fill=white] (5,3) circle (.5);
      \node at (5,3) {\scriptsize$L_{i,2}$};
      \path[fill=white] (5,1) circle (.7);
      \node at (5,1) {\scriptsize$L_{i+1,2}$};
      \path[fill=white] (9,3) circle (.7);
      \node at (9,3) {\scriptsize$L_{i,N+1}$};
      \path[fill=white] (9,1) circle (.8);
      \node at (9,1) {\scriptsize$L_{i+1,N+1}$};

      \draw[densely dashed] (1,3.75) to (1,4.25);
      \draw[densely dashed] (1,0.25) to (1,-0.25);
      \draw[densely dashed] (5,3.75) to (5,4.25);
      \draw[densely dashed] (5,0.25) to (5,-0.25);
      \draw[densely dashed] (9,0.25) to (9,-0.25);
      \draw[densely dashed] (9,3.75) to (9,4.25);

      \draw (1,0.25) to (1,.75);
      \draw (1,1.25) to (1,2.75);
      \draw (1,3.25) to (1,3.75);

      \draw (5,0.25) to (5,.75);
      \draw (5,1.25) to (5,2.75);
      \draw (5,3.25) to (5,3.75);

      \draw (9,0.25) to (9,.75);
      \draw (9,1.25) to (9,2.75);
      \draw (9,3.25) to (9,3.75);

    \end{scope}

    \begin{scope}[shift={(1,-7)}]
        \node at (5,0) {\Huge \rotatebox{90}{$\hdots$}};
    \end{scope}

    \begin{scope}[shift={(1,-13)}]
      \draw (7.5,1) to (10,1) to [out = 0, in = 180] (12,3);
      \draw (7.5,3) to (10,3) to [out = 0, in = 180] (12,1);
      \draw[fill=white] (12,3) circle (.3);
      \draw[fill=white] (12,1) circle (.3);
      \draw (-1,1) to (2,1);
      \draw (-1,3) to (2,3);
      \draw (2,1) to (6.5,1);
      \draw (2,3) to (6.5,3);
      \draw[line width=0.5mm,blue] (-1,1) to (1,1);
      \draw[line width=0.5mm,red] (-1,3) to (1,3);
      \draw[line width=0.5mm,blue,fill=white] (-1,1) circle (.3);
      \draw[line width=0.5mm,red,fill=white] (-1,3) circle (.3);

      \node at (7,1) {$\cdots$};
      \node at (7,3) {$\cdots$};

      \path[fill=white] (1,3) circle (.7);
      \node at (1,3) {\scriptsize$L_{i+1,1}$};
      \path[fill=white] (1,1) circle (.5);
      \node at (1,1) {\scriptsize$L_{i,1}$};
      \path[fill=white] (5,3) circle (.7);
      \node (a) at (5,3) {\scriptsize$L_{i+1,2}$};
      \path[fill=white] (5,1) circle (.5);
      \node (a) at (5,1) {\scriptsize$L_{i,2}$};
      \path[fill=white] (9,3) circle (.9);
      \node at (9,3) {\scriptsize$L_{i+1,N+1}$};
      \path[fill=white] (9,1) circle (.8);
      \node at (9,1) {\scriptsize$L_{i,N+1}$};

      \path[fill=white] (11,2) circle (.3);
      \node at (11,2) {\scriptsize$R_{i,i+1}$};
      \node at (12,1) {$+$};
      \node at (12,3) {$+$};
      \node at (-1,3) {$a$};
      \node at (-1,1) {$b$};

      \draw[densely dashed] (1,3.75) to (1,4.25);
      \draw[densely dashed] (1,0.25) to (1,-0.25);
      \draw[densely dashed] (5,3.75) to (5,4.25);
      \draw[densely dashed] (5,0.25) to (5,-0.25);
      \draw[densely dashed] (9,3.75) to (9,4.25);
      \draw[densely dashed] (9,0.25) to (9,-0.25);

      \draw (1,0.25) to (1,.75);
      \draw (1,1.25) to (1,2.75);
      \draw (1,3.25) to (1,3.75);

      \draw (5,0.25) to (5,.75);
      \draw (5,1.25) to (5,2.75);
      \draw (5,3.25) to (5,3.75);

      \draw (9,0.25) to (9,.75);
      \draw (9,1.25) to (9,2.75);
      \draw (9,3.25) to (9,3.75);
    \end{scope}
    \end{tikzpicture}
    \end{figure}

    \noindent
    In accordance with Figure \ref{R-weightsKirillov}, there are two ways to label the right-hand edges of the vertex on the left-hand side of the first lattice model and one way to label the left-hand edges of the vertex on the right-hand side of the last lattice model.
    The equality of partition functions above can then be rephrased as
    \begin{equation}
    \label{partitionvertex}
    Z(\mathcal{S}_{s_iw}^{\lambda+\rho}) = \frac{\text{wt} \left( \begin{tikzpicture}[scale=0.6, baseline=5mm]
            \draw[line width = .5mm, black] (0,0) to [out = 0, in = 180] (2,2);
            \draw[line width = .5mm, black] (0,2) to [out = 0, in = 180] (2,0);
            \draw[line width=0.5mm, black, fill=white] (0,0) circle (.35);
            \draw[line width=0.5mm, black, fill=white] (0,2) circle (.35);
            \draw[line width=0.5mm, black, fill=white] (2,2) circle (.35);
            \draw[line width=0.5mm, black, fill=white] (2,0) circle (.35);
            \node at (0,0) {$+$};
            \node at (0,2) {$+$};
            \node at (2,2) {$+$};
            \node at (2,0) {$+$};
            \end{tikzpicture} \right) \cdot Z(\mathcal{S}_w^{\lambda+\rho})^{s_i} - \text{wt} \left( \begin{tikzpicture}[scale=0.6, baseline=5mm]
            \draw[line width = .5mm,blue] (0,2) to [out = 0, in = 120] (1,1);
            \draw[line width = .5mm, blue] (2,2) to [out = 180, in=60] (1,1);
            \draw[line width = .5mm, red] (0,0) to [out = 0, in = -120] (1,1);
            \draw[line width = .5mm, red] (2,0) to [out = 180, in = -60] (1,1);
            \draw[line width=0.5mm, red, fill=white] (0,0) circle (.35);
            \draw[line width=0.5mm, blue, fill=white] (0,2) circle (.35);
            \draw[line width=0.5mm, blue, fill=white] (2,2) circle (.35);
            \draw[line width=0.5mm, red, fill=white] (2,0) circle (.35);
            \node at (0,0) {$a$};
            \node at (0,2) {$b$};
            \node at (2,2) {$b$};
            \node at (2,0) {$a$};
            \end{tikzpicture} \right) \cdot Z(\mathcal{S}_w^{\lambda+\rho}) }{\text{wt} \left( \begin{tikzpicture}[scale=0.6, baseline=5mm]
            \draw[line width = .5mm, red] (0,0) to [out = 0, in = 180] (2,2);
            \draw[line width = .5mm, blue] (0,2) to [out = 0, in = 180] (2,0);
            \draw[line width=0.5mm, red, fill=white] (0,0) circle (.35);
            \draw[line width=0.5mm, blue, fill=white] (0,2) circle (.35);
            \draw[line width=0.5mm, red, fill=white] (2,2) circle (.35);
            \draw[line width=0.5mm, blue, fill=white] (2,0) circle (.35);
            \node at (0,0) {$a$};
            \node at (0,2) {$b$};
            \node at (2,2) {$a$};
            \node at (2,0) {$b$};
            \end{tikzpicture} \right)},
    \end{equation}
    where $b = n+1 - w^{-1}(i)$ and $a = n + 1 -w^{-1}(i+1)$, and the notation $Z(\mathcal{S}_w^{\lambda+\rho})^{s_i} \coloneqq Z(\mathcal{S}_w^{\lambda+\rho})^{s_i}(\boldsymbol{x})$ used above refers to the usual permutation action on variables given in general for simple reflections $s_i$ by
    $$
    f^{s_i}(\boldsymbol{x}) = f^{s_i}(x_1, x_2, \ldots, x_n) \coloneqq f(x_1, \ldots, x_{i-1}, x_{i+1}, x_i, x_{i+2}, \cdots, x_n).
    $$
    If $a < b$, then according to Figure \ref{R-weightsKirillov} and a short computation, (\ref{partitionvertex}) above is precisely
    $$
    Z(\mathcal{S}^{\lambda+\rho}_{s_iw}) =  T_i^{(\alpha, \beta, \gamma)} Z(\mathcal{S}^{\lambda+\rho}_w).
    $$
    Moreover, $a < b$ implies $w^{-1}(i) < w^{-1}(i+1)$, which is equivalent to $\ell(s_iw) > \ell(w)$, in which case $\mathcal{KN}_{s_iw}^{(\alpha, \beta, \gamma)}(\boldsymbol{x}; \lambda) =  T_i^{(\alpha, \beta, \gamma)}\mathcal{KN}_w^{(\alpha, \beta, \gamma)}(\boldsymbol{x};\lambda)$.
    Thus, the partition functions $Z(\mathcal{S}_w^{\lambda+\rho})$ satisfy the same relation as the polynomials $\mathcal{KN}_w^{(\alpha, \beta, \gamma)}(\boldsymbol{x};\lambda)$ under the operators $T_i^{(\alpha, \beta, \gamma)}$ and we conclude $Z(\mathcal{S}_w^{\lambda+\rho}) = \mathcal{KN}_w^{(\alpha, \beta, \gamma)}(\boldsymbol{x};\lambda)$ for all $w \in S_n$ and all integer partitions $\lambda = (\lambda_1, \lambda_2, \ldots, \lambda_n)$.
\end{proof}

By a similar argument, the specializations $K_\zeta^{(\alpha, \beta, \gamma)}(\boldsymbol{x})$ of the generalized key polynomials may be recognized as partition functions of the lattice model, up to a scalar which depends upon $\zeta$.

\begin{theorem}
    Let $\zeta = (\zeta_1, \zeta_2, \ldots, \zeta_n)$ be a weak composition, let $\zeta^+ = (\zeta_1^+, \zeta_2^+, \ldots, \zeta_n^+)$ denote the unique partition obtained by permuting the parts of $\zeta$, and let $v_\zeta \in S_n$ be the permutation of minimal length such that $v_\zeta \cdot \zeta = \zeta^+$.
    Denote by $\mathcal{S}_{v_\zeta}^{\mu}$ the system with $n$ rows and $\zeta_1^+ + 1$ columns whose boundary condition is determined by $v_\zeta$ and the integer partition $\mu = (\mu_1, \mu_2, \ldots, \mu_n)$, where $\mu_i = \zeta_1^+ - \zeta_{n + 1 - i}^+$.
    Then
    $$
    K_\zeta^{(\alpha, \beta, \gamma)}(\boldsymbol{x}) = \frac{1}{C} \cdot Z(\mathcal{S}_{v_\zeta}^{\mu}),
    $$
    where $C = \prod_{i = 1}^k \prod_{j = 1}^{n_i} (-\alpha \beta \cdot h_{j-3}(\alpha,\beta) - \gamma \cdot h_{j-2}(\alpha,\beta))$ in the notation of Proposition \ref{basecase}.
\end{theorem}

\begin{proof}[Proof.]
    First consider the system $\mathcal{S}_\text{id}^{\mu}$ with $n$ rows and $\zeta_1^++1$ columns whose boundary condition is determined by the identity permutation $\text{id}$ in $S_n$ and the partition $\mu$.
    By Proposition \ref{basecase},
    $$
    Z(\mathcal{S}_\text{id}^{\mu}) = C \cdot \prod_{i = 1}^n x_i^{\zeta_i^+}
    = C \cdot K_{\zeta^+}^{(\alpha, \beta, \gamma)}(\boldsymbol{x}).
    $$
    The solvability of the lattice model implies $Z(\mathcal{S}_{v_\zeta}^{\mu}) = T_{v_\zeta}^{(\alpha, \beta, \gamma)} Z(\mathcal{S}_\text{id}^{\mu}) = T_{v_\zeta}^{(\alpha, \beta, \gamma)} C \cdot \boldsymbol{x}^{\zeta^+}$, and since scalar multiplication commutes with the application of the operators $T_i^{(\alpha, \beta, \gamma)}$, we conclude $K_\zeta^{(\alpha, \beta, \gamma)}(\boldsymbol{x}) = \frac{1}{C} \cdot Z(\mathcal{S}_{v_\zeta}^{\mu})$.
\end{proof}

When $\zeta$ satisfies an additional hypothesis, $K_\zeta^{(\alpha, \beta, \gamma)}$ is a twisted Kirillov polynomial.

\begin{corollary}
    If $\zeta$ is a weak composition such that $\lambda \coloneqq \mu - \rho$ is an integer partition, then $K_\zeta^{(\alpha, \beta, \gamma)}(\boldsymbol{x}) = \mathcal{KN}_{v_\zeta}^{(\alpha, \beta, \gamma)}(\boldsymbol{x}; \lambda)$.
\end{corollary}

\noindent
We note that $\lambda$ is an integer partition if and only if $\zeta_i^+ \geq \zeta_{i + 1}^+ + 1$ for all $1 \leq i \leq n - 1$.

\section{Applications and Kirillov's conjectures}
\label{applications}

As an application of our lattice model, we obtain several new results about the non-negativity of the coefficients of the polynomials $\mathcal{KN}_w^{(\alpha, \beta, \gamma)}(\boldsymbol{x}; \boldsymbol{0})$ for specializations of the parameters $\alpha$, $\beta$, and $\gamma$.
Kirillov's conjecture that the polynomials $\mathcal{KN}_w^{(\alpha, \beta, \gamma)}(\boldsymbol{x}; \boldsymbol{0})$ have non-negative coefficients is already known to hold for numerous special cases.
Results for various choices of $\alpha$, $\beta$, $\gamma$ and specializations of the variables $x_i$ to $1$ are discussed in \cite{kirillov2016notes} and the citations therein.
An even stronger partial result was proven by Chen and Zhang:

\begin{theorem}[{\cite[Theorem A]{chen-zhang-2022}}]
    Let $\alpha$, $\beta$, and $\gamma$ be parameters with $\gamma = 0$ and let $\eta = \alpha - \beta$.
    Then for every integer $n \geq 2$ and for every permutation $w \in S_n$,
    $$
    \mathcal{KN}_w^{(\alpha, \beta, \gamma = 0)}(\boldsymbol{x}; \boldsymbol{0}) \in \mathbb{N}[\alpha, \beta, \eta][x_1, x_2, \ldots, x_n].
    $$
\end{theorem}

\noindent
Although Kirillov indicates the conjecture may be false for $\gamma = 0$ (see Comment (b) following Conjecture 4.9 \cite{kirillov2016notes}), a formal counter-example has not previously appeared in the literature.
We thus present one here.

\begin{proposition}
\label{negativeprop}
    There exist choices of $n > 0$ and $w \in S_n$ for which the polynomials $\mathcal{KN}_w^{(\alpha, \beta, \gamma)}(\boldsymbol{x}; \boldsymbol{0})$ have negative coefficients.
\end{proposition}

\begin{proof}[Proof.]
    The smallest $n$ for which this occurs is $n = 4$, where exactly two $w \in S_4$ give polynomials $\mathcal{KN}_w^{(\alpha, \beta, \gamma)}(\boldsymbol{x};\boldsymbol{0})$ with negative coefficients.
    In cycle notation, they are $w = (1,3,2,4)$ and $w = (1,4,2,3)$.
    Their reduced words are $s_3s_2s_3s_1s_2$ and $s_2s_3s_1s_2s_3$, respectively, understood as acting on the left.
    The resulting polynomials require several pages to record them, and so are excluded from this text.
    Their negative terms in the monomial expansion are $- \beta^3 \gamma x_1^2 x_2 x_3 x_4 - \beta^2 \gamma x_1 x_2 x_3 x_4$ and $- (\alpha + \gamma)\beta^4 \gamma^2x_1^2 x_2^2 x_3^2 x_4^2 - \beta^4 \gamma^2 x_1^2 x_2^2 x_3^2 x_4$, respectively.
\end{proof}

\noindent
Similarly, the coefficients of $K_\zeta^{(\alpha, \beta, \gamma)}(\boldsymbol{x})$ are non-negative for many specializations of the parameters $\alpha$, $\beta$, and $\gamma$ \cite[Lemma 4.18]{kirillov2016notes}, but this does not hold in general, even with Kirillov's restriction $\zeta \subset \rho$.

\begin{proposition}
There exist compositions $\zeta$ for which the polynomials $K_\zeta^{(\alpha, \beta, \gamma)}(\boldsymbol{x})$ have negative coefficients.
\end{proposition}

\begin{proof}[Proof.]
    Consider $\zeta = (1, 2, 2, 1) \subset (4, 3, 2, 1)$.
\end{proof}

\noindent
As a corollary of these propositions, we resolve Kirillov's \cite[Conjecture 4.9]{kirillov2016notes}, as well as \cite[Conjecture 1.3]{kirillov2016notes}, which is an analogous statement involving the parameter $d$.
Although Kirillov's conjectures do not hold in general, we nevertheless identify the strongest case which does hold.
Our positive results for special cases of Kirillov's conjectures are essentially a consequence of recovering the polynomials $\mathcal{KN}_w^{(\alpha, \beta, \gamma)}(\boldsymbol{x}; \boldsymbol{0})$ as partition functions of our lattice model by the following special case of Theorem \ref{rephrase}.

\begin{corollary}
\label{LMthm}
Given a positive integer $n$, let $\mathcal{S}_w^{\rho}$ be a system whose grid has $n$ rows and $n$ columns, and whose boundary condition is determined by the integer partition $\rho = (n-1, n-2, \ldots, 1, 0)$ and the permutation $w$ in $S_n$.
Then $Z(\mathcal{S}_w^{\rho}) = \mathcal{KN}_w^{(\alpha, \beta, \gamma)}(\boldsymbol{x}; \boldsymbol{0})$.
\end{corollary}

\noindent
Having established an alternative method of computing the polynomials $\mathcal{KN}_w^{(\alpha, \beta, \gamma)}(\boldsymbol{x}; \boldsymbol{0})$, we are in a position to prove that the coefficients of the Hecke--Grothendieck polynomials $\mathcal{KN}_w^{(\alpha, \beta, \gamma = 0)}(\boldsymbol{x};\boldsymbol{0})$ are non-negative.
We first introduce the following lemma.

\begin{lemma}
\label{gamma_zero_lemma}
If $\gamma = 0$, then in any admissible state of any lattice model with boundary conditions as defined in Section \ref{latticemodelbasics}, the vertical edges of every vertex have labels which contain at most one color.
\end{lemma}

\begin{proof}[Proof.]
Suppose there exists an admissible state of the lattice model in which there is some vertex with a vertical edge whose label contains at least two colors.
Then a vertex of the form
    $$
    \arrowdiagramug{a}{\{a,b\}}{+}{\{b\}}
    $$
    necessarily occurs in the state also, since all paths must exit the grid on the left via separate rows.
    The Boltzmann weight of this vertex is equal to $\beta \gamma$ if $a < b$ and $-\gamma$ if $b < a$, so the Boltzmann weight of the state is divisible by $\gamma$ and thus is equal to $0$ in the specialization $\gamma = 0$.
\end{proof}

\begin{proof}[Proof of Theorem \ref{HGcoeffs}.]
    Let $n$ be a positive integer and let $\mathcal{S}_w^{\rho}$ denote the system whose grid has $n$ rows and $n$ columns, and whose boundary condition is determined by $\rho = (n-1, n-2, \ldots, 1, 0)$ and the permutation $w$ in $S_n$, so that $Z(\mathcal{S}_w^{\rho}) = \mathcal{KN}_w^{(\alpha, \beta, \gamma)}(\boldsymbol{x};\boldsymbol{0})$ by Corollary \ref{LMthm}.
    As shown in Figure \ref{coloredalabw}, if the Boltzmann weight of a vertex is necessarily expressed with a negative summand, then at least one of its adjacent vertical edges is labeled by a subset whose cardinality is greater than $1$.
    By Lemma \ref{gamma_zero_lemma}, such a state cannot occur when $\gamma = 0$.
    Thus, $\gamma = 0$ implies that every admissible state has a non-negative Boltzmann weight and therefore $\mathcal{KN}_w^{(\alpha, \beta, \gamma = 0)}(\boldsymbol{x};\boldsymbol{0}) \in \mathbb{N}[\alpha, \beta][x_1, x_2, \ldots, x_n]$ for all $w \in S_n$.
\end{proof}

Kirillov also presents the recursively defined \textit{Di Francesco--Zinn-Justin polynomials} $\mathcal{DZ}_w(\boldsymbol{x})$ \cite[Definition 4.4]{kirillov2016notes}, in reference to \cite{difrancesco-zinn-justin}, which are obtained by the further specialization $\mathcal{DZ}_w(\boldsymbol{x}) = \mathcal{KN}_{w^{-1}w_0}^{(\alpha, \beta, \gamma)}(\boldsymbol{x}; \boldsymbol{0})$ for $\alpha = \beta = 1$ and $\gamma = 0$ \cite[Remark 4.7 (b)]{kirillov2016notes}.
Kirillov's conjecture that the Di Francesco--Zinn-Justin polynomials have non-negative coefficients {\cite[Conjecture 4.5 (1)]{kirillov2016notes}} follows as an immediate special case of Theorem \ref{HGcoeffs}:

\begin{corollary}
\label{DZconjecture}
    The Di Francesco--Zinn-Justin polynomials $\mathcal{DZ}_w(\boldsymbol{x})$ have non-negative integer coefficients.
\end{corollary}

The above arguments use the fact that, although vertices with arbitrary subsets of colors assigned to their vertical edges may have non-zero Boltzmann weights even in the specialization $\gamma = 0$, such vertices are prohibited from appearing in admissible states of the lattice model unless $\gamma \neq 0$.
In addition to this method of argument, our lattice model may more obviously be used to prove monomial positivity when all Boltzmann weights specialize to non-negative values.
Using this kind of argument, we recover the following result due to Liu~\cite{liu2022} as another application of our lattice model.

\begin{theorem}[{\cite[Theorem 4.2]{liu2022}}]
  For each $w \in S_n$, the twisted Schubert polynomial
  $$
  \widetilde{\mathfrak{S}}_w = \mathcal{KN}_{w^{-1}w_0}^{(-1, -1, 1)}(\boldsymbol{x} ; \boldsymbol{0})
  $$
  has non-negative coefficients when expressed in the monomial basis.
\end{theorem}

\begin{proof}[Proof.]
The Boltzmann weights in the specialization $(\alpha, \beta, \gamma) = (-1, -1, 1)$ reduce to those in the figure below.

\begin{figure}[H]
\centering
\begin{equation*}
\def\arraystretch{1.8}
  \begin{array}{|c|c|c|}\hline
    \arrowdiagramgg{+}{\Sigma}{+}{\Sigma} & \arrowdiagram{c}{\Sigma}{c}{\Sigma} & \arrowdiagramgu{+}{\Sigma}{c}{\Sigma^+_c} \\
    \hline
    \rule{0pt}{2em}
    \rule[-1.5em]{0pt}{1em}
    1 & \begin{cases} 1 - x & \text{if } c \in \Sigma \\ x & \text{if } c \notin \Sigma\end{cases} & 1
    \\
    \hline
    \arrowdiagramug{c}{\Sigma}{+}{\Sigma^-_c}
    & \arrowdiagram{c}{\Sigma}{d}{\tensor*{\Sigma}{*^+_d^-_c}} & \arrowdiagram{d}{\Sigma}{c}{\tensor*{\Sigma}{*^+_c^-_d}} \\
    \hline
    1 & 1 &  1 \\
    \hline
  \end{array}
\end{equation*}
\caption{Boltzmann weights from Figure \ref{coloredalabw} in the specialization $(\alpha, \beta, \gamma) = (-1, -1, 1)$.
Note that $h_k(-1,-1) = (-1)^k(k+1)$.}
\end{figure}

\noindent
There is one Boltzmann weight with a negative coefficient, but the corresponding vertex is prohibited from occurring in an admissible state of the lattice model by our choice of boundary conditions.
Thus, all Boltzmann weights are effectively non-negative, and so the coefficients of the partition functions $\widetilde{\mathfrak{S}}_w$ are as well.
\end{proof}

As noted in the Introduction, there is a corresponding increase in complexity of the Boltzmann weights when vertical edges are decorated with multiple colors, as it allows for phenomena such as the occurrence of negative signs and the complete homogeneous symmetric functions in $\alpha$ and $\beta$ of arbitrary degree.
Thus positivity results based on our lattice models are especially interesting in the case $\gamma \ne 0$.
We present another such example now.

\begin{theorem}
    For all $n \geq 2$ and any permutation $w$ in $S_n$, the twisted Kirillov polynomials $\mathcal{KN}_w^{(-\alpha, -\beta, \alpha+\beta)}(\boldsymbol{x}; \boldsymbol{0})$ are in $\mathbb{N}[\alpha, \beta][x_1, x_2, \ldots, x_n]$.
\label{thm:coolerpositivity}
\end{theorem}

\begin{proof}[Proof.]
    By substitution, this is equivalent to the statement that the polynomials $\mathcal{KN}_w^{(\alpha, \beta, -(\alpha+\beta))}(\boldsymbol{x}; \boldsymbol{0})$ are in $\mathbb{N}[-\alpha, -\beta][x_1, x_2, \ldots, x_n]$.
    This follows from the fact that all our Boltzmann weights in the specialization $\gamma = -\alpha - \beta$ are positive in $(-\alpha)$ and $(-\beta)$ and hence the partition functions, representing the twisted Kirillov polynomials, are positive as well.

    The two most difficult sets of manipulations to see positivity in $(-\alpha)$ and $(-\beta)$ are handled as follows.
    Using the specialization $\gamma = - \alpha - \beta$ together with the identity $h_k = \alpha h_{k-1} + \beta^k$, we may rewrite the Boltzmann weights
\begin{align*}
    (\dagger) &= (-1)^{|\Sigma|}(\beta^{|\Sigma|} - (\beta + \gamma)\cdot h_{|\Sigma|-1}(\alpha,\beta)\cdot(x\oplus_{\alpha,\gamma}1)) \\
    (\ddagger) &= (-1)^{|\Sigma|-1}(-\beta)^{|\Sigma_{[c+1,n]}|}(\beta^{|\Sigma|-1} - (\beta + \gamma)\cdot h_{|\Sigma|-2}(\alpha,\beta)) \\
\end{align*}
as
\begin{align*}
    (\dagger) &= (-1)^{|\Sigma|}(\beta^{|\Sigma|} + \alpha \cdot h_{|\Sigma|-1}(\alpha,\beta)\cdot(1 - \beta x)) = (-1)^{|\Sigma|} (h_{|\Sigma|}(\alpha,\beta) - \alpha \beta h_{|\Sigma|-1}(\alpha,\beta) x) \\
    (\ddagger) &= (-1)^{|\Sigma|-1}(-\beta)^{|\Sigma_{[c+1,n]}|}(\beta^{|\Sigma|-1} + \alpha \cdot h_{|\Sigma|-2}(\alpha,\beta)) = (-1)^{|\Sigma|-1}(-\beta)^{|\Sigma_{[c+1,n]}|} h_{|\Sigma|-1}(\alpha,\beta).
\end{align*}
Applying the identity $h_k(\alpha, \beta) = (-1)^k h_k(-\alpha, -\beta)$ arising from the homogeneity, we see that both Boltzmann weights $(\dagger)$ and $(\ddagger)$ are positive in $(-\alpha)$ and $(-\beta)$:
\begin{align*}
    (\dagger) &= h_{|\Sigma|}(-\alpha,-\beta) + (- \alpha)(- \beta) h_{|\Sigma|-1}(-\alpha,-\beta) x  \\
    (\ddagger) &= (-\beta)^{|\Sigma_{[c+1,n]}|} h_{|\Sigma|-1}(-\alpha,-\beta).
\end{align*}

    The remaining Boltzmann weights are easily seen to be positive in $(-\alpha)$ and $(-\beta)$, remembering that $x\oplus_{\alpha,\gamma}1 = 1- \beta x$ and $x\oplus_{\beta,\gamma}1 = 1 - \alpha x$.
\end{proof}

The above positivity proof is, from the lattice model point of view, much more fascinating than Kirillov's positivity conjectures proved in Theorem~\ref{HGcoeffs} (the case $\gamma = 0$).
Indeed when the parameter $\gamma = \alpha + \beta$ is non-zero, we witness all of the additional complexity in the lattice states and their Boltzmann weights (given in Figure~\ref{coloredalabw}) described above.
It would be very interesting to provide a geometric connection to explain the positivity in either of Theorems~\ref{HGcoeffs} or~\ref{thm:coolerpositivity}, and to try to develop the most general positivity results from this three-parameter family of twisted Kirillov polynomials.

\newpage

\appendix

\section{Alternative Boltzmann weights}

As a final remark, we record an alternative solvable lattice model for computing the Hecke--Grothendieck polynomials.
It was constructed during the Polymath Jr Program in 2023 as a preliminary tool for proving Kirillov's Conjecture \ref{conjecture4.9} for the Hecke--Grothendieck polynomials, prior to defining the lattice model presented in the main text.
It has the advantage of providing a simpler proof of Theorem \ref{HGcoeffs}.
This lattice model is composed of an $n \times n$ grid, and in the same manner as described in Section \ref{boundary_conditions}, it is assigned a boundary condition determined by a permutation $w$ in $S_n$ and the integer partition $\rho = (n-1, n-2, \ldots, 1, 0)$.
Like the lattice model given in Section \ref{latticemodelbasics}, its admissible states consist of $n$ paths which travel downward and leftward from the top boundary edges to the left boundary edges.
By an argument nearly identical to the one given in Section \ref{proofsection}, its partition functions are precisely the Hecke--Grothendieck polynomials:

\begin{proposition}
Let $n$ be a positive integer and let $\boldsymbol{x} = (x_1, x_2, \ldots, x_n)$.
Denote by $S_w^\rho$ the system whose grid has $n$ rows and $n$ columns, whose boundary condition is determined by the permutation $w$ in $S_n$ and the integer partition $\rho = (n-1, n-2, \ldots, 1, 0)$, and whose vertices take Boltzmann weights as dictated by Figure \ref{polymath_weights}.
Then
$$
Z(S_w^\rho) = \mathcal{KN}^{(\alpha, \beta, \gamma = 0)}_w(\boldsymbol{x}; \boldsymbol{0}).
$$
\end{proposition}
\noindent
Since all Boltzmann weights in Figure \ref{polymath_weights} are non-negative, this lattice model yields a manifestly positive formula for the Hecke--Grothendieck polynomials, thereby providing a direct proof of Theorem \ref{HGcoeffs}.

\begin{figure}[H]
  \centering
  \scalebox{.8}{
    $
    \begin{array}{c@{\hspace{15pt}}c@{\hspace{15pt}}c@{\hspace{15pt}}c@{\hspace{15pt}}c@{\hspace{15pt}}}
    \toprule
    \tt{a}_1&\tt{b}_1&\tt{c}_1&\tt{d}_1&\tt{e}_1\\
    \midrule
    \begin{tikzpicture}
    \coordinate (a) at (-.75, 0);
    \coordinate (b) at (0, .75);
    \coordinate (c) at (.75, 0);
    \coordinate (d) at (0, -.75);
    \coordinate (aa) at (-.75,.5);
    \coordinate (cc) at (.75,.5);
    \draw[line width=0.5mm, violet] (a)--(0,0);
    \draw[line width=0.6mm, violet] (b)--(0,0);
    \draw[line width=0.5mm, violet] (c)--(0,0);
    \draw[line width=0.6mm, violet] (d)--(0,0);
    \draw[line width=0.5mm, violet,fill=white] (a) circle (.25);
    \draw[line width=0.5mm, violet,fill=white] (b) circle (.25);
    \draw[line width=0.5mm, violet, fill=white] (c) circle (.25);
    \draw[line width=0.5mm, violet, fill=white] (d) circle (.25);
    \node at (0,1) { };
    \node at (a) {$c$};
    \node at (b) {$c$};
    \node at (c) {$c$};
    \node at (d) {$c$};
    \end{tikzpicture}
    &
    \begin{tikzpicture}
    \coordinate (a) at (-.75, 0);
    \coordinate (b) at (0, .75);
    \coordinate (c) at (.75, 0);
    \coordinate (d) at (0, -.75);
    \coordinate (aa) at (-.75,.5);
    \coordinate (cc) at (.75,.5);
    \draw[line width=0.5mm, blue] (a)--(c);
    \draw[line width=0.6mm, red] (b)--(d);
    \draw[line width=0.5mm,blue,fill=white] (a) circle (.25);
    \draw[line width=0.5mm,blue,fill=white] (c) circle (.25);
    \draw[line width=0.5mm,red,fill=white] (b) circle (.25);
    \draw[line width=0.5mm,red,fill=white] (d) circle (.25);
    \node at (0,1) { };
    \node at (a) {$b$};
    \node at (b) {$a$};
    \node at (c) {$b$};
    \node at (d) {$a$};
    \end{tikzpicture}
    &
    \begin{tikzpicture}
    \coordinate (a) at (-.75, 0);
    \coordinate (b) at (0, .75);
    \coordinate (c) at (.75, 0);
    \coordinate (d) at (0, -.75);
    \coordinate (aa) at (-.75,.5);
    \coordinate (cc) at (.75,.5);
    \draw[line width=0.5mm, red] (a)--(c);
    \draw[line width=0.6mm, blue] (b)--(d);
    \draw[line width=0.5mm,red,fill=white] (a) circle (.25);
    \draw[line width=0.5mm,red,fill=white] (c) circle (.25);
    \draw[line width=0.5mm,blue,fill=white] (b) circle (.25);
    \draw[line width=0.5mm,blue,fill=white] (d) circle (.25);
    \node at (0,1) { };
    \node at (a) {$a$};
    \node at (b) {$b$};
    \node at (c) {$a$};
    \node at (d) {$b$};
    \end{tikzpicture}
    %%%%%%%
    & \begin{tikzpicture}
    \coordinate (a) at (-.75, 0);
    \coordinate (b) at (0, .75);
    \coordinate (c) at (.75, 0);
    \coordinate (d) at (0, -.75);
    \coordinate (aa) at (-.75,.5);
    \coordinate (cc) at (.75,.5);
    \draw[line width=0.5mm, blue](a)--(0,0)--(b);
    \draw[line width=0.5mm,blue,fill=white] (b) circle (.25);
    \draw[line width=0.5mm,blue,fill=white] (a) circle (.25);
    \draw[line width=0.5mm, red](d)--(0,0)--(c);
    \draw[line width=0.5mm,red,fill=white] (c) circle (.25);
    \draw[line width=0.5mm,red,fill=white] (d) circle (.25);
    \node at (0,1) { };
    \node at (a) {$b$};
    \node at (b) {$b$};
    \node at (c) {$a$};
    \node at (d) {$a$};
    \end{tikzpicture}
    %%%%%%%
    & \begin{tikzpicture}
    \coordinate (a) at (-.75, 0);
    \coordinate (b) at (0, .75);
    \coordinate (c) at (.75, 0);
    \coordinate (d) at (0, -.75);
    \coordinate (aa) at (-.75,.5);
    \coordinate (cc) at (.75,.5);
    \draw[line width=0.5mm, red](a)--(0,0)--(b);
    \draw[line width=0.5mm,red,fill=white] (b) circle (.25);
    \draw[line width=0.5mm,red,fill=white] (a) circle (.25);
    \draw[line width=0.5mm, blue](d)--(0,0)--(c);
    \draw[line width=0.5mm,blue,fill=white] (c) circle (.25);
    \draw[line width=0.5mm,blue,fill=white] (d) circle (.25);
    \node at (0,1) { };
    \node at (a) {$a$};
    \node at (b) {$a$};
    \node at (c) {$b$};
    \node at (d) {$b$};
    \end{tikzpicture}
    %%%%%%%%%
    \\
    \midrule
    1 & x_i & \alpha \beta x_i & 1 + \beta x_i & 1 + \alpha x_i  \\
    \bottomrule
    \\
    \toprule
    \tt{a}_2&\tt{b}_2&\tt{c}_2&\tt{d}_2&\tt{e}_2\\
    \midrule
    \begin{tikzpicture}
    \coordinate (a) at (-.75, 0);
    \coordinate (b) at (0, .75);
    \coordinate (c) at (.75, 0);
    \coordinate (d) at (0, -.75);
    \coordinate (aa) at (-.75,.5);
    \coordinate (cc) at (.75,.5);
    \draw[line width=0.5mm] (a)--(0,0);
    \draw[line width=0.6mm] (b)--(0,0);
    \draw[line width=0.5mm] (c)--(0,0);
    \draw[line width=0.6mm] (d)--(0,0);
    \draw[line width=0.5mm,fill=white] (a) circle (.25);
    \draw[line width=0.5mm,fill=white] (b) circle (.25);
    \draw[line width=0.5mm, fill=white] (c) circle (.25);
    \draw[line width=0.5mm, fill=white] (d) circle (.25);
    \node at (0,1) { };
    \node at (a) {$+$};
    \node at (b) {$+$};
    \node at (c) {$+$};
    \node at (d) {$+$};
    \end{tikzpicture}
    &
    \begin{tikzpicture}
    \coordinate (a) at (-.75, 0);
    \coordinate (b) at (0, .75);
    \coordinate (c) at (.75, 0);
    \coordinate (d) at (0, -.75);
    \coordinate (aa) at (-.75,.5);
    \coordinate (cc) at (.75,.5);
    \draw[line width=0.5mm, violet] (a)--(c);
    \draw[line width=0.6mm] (b)--(d);
    \draw[line width=0.5mm,violet,fill=white] (a) circle (.25);
    \draw[line width=0.5mm,violet,fill=white] (c) circle (.25);
    \draw[line width=0.5mm,fill=white] (b) circle (.25);
    \draw[line width=0.5mm,fill=white] (d) circle (.25);
    \node at (0,1) { };
    \node at (a) {$c$};
    \node at (b) {$+$};
    \node at (c) {$c$};
    \node at (d) {$+$};
    \end{tikzpicture}
    &
    \begin{tikzpicture}
    \coordinate (a) at (-.75, 0);
    \coordinate (b) at (0, .75);
    \coordinate (c) at (.75, 0);
    \coordinate (d) at (0, -.75);
    \coordinate (aa) at (-.75,.5);
    \coordinate (cc) at (.75,.5);
    \draw[line width=0.5mm] (a)--(c);
    \draw[line width=0.6mm, violet] (b)--(d);
    \draw[line width=0.5mm,fill=white] (a) circle (.25);
    \draw[line width=0.5mm,fill=white] (c) circle (.25);
    \draw[line width=0.5mm,violet,fill=white] (b) circle (.25);
    \draw[line width=0.5mm,violet,fill=white] (d) circle (.25);
    \node at (0,1) { };
    \node at (a) {$+$};
    \node at (b) {$c$};
    \node at (c) {$+$};
    \node at (d) {$c$};
    \end{tikzpicture}
    %%%%%%%
    & \begin{tikzpicture}
    \coordinate (a) at (-.75, 0);
    \coordinate (b) at (0, .75);
    \coordinate (c) at (.75, 0);
    \coordinate (d) at (0, -.75);
    \coordinate (aa) at (-.75,.5);
    \coordinate (cc) at (.75,.5);
    \draw[line width=0.5mm, violet](a)--(0,0)--(b);
    \draw[line width=0.5mm,violet,fill=white] (b) circle (.25);
    \draw[line width=0.5mm,violet,fill=white] (a) circle (.25);
    \draw[line width=0.5mm](d)--(0,0)--(c);
    \draw[line width=0.5mm,fill=white] (c) circle (.25);
    \draw[line width=0.5mm,fill=white] (d) circle (.25);
    \node at (0,1) { };
    \node at (a) {$c$};
    \node at (b) {$c$};
    \node at (c) {$+$};
    \node at (d) {$+$};
    \end{tikzpicture}
    %%%%%%%
    & \begin{tikzpicture}
    \coordinate (a) at (-.75, 0);
    \coordinate (b) at (0, .75);
    \coordinate (c) at (.75, 0);
    \coordinate (d) at (0, -.75);
    \coordinate (aa) at (-.75,.5);
    \coordinate (cc) at (.75,.5);
    \draw[line width=0.5mm](a)--(0,0)--(b);
    \draw[line width=0.5mm,fill=white] (b) circle (.25);
    \draw[line width=0.5mm,fill=white] (a) circle (.25);
    \draw[line width=0.5mm, violet](d)--(0,0)--(c);
    \draw[line width=0.5mm,violet,fill=white] (c) circle (.25);
    \draw[line width=0.5mm,violet,fill=white] (d) circle (.25);
    \node at (0,1) { };
    \node at (a) {$+$};
    \node at (b) {$+$};
    \node at (c) {$c$};
    \node at (d) {$c$};
    \end{tikzpicture}
    %%%%%%%%%
    \\
    \midrule
    1 & x_i & \alpha \beta x_i & 1 & (1 + \alpha x_i)(1 + \beta x_i)  \\
    \bottomrule
    \end{array}
    $}
\caption{Boltzmann weights which yield partition functions that are the Hecke--Grothendieck polynomials.
Here $a < b$ and $c$ is arbitrary.
Note that these Boltzmann weights are not obtained from Figure \ref{coloredalabw} by setting $\gamma = 0$.}
\label{polymath_weights}
\end{figure}

\noindent
The $R$-vertices which give a solution to the Yang--Baxter equation for the lattice model are pictured in Figure \ref{polymath_r_vertices} below.
Their Boltzmann weights were calculated using a modification of code authored by H. Gustafsson.

\newpage

\begin{figure}[H]
    \centering
    \scalebox{.8}{$\begin{array}{c@{\hspace{10pt}}}
        \toprule
        \tt{A} \\
        \midrule
        \begin{tikzpicture}[scale=0.7]
        \draw[line width = .5mm, violet] (0,0) to [out = 0, in = 180] (2,2);
        \draw[line width = .5mm, violet] (0,2) to [out = 0, in = 180] (2,0);
        \draw[line width=0.5mm, violet, fill=white] (0,0) circle (.35);
        \draw[line width=0.5mm, violet, fill=white] (0,2) circle (.35);
        \draw[line width=0.5mm, violet, fill=white] (2,2) circle (.35);
        \draw[line width=0.5mm, violet, fill=white] (2,0) circle (.35);
        \node at (0,0) {$c$};
        \node at (0,2) {$c$};
        \node at (2,2) {$c$};
        \node at (2,0) {$c$};
        \end{tikzpicture}
        \qquad
    \begin{tikzpicture}[scale=0.7]
        \draw[line width = .5mm, black] (0,0) to [out = 0, in = 180] (2,2);
        \draw[line width = .5mm, black] (0,2) to [out = 0, in = 180] (2,0);
        \draw[line width=0.5mm, black, fill=white] (0,0) circle (.35);
        \draw[line width=0.5mm, black, fill=white] (0,2) circle (.35);
        \draw[line width=0.5mm, black, fill=white] (2,2) circle (.35);
        \draw[line width=0.5mm, black, fill=white] (2,0) circle (.35);
        \node at (0,0) {$+$};
        \node at (0,2) {$+$};
        \node at (2,2) {$+$};
        \node at (2,0) {$+$};
        \end{tikzpicture}
        \\
        \midrule
        \alpha \beta x_i x_j + (\alpha + \beta) x_i + 1 \\
        \bottomrule
        \end{array}$}
        \vspace{15pt}

    \scalebox{.8}{$\begin{array}{c@{\hspace{10pt}}|c@{\hspace{10pt}}}
        \toprule
        \tt{B} & \tt{C} \\
        \midrule
        \begin{tikzpicture}[scale=0.7]
        \draw[line width = .5mm, red] (0,0) to [out = 0, in = 180] (2,2);
        \draw[line width = .5mm, blue] (0,2) to [out = 0, in = 180] (2,0);
        \draw[line width=0.5mm, red, fill=white] (0,0) circle (.35);
        \draw[line width=0.5mm, blue, fill=white] (0,2) circle (.35);
        \draw[line width=0.5mm, red, fill=white] (2,2) circle (.35);
        \draw[line width=0.5mm, blue, fill=white] (2,0) circle (.35);
        \node at (0,0) {$a$};
        \node at (0,2) {$b$};
        \node at (2,2) {$a$};
        \node at (2,0) {$b$};
        \end{tikzpicture} \qquad
        \begin{tikzpicture}[scale=0.7]
        \draw[line width = .5mm, black] (0,0) to [out = 0, in = 180] (2,2);
        \draw[line width = .5mm, violet] (0,2) to [out = 0, in = 180] (2,0);
        \draw[line width=0.5mm, black, fill=white] (0,0) circle (.35);
        \draw[line width=0.5mm, violet, fill=white] (0,2) circle (.35);
        \draw[line width=0.5mm, black, fill=white] (2,2) circle (.35);
        \draw[line width=0.5mm, violet, fill=white] (2,0) circle (.35);
        \node at (0,0) {$+$};
        \node at (0,2) {$c$};
        \node at (2,2) {$+$};
        \node at (2,0) {$c$};
        \end{tikzpicture}
        &
        \begin{tikzpicture}[scale=0.7]
        \draw[line width = .5mm, blue] (0,0) to [out = 0, in = 180] (2,2);
        \draw[line width = .5mm, red] (0,2) to [out = 0, in = 180] (2,0);
        \draw[line width=0.5mm, blue, fill=white] (0,0) circle (.35);
        \draw[line width=0.5mm, red, fill=white] (0,2) circle (.35);
        \draw[line width=0.5mm, blue, fill=white] (2,2) circle (.35);
        \draw[line width=0.5mm, red, fill=white] (2,0) circle (.35);
        \node at (0,0) {$b$};
        \node at (0,2) {$a$};
        \node at (2,2) {$b$};
        \node at (2,0) {$a$};
        \end{tikzpicture} \qquad
        \begin{tikzpicture}[scale=0.7]
        \draw[line width = .5mm, violet] (0,0) to [out = 0, in = 180] (2,2);
        \draw[line width = .5mm, black] (0,2) to [out = 0, in = 180] (2,0);
        \draw[line width=0.5mm, violet, fill=white] (0,0) circle (.35);
        \draw[line width=0.5mm, black, fill=white] (0,2) circle (.35);
        \draw[line width=0.5mm, violet, fill=white] (2,2) circle (.35);
        \draw[line width=0.5mm, black, fill=white] (2,0) circle (.35);
        \node at (0,0) {$c$};
        \node at (0,2) {$+$};
        \node at (2,2) {$c$};
        \node at (2,0) {$+$};
        \end{tikzpicture}
        \\
        \midrule
        x_j-x_i & \alpha \beta (x_j -x_i) \\
        \bottomrule
        \end{array}$}
        
        \vspace{15pt}

    \scalebox{.8}{$\begin{array}{c@{\hspace{10pt}}|c@{\hspace{10pt}}|c@{\hspace{10pt}}|c@{\hspace{10pt}}}
        \toprule
        \tt{D_1} & \tt{E_1} & \tt{D_2} & \tt{E_2} \\
        \midrule
        \begin{tikzpicture}[scale=0.7]
        \draw[line width = .5mm,blue] (0,2) to [out = 0, in = 120] (1,1);
        \draw[line width = .5mm, blue] (2,2) to [out = 180, in=60] (1,1);
        \draw[line width = .5mm, red] (0,0) to [out = 0, in = -120] (1,1);
        \draw[line width = .5mm, red] (2,0) to [out = 180, in = -60] (1,1);
        \draw[line width=0.5mm, red, fill=white] (0,0) circle (.35);
        \draw[line width=0.5mm, blue, fill=white] (0,2) circle (.35);
        \draw[line width=0.5mm, blue, fill=white] (2,2) circle (.35);
        \draw[line width=0.5mm, red, fill=white] (2,0) circle (.35);
        \node at (0,0) {$a$};
        \node at (0,2) {$b$};
        \node at (2,2) {$b$};
        \node at (2,0) {$a$};
        \end{tikzpicture}
        &
        \begin{tikzpicture}[scale=0.7]
        \draw[line width = .5mm,red] (0,2) to [out = 0, in = 120] (1,1);
        \draw[line width = .5mm, red] (2,2) to [out = 180, in=60] (1,1);
        \draw[line width = .5mm, blue] (0,0) to [out = 0, in = -120] (1,1);
        \draw[line width = .5mm, blue] (2,0) to [out = 180, in = -60] (1,1);
        \draw[line width=0.5mm, blue, fill=white] (0,0) circle (.35);
        \draw[line width=0.5mm, red, fill=white] (0,2) circle (.35);
        \draw[line width=0.5mm, red, fill=white] (2,2) circle (.35);
        \draw[line width=0.5mm, blue, fill=white] (2,0) circle (.35);
        \node at (0,0) {$b$};
        \node at (0,2) {$a$};
        \node at (2,2) {$a$};
        \node at (2,0) {$b$};
        \end{tikzpicture}
        &
        \begin{tikzpicture}[scale=0.7]
        \draw[line width = .5mm,violet] (0,2) to [out = 0, in = 120] (1,1);
        \draw[line width = .5mm, violet] (2,2) to [out = 180, in=60] (1,1);
        \draw[line width = .5mm, black] (0,0) to [out = 0, in = -120] (1,1);
        \draw[line width = .5mm, black] (2,0) to [out = 180, in = -60] (1,1);
        \draw[line width=0.5mm, black, fill=white] (0,0) circle (.35);
        \draw[line width=0.5mm, violet, fill=white] (0,2) circle (.35);
        \draw[line width=0.5mm, violet, fill=white] (2,2) circle (.35);
        \draw[line width=0.5mm, black, fill=white] (2,0) circle (.35);
        \node at (0,0) {$+$};
        \node at (0,2) {$c$};
        \node at (2,2) {$c$};
        \node at (2,0) {$+$};
        \end{tikzpicture}
        &
        \begin{tikzpicture}[scale=0.7]
        \draw[line width = .5mm,black] (0,2) to [out = 0, in = 120] (1,1);
        \draw[line width = .5mm, black] (2,2) to [out = 180, in=60] (1,1);
        \draw[line width = .5mm, violet] (0,0) to [out = 0, in = -120] (1,1);
        \draw[line width = .5mm, violet] (2,0) to [out = 180, in = -60] (1,1);
        \draw[line width=0.5mm, violet, fill=white] (0,0) circle (.35);
        \draw[line width=0.5mm, black, fill=white] (0,2) circle (.35);
        \draw[line width=0.5mm, black, fill=white] (2,2) circle (.35);
        \draw[line width=0.5mm, violet, fill=white] (2,0) circle (.35);
        \node at (0,0) {$c$};
        \node at (0,2) {$+$};
        \node at (2,2) {$+$};
        \node at (2,0) {$c$};
        \end{tikzpicture}
        \\
        \midrule
        (1+ \beta x_j)(1+\alpha x_i) & (1 + \beta x_i)(1 + \alpha x_j) & (1+\beta x_i)(1+\alpha x_i) & (1 + \beta x_j)(1 + \alpha x_j) \\
        \bottomrule
        \end{array}$}
        \vspace{.5\baselineskip}
        \caption{Solution to the Yang--Baxter equation given by the Boltzmann weights in Figure \ref{polymath_weights}.
        Note that these Boltzmann weights are not obtained from those in Figure \ref{R-weightsKirillov} by setting $\gamma = 0$.}
        \label{polymath_r_vertices}
\end{figure}

This lattice model fundamentally differs from the one presented in Section \ref{latticemodelbasics}.
Most notably, multiple colors are prohibited from occupying the same vertical edge of a vertex.
Note that the Boltzmann weights in Figure \ref{polymath_weights} are not specializations of those in Figure \ref{coloredalabw} upon setting $\gamma = 0$, even when the number of colors on a vertical edge is restricted to $0$ or $1$:
a vertex of type $\tt{a}_1$ has Boltzmann weight $1$, whereas the corresponding vertex in Figure \ref{coloredalabw} has Boltzmann weight $1 + (\alpha + \beta)x_i$ when $\gamma = 0$.
Similarly, the Boltzmann weights in Figure \ref{polymath_r_vertices} are not specializations of those in Figure \ref{R-weightsKirillov}, as the Boltzmann weight of type $\tt{A}$ is not a simultaneous specialization of types $\tt{A}_1$ and $\tt{A}_2$.
Algebraically, this is a consequence of a differing connection to quantum group modules, as in the comment preceding Theorem \ref{(thma,0,0)}.

We propose that the $R$-matrix and $T$-matrix appearing in \cite{frozen-pipes} are Drinfeld twists of our $R$-matrix and $L$-matrix, respectively, after a specialization of parameters and a change of basis.
We prove this for the case $n = 1$.
The parameters of each lattice model are first specialized so that their partition functions are the $\beta$-Grothendieck polynomials: set
$\alpha = 0$ for the Boltzmann weights in Figures \ref{polymath_weights} and \ref{polymath_r_vertices}, and set $q = 0$ and $y_i = 0$ for all $i$  for the Boltzmann weights in \cite[Figure 3, Figure 5]{frozen-pipes}.
Then make the change of basis in \cite{frozen-pipes} so that the label $+$ is regarded as less than than any color, to be consistent with the conventions stated in the proof of Theorem \ref{thm:twistedsuperalgebra}.
Let
$$
F =
\begin{bmatrix}
1 & & & \\
& 1 & \beta & \\
& 1 & 1 & \\
& & & 1
\end{bmatrix}.
$$
From Figure \ref{polymath_weights} in the specialization $\alpha = 0$, for $n = 1$ we have
$$
L = 
\begin{bmatrix}
1 & & & \\
& x_i & 1 + \beta x_i & \\
& 1 & 0 &  \\
 &  &  & 1\\
\end{bmatrix}.
$$
From \cite[Figure 3]{frozen-pipes} in the specialization $q = 0$, $y_i = 0$, and the reordered basis, we have
$$
T = 
\begin{bmatrix}
1 & & & \\
& x_i & 1 & \\
& 1 + \beta x_i & 0 &  \\
 &  &  & 1\\
\end{bmatrix}.
$$
Then
\begin{align*}
F_{21} L F^{-1} &=
\begin{bmatrix}
1 & & & \\
& 1 & 1 & \\
& \beta & 1 & \\
& & & 1
\end{bmatrix}
\begin{bmatrix}
1 & & & \\
& x_i & 1 + \beta x_i & \\
& 1 & 0 &  \\
 &  &  & 1\\
\end{bmatrix}
\begin{bmatrix}
1 & & & \\
& \frac{-1}{\beta-1} & \frac{\beta}{\beta-1} & \\
& \frac{1}{\beta - 1} & -\frac{1}{\beta - 1} & \\
& & & 1
\end{bmatrix}
= T.
\end{align*}

\noindent
A similar computation shows that for $n = 1$, the $R$ in \cite{frozen-pipes} with the same specialization and change of basis is a Drinfeld twist of our $R$ by the same $F$.

\newpage

\section{SageMath script}
\label{sage_code}

\begin{lstlisting}[language=Python]
import itertools

#--------------------------------------Variables------------------------------------------

var("a", "b", "c", "d", "north_card", "A", "B", "G", "xi", "xj")

# The variables a, b, c, d are placeholders for the labels on the horizontal boundary
# edges, ordered southwest, northwest, northeast, and southeast, respectively.
#
# north_card is the cardinality of the northern boundary edge label of the lattice model
#
# A, B, G are alpha, beta, gamma as in the notation of the paper
#
# xi, xj are the same variables xi, xj in the Boltzmann weights


#------------------------Complete homogeneous symmetric functions-------------------------

for k in range(5):
    var("h_%s_%s" % (north_card, k))

# The variable h_north_card_i represents the (north_card - i)-th
# complete homogeneous symmetric function in A and B. 
    
h_north_card_3 = A*h_north_card_4 + B^(north_card-3) 
h_north_card_2 = A*h_north_card_3 + B^(north_card-2)
h_north_card_1 = A*h_north_card_2 + B^(north_card-1)
h_north_card_0 = A*h_north_card_1 + B^north_card

h_polys = {0: h_north_card_0,
           1: h_north_card_1,
           2: h_north_card_2,
           3: h_north_card_3,
           4: h_north_card_4}


def h(i):
    """
    The function h computes the complete homogeneous symmetric functions in A and B.
    They are related by the recursion h_k(A, B) = A*h_{k - 1}(A, B) + B^k.

    When h is passed north_card - i for i = 0, 1, 2, 3, 4, it returns h_north_card_i.
    Only h_north_card_i for i = 0, 1, 2, 3, 4 can occur in the Boltzmann weights,
    so only these are computed.

    If h is passed i >= -3, it computes the function explicitly, but this is not used in
    the proof that the Yang--Baxter equation is satisfied. It does not compute h(i) for 
    i < -3 because these cannot occur in the Boltzmann weights.
    """
    if isinstance(i, sage.rings.integer.Integer):
        if i < -3:
            raise AssertionError("h outside of range")
        elif i == -3:
            return -1/(A^2 * B) - 1/(A * B^2)
        elif i > -3:
            return (A*h(i - 1) + B^i).simplify_full()
    for k in range(0, 5):
        if i == north_card - k:
            return h_polys[k]
        elif i not in [north_card - k for k in range(0, 5)]:
            raise AssertionError("h outside of range")


#------------------------Edge label functions---------------------------------------------
            
# The horizontal edge label "+" in the paper is represented here by the integer 0.
#
# Throughout the code, H is a list of the non-zero horizontal boundary edge labels.
#
# A vertical edge label vert is represented by a list of three items:
#
#     variable representing the cardinality of vert
#
#     subset of items of H contained in vert
#
#     list of variables representing |vert_{[l + 1, n]}| for each l in H
#
# Sigma denotes the label on the northern edge in the boundary condition.
#
# The variable "subcard_c_Sigma" denotes |Sigma_{[c + 1, n]}|, etc.


def list_plus(subcard_list, edge, H):
    """
    Given a list of items each representing |vert_{[l + 1, n]}| for l in H,
    list_plus calculates a new list of cardinalities resulting from adding
    an element of H to a vertical edge label vert.

    In particular, list_plus calculates |vert_{[l + 1, n]}| in terms of
    subcard_l_Sigma for l in H if vert is obtained by adding an element to Sigma.

    INPUTS:
        
        subcard_list: list of |vert_{[l + 1, n]}| for each l in H,
            e.g., [subcard_a_Sigma + 1, subcard_b_Sigma, subcard_c_Sigma]

        edge: horizontal edge label in the boundary condition to be added to vert

        H: list of all non-zero horizontal edge labels in the boundary condition

    It is necessary to first assume an ordering on the items of H.
    """
    vert_subcardinalities = []
    for low_edge in H:
        if edge > low_edge:
            vert_subcardinalities.append(subcard_list[H.index(low_edge)] + 1)
        elif edge <= low_edge:
            vert_subcardinalities.append(subcard_list[H.index(low_edge)])
    return vert_subcardinalities


def list_minus(subcard_list, edge, H):
    """
    Given a list of items each representing |vert_{[l + 1, n]}| for l in H,
    list_minus calculates a new list of cardinalities resulting from removing
    an element of H from a vertical edge label vert.

    In particular, list_minus calculates |vert_{[l + 1, n]}| in terms of
    subcard_l_Sigma for l in H if vert is obtained by removing an element from Sigma.

    INPUTS:
        
        subcard_list: list of |vert_{[l + 1, n]}| for each l in H,
            e.g., [subcard_a_Sigma + 1, subcard_b_Sigma, subcard_c_Sigma]

        edge: horizontal edge label in the boundary condition to be removed from vert

        H: list of all non-zero horizontal edge labels in the boundary condition

    It is necessary to first assume an ordering on the items of H.
    """
    vert_subcardinalities = []
    for low_edge in H:
        if edge > low_edge:
            vert_subcardinalities.append(subcard_list[H.index(low_edge)] - 1)
        elif edge <= low_edge:
            vert_subcardinalities.append(subcard_list[H.index(low_edge)])
    return vert_subcardinalities


def Sigma_plus(vert, edge, H):
    """
    Sigma_plus is used to obtain a vertical edge label by adding a horizontal edge
    label to an existing vertical edge label. In the notation of the paper, from
    Sigma it produces Sigma_c^+, etc.

    INPUTS:

        vert: vertical edge label
        edge: horizontal edge label to be added to vert
        H: list of all non-zero horizontal edge labels in the boundary condition

    EXAMPLE:

        c > b > a > 0, d = 0
        H = [c, b, a]
        Sigma = [north_card, set(), [subcard_c_Sigma, subcard_b_Sigma, subcard_a_Sigma]]
        Sigma_plus(Sigma, b, H) = [north_card + 1, {b}, [subcard_c_Sigma, \
                subcard_b_Sigma, subcard_a_Sigma + 1]]
    """
    if edge == 0:
        return vert
    elif edge != 0 and edge not in vert[1]:
        return [vert[0] + 1, vert[1].union(Set({edge})), list_plus(vert[2], edge, H)]
    elif edge in vert[1]:
        return vert


def Sigma_minus(vert, edge, H):
    """
    Sigma_minus is used to obtain a vertical edge label by removing a horizontal edge
    label from an existing vertical edge label. In the notation of the paper, from
    Sigma it produces Sigma_c^-, etc.

    INPUTS:

        vert: vertical edge label
        edge: horizontal edge label to be removed from vert
        H: list of all non-zero horizontal edge labels in the boundary condition
    """
    if edge == 0:
        return vert
    elif edge in vert[1]:
        return [vert[0] - 1, vert[1].difference(Set({edge})), \
                list_minus(vert[2], edge, H)]
    elif edge != 0 and edge not in vert[1]:
        return vert

        
def Sigma_plus_minus(vert, left_edge, right_edge, H):
    """
    Sigma_plus_minus is used to obtain a vertical edge label by adding a path to and
    removing a path from an existing vertical edge label.
    In the notation of the paper, from Sigma it produces Sigma_{c,d}^{+-}, etc.

    INPUTS:

        vert: vertical edge label
        left_edge: horizontal edge label to be added to vert
        right_edge: horizontal edge label to be removed from vert
        H: list of all non-zero horizontal edge labels in the boundary condition
    """
    if left_edge != right_edge:
        return Sigma_minus(Sigma_plus(vert, left_edge, H), right_edge, H)
    elif left_edge == right_edge:
        return vert


#-----------------------------Boltzmann weights-------------------------------------------

# Assumptions about a, b, c, d must specify that the non-zero variables are > 0
# for the Boltzmann weight functions to work correctly.


def dagger(vert, variable):
    """
    Supporting function for rect_vert_wt
    """
    return (-1)^(vert[0]+1) * (((A+G)*(B+G)*variable + G)*h(vert[0] - 1) \
            + A*B*h(vert[0] - 2))


def double_dagger(vert, left_edge, variable):
    """
    Supporting function for rect_vert_wt
    """
    return (-1)^vert[0] * (-B)^vert[2][H.index(left_edge)] * (A*B*h(vert[0] - 3) \
            + G*h(vert[0] - 2))


def rect_vert_wt(variable, west, north, east, south):
    """
    Boltzmann weights in Figure 2.1.

    A NoneType edge label will arise on an internal vertical edge
    if there are no admissible states for a given boundary condition.
    A vertex with such an edge label is assigned a Boltzmann weight of 0.
    """
    if north == None or south == None:
        return 0
    elif west == east and west == 0 and north == south:
        return dagger(north, variable)
    elif west == east and west!= 0 and north == south:
        if east not in north[1]:
            return variable * (A*B)^north[2][H.index(west)]
        elif east in north[1]:
            return (1 + (A+B+G)*variable) * (A*B)^north[2][H.index(west)]
    elif west == 0 and east != 0 and south == Sigma_plus(north, east, H) \
            and east not in north[1]:
        return (-A)^north[2][H.index(east)] * (1 + (A+G)*variable) * (1 + (B+G)*variable)
    elif east == 0 and west != 0 and south == Sigma_minus(north, west, H) \
            and west in north[1]:
        return double_dagger(north, west, variable)
    elif west !=0 and east > west and south == Sigma_plus_minus(north, east, west, H) \
            and west in north[1] and east not in north[1] and west not in south[1] \
            and east in south[1]:
        return (-A)^north[2][H.index(east)] * (-B)^north[2][H.index(west)] \
                * (1 + (A+G)*variable)
    elif east != 0 and west > east and south == Sigma_plus_minus(north, east, west, H) \
            and west in north[1] and east not in north[1] and west not in south[1] \
            and east in south[1]:
        return (-A)^Sigma_minus(north, west, H)[2][H.index(east)] \
                * (-B)^north[2][H.index(west)] * (1 + (B+G)*variable)
    else:
        return 0


def S(west, north, east, south):
    """
    Boltzmann weight of vertices in Figure 2.1 as a function of xi,
    denoted L_{ik} in the paper.
    """
    return rect_vert_wt(xi, west, north, east, south)


def T(west, north, east, south):
    """
    Boltzmann weight of vertices in Figure 2.1 as a function of xj,
    denoted L_{jk} in the paper.
    """
    return rect_vert_wt(xj, west, north, east, south)


def R(southwest, northwest, northeast, southeast):
    """
    Boltzmann weight of vertices in Figure 3.2 as a function of xi and xj,
    denoted R_{ij} in the paper.
    """
    if southwest == northwest and northwest == northeast and northeast == southeast:
        if southwest == 0:
            return (A+B+G)*xi + G*xj + 1 + (B+G)*(A+G)*xi*xj
        elif southwest != 0:
            return (A+B+G)*xj + G*xi + 1 + (B+G)*(A+G)*xi*xj
    elif northwest == southeast and northeast == southwest and northwest > southwest:
        return xj - xi
    elif northwest == southeast and northeast == southwest and northwest < southwest:
        return A * B * (xj-xi)
    elif northwest == northeast and northwest > 0 and southwest == southeast \
            and southwest == 0:
        return (1 + (B+G)*xi) * (1 + (A+G)*xi)
    elif northwest == northeast and northwest == 0 and southwest == southeast \
            and southwest > 0:
        return (1 + (B+G)*xj) * (1 + (A+G)*xj)
    elif northwest == northeast and southwest == southeast and northwest > southwest \
            and southwest > 0:
        return (1 + (B+G)*xj) * (1 + (A+G)*xi)
    elif northwest == northeast and southwest == southeast and northwest < southwest \
            and northwest > 0:
        return (1 + (B+G)*xi) * (1 + (A+G)*xj)
    else:
        return 0


#------------Functions for generating states and computing partition functions------------

def construct_vert(west, north, east, H):
    """
    Given labels for the western, northern, and eastern edges of a vertex, construct_vert
    returns the uniquely determined label for the southern edge such that the Boltzmann
    weight of the vertex is non-zero, if it exists. Otherwise, it returns None.
    """
    if west == east:
        vert = north
    elif west == 0 and east > 0 and east not in north[1]:
        vert = Sigma_plus(north, east, H)
    elif west > 0 and east == 0 and west in north[1]:
        vert = Sigma_minus(north, west, H)
    elif west > 0 and west < east and west in north[1] and east not in north[1]:
        vert = Sigma_plus_minus(north, east, west, H)
    elif east > 0 and west > east and east not in north[1] and west in north[1]:
        vert = Sigma_plus_minus(north, east, west, H)
    else:
        vert = None
    return vert


def LHS_inner(a, b, Sigma, c, H):
    """
    LHS_inner returns a list of labels for the inner edges of the left-hand side
    lattice model with a given boundary condition. Only edge labels which could
    arise in an admissible state are generated. The output of LHS_inner is
    determined by Sigma and the relative ordering of a and c and b and c.
    """
    inner_edges = []
    for (top_horiz, bot_horiz) in {(a, b), (b, a)}:
        inner_edges.append([top_horiz, bot_horiz, construct_vert(top_horiz, Sigma, c, H)])            
    return(inner_edges)


def RHS_inner(b, Sigma, c, d, H):
    """
    RHS_inner returns a list of labels for the inner edges of the right-hand side
    lattice model with a given boundary condition. Only edge labels which could
    arise in an admissible state are generated. The output of RHS_inner is
    determined by Sigma and the relative ordering of b and c and b and d.
    """
    inner_edges = []
    for (top_horiz, bot_horiz) in {(c, d), (d, c)}:
        inner_edges.append([top_horiz, bot_horiz, construct_vert(b, Sigma, top_horiz, H)])
    return(inner_edges)


def LHS_part_fun(a, b, Sigma, c, d, Sigmaprime, H):
    """
    Computes the partition function of the left-hand side lattice model in Figure 3.1.
    """
    LHS_state_weights = []
    for (top_horiz, bot_horiz, vert) in LHS_inner(a, b, Sigma, c, H):
        state_weight = R(a, b, top_horiz, bot_horiz) * S(top_horiz, Sigma, c, vert) \
                * T(bot_horiz, vert, d, Sigmaprime)
        LHS_state_weights.append(state_weight)
    return sum(LHS_state_weights)


def RHS_part_fun(a, b, Sigma, c, d, Sigmaprime, H):
    """
    Computes the partition function of the right-hand side lattice model in Figure 3.1.
    """
    RHS_state_weights = []
    for (top_horiz, bot_horiz, vert) in RHS_inner(b, Sigma, c, d, H):
        state_weight = R(bot_horiz, top_horiz, c, d) * S(a, vert, bot_horiz, Sigmaprime) \
                * T(b, Sigma, top_horiz, vert)
        RHS_state_weights.append(state_weight)
    return sum(RHS_state_weights)


def YBE(a, b, Sigma, c, d, Sigmaprime, H):
    """
    Tests partition functions of the LHS and RHS lattice models for equality.
    """
    return bool((LHS_part_fun(a, b, Sigma, c, d, Sigmaprime, H) \
            - RHS_part_fun(a, b, Sigma, c, d, Sigmaprime, H)) == 0)


#----------------------Script for verifying the Yang--Baxter equation---------------------

# The following script generates all boundary conditions it suffices to consider, computes
# the partition functions of the two resulting lattice models for each boundary condition,
# and tests them for equality.
#
# It proceeds by first generating lists [a, b, c, d] of horizontal boundary edge labels
# for all choices of a, b, c, d zero or non-zero and all orderings of the non-zero labels.
#
# For each resulting list [a, b, c, d] it next computes all pairs of northern boundary
# edge labels Sigma and southern boundary edge labels Sigmaprime for which there could
# exist an admissible state of the lattice model.

# Generate binary lists for horizontal labels 0 or non-zero
horiz_on_off = []
for i, j, k, l in itertools.product({0,1}, repeat=4):
    horiz_on_off.append([i, j, k, l])

# Convert binary lists to valid lists of 0, a, b, c, d
for binary_list in horiz_on_off:
# a, b, c, d = 1, 2, 3, 4
    letter_lists = set()
    ints = [a, b, c, d]
    for w in range(1, 2):
        for x in range(1, 3):
            for y in range(1, 4):
                for z in range(1, 5):
                    L0 = [w, x, y, z]
                    L1 = [L0[k] if binary_list[k] == 1 else 0 for k in range(4)]
                    if any(k in L1 and L1.index(k) > k - 1 for k in range(1, 5)):
                        continue    # invalid
                    else:   # valid              
                        S1 = [ints[L1[k] - 1] if L1[k] != 0 else 0 for k in range(4)] 
                        letter_lists.add(tuple(S1))
    # Generate possible containments for vertical edges
    for horiz_bdry in letter_lists:
        Horizontals = Set(horiz_bdry).difference({0})
        H = sorted(list(Horizontals))
        northern_subsets = list(Subsets(Set({horiz_bdry[0], \
                horiz_bdry[1]}).difference({0})))
        southern_subsets = list(Subsets(Set({horiz_bdry[2], \
                horiz_bdry[3]}).difference({0})))
        # Choose northern vertical edges
        for north_set, south_set in itertools.product(northern_subsets, southern_subsets):
            Sigma_subcardinalities = [var("subcard_%s_Sigma" % low_edge) \
                    for low_edge in H]
            Sigma = [north_card, north_set, Sigma_subcardinalities]
            # Assumptions on horizontal edge labels: Case 1
            if len(H) < 2:
                # Assume cardinalities are non-negative integers
                for subcardinality in Sigma_subcardinalities:
                    assume(subcardinality, "integer")
                    assume(subcardinality >= 0)
                # Assume non-zero edge labels are positive integers
                for i in range(len(H)):
                    assume(H[i], "integer")
                    assume(H[i] > 0)
                # Construct Sigmaprime corresponding to selected Sigma
                Sigmaprime = Sigma.copy()
                # Add to Sigmaprime anything in south_set
                for i in south_set:
                    Sigmaprime = Sigma_plus(Sigmaprime, i, H)
                # Remove from Sigmaprime anything in north_set that is not in south_set
                for j in north_set.difference(south_set):
                    Sigmaprime = Sigma_minus(Sigmaprime, j, H)
                if not YBE(horiz_bdry[0], horiz_bdry[1], Sigma, horiz_bdry[2], \
                        horiz_bdry[3], Sigmaprime, H):
                    print("FALSE: Yang--Baxter equation fails for: ")
                    print("Horizontal edges: ", horiz_bdry)
                    print(assumptions())
                    print("Northern edge = ", Sigma)
                    print("Southern edge = ", Sigmaprime, "\n")
                forget()
               
            # Assumptions on horizontal edge labels: Case 2
            if len(H) > 1:
                for edge_perm in Permutations(H).list():
                    # Necessary to assume subcardinalities are non-negative integers
                    for subcardinality in Sigma_subcardinalities:
                        assume(subcardinality, "integer")
                        assume(subcardinality >= 0)
                    # Necessary to assume horizontal edge labels are positive integers 
                    for i in range(len(H)):
                        assume(H[i], "integer")
                        assume(H[i] > 0)
                    # Assume ordering on non-zero horizontal edge labels
                    for i in range(len(H) - 1):
                        assume(operator.lt(edge_perm[i], edge_perm[i + 1]))
                    # Choose southern edge label corresponding to northern edge label
                    Sigmaprime = Sigma.copy()
                    # Add to Sigmaprime anything in south_set
                    for i in south_set:
                        Sigmaprime = Sigma_plus(Sigmaprime, i, H)
                    # Remove from Sigmaprime anything in north_set
                    # that is not in south_set
                    for j in north_set.difference(south_set):
                        Sigmaprime = Sigma_minus(Sigmaprime, j, H)
                    if not YBE(horiz_bdry[0], horiz_bdry[1], Sigma, horiz_bdry[2], \
                            horiz_bdry[3], Sigmaprime, H):
                        print("FALSE: Yang--Baxter equation fails for: ")
                        print("Horizontal edges: ", horiz_bdry)
                        print(assumptions())
                        print("Northern edge = ", Sigma)
                        print("Southern edge = ", Sigmaprime, "\n")
                    forget()
\end{lstlisting}

\newpage

\bibliographystyle{habbrv}
\bibliography{kirillov-conjecture}

@article{Polymath22,
   title={Solving the $n$-{C}olor {I}ce {M}odel},
   ISSN={0219-3094},
   url={http://dx.doi.org/10.1007/s00026-025-00770-1},
   DOI={10.1007/s00026-025-00770-1},
   journal={Annals of Combinatorics},
   publisher={Springer Science and Business Media LLC},
   author={Addona, Patrick and Bockenhauer, Ethan and Brubaker, Ben and Cauthorn, Michael and Conefrey-Shinozaki, Cianan and Donze, David and Dudarov, William and Dukes, Jessamyn and Hardt, Andrew and Li, Cindy and Li, Jigang and Liu, Yanli and Puthanveetil, Neelima and Qudsi, Zain and Simons, Jordan and Sullivan, Joseph and Young, Autumn},
   year={2025},
   month=jul }

@article {ABPW2023,
    AUTHOR = {Aggarwal, Amol and Borodin, Alexei and Petrov, Leonid and
              Wheeler, Michael},
     TITLE = {Free fermion six vertex model: symmetric functions and random
              domino tilings},
   JOURNAL = {Selecta Math. (N.S.)},
  FJOURNAL = {Selecta Mathematica. New Series},
    VOLUME = {29},
      YEAR = {2023},
    NUMBER = {3},
     PAGES = {Paper No. 36, 138},
      ISSN = {1022-1824,1420-9020},
   MRCLASS = {82B23 (05E05 60C05)},
  MRNUMBER = {4581739},
MRREVIEWER = {Ren\ Ding},
       DOI = {10.1007/s00029-023-00837-y},
       URL = {https://doi.org/10.1007/s00029-023-00837-y},
}

@article {ABW2023,
    AUTHOR = {Aggarwal, Amol and Borodin, Alexei and Wheeler, Michael},
     TITLE = {Colored fermionic vertex models and symmetric functions},
   JOURNAL = {Commun. Am. Math. Soc.},
  FJOURNAL = {Communications of the American Mathematical Society},
    VOLUME = {3},
      YEAR = {2023},
     PAGES = {400--630},
      ISSN = {2692-3688},
   MRCLASS = {05E05 (17B37)},
  MRNUMBER = {4628347},
       DOI = {10.1090/cams/24},
       URL = {https://doi.org/10.1090/cams/24},
}

@article {ArtinSchelterTate,
    AUTHOR = {Artin, Michael and Schelter, William and Tate, John},
     TITLE = {Quantum deformations of {${\rm GL}_n$}},
   JOURNAL = {Comm. Pure Appl. Math.},
  FJOURNAL = {Communications on Pure and Applied Mathematics},
    VOLUME = {44},
      YEAR = {1991},
    NUMBER = {8-9},
     PAGES = {879--895},
      ISSN = {0010-3640},
     CODEN = {CPAMA},
   MRCLASS = {17B37 (14A22 16W30)},
  MRNUMBER = {1127037},
MRREVIEWER = {S. Paul Smith},
       DOI = {10.1002/cpa.3160440804},
       URL = {http://dx.doi.org/10.1002/cpa.3160440804},
}

@book {Baxter,
    AUTHOR = {Baxter, Rodney J.},
     TITLE = {Exactly {S}olved {M}odels in {S}tatistical {M}echanics},
      NOTE = {Reprint of the 1982 original},
 PUBLISHER = {Academic Press, Inc. [Harcourt Brace Jovanovich, Publishers],
              London},
      YEAR = {1989},
     PAGES = {xii+486},
      ISBN = {0-12-083182-1},
   MRCLASS = {82-02 (82A05 82A69)},
  MRNUMBER = {998375},
}

@article {borodin-wheeler,
    AUTHOR = {Borodin, Alexei and Wheeler, Michael},
     TITLE = {Colored stochastic vertex models and their spectral theory},
   JOURNAL = {Ast\'{e}risque},
  FJOURNAL = {Ast\'{e}risque},
    NUMBER = {437},
      YEAR = {2022},
     PAGES = {ix+225},
      ISSN = {0303-1179,2492-5926},
      ISBN = {978-2-85629-963-0},
   MRCLASS = {05E05 (05E10 60K35 82B23)},
  MRNUMBER = {4518478},
       DOI = {10.24033/ast.1180},
       URL = {https://doi.org/10.24033/ast.1180},
}

@article {BBB2019,
  AUTHOR = {Brubaker, Ben and Buciumas, Valentin and Bump, Daniel and
  Gray, Nathan},
  TITLE = {A {Y}ang-{B}axter equation for metaplectic ice},
  JOURNAL = {Commun. Number Theory Phys.},
  FJOURNAL = {Communications in Number Theory and Physics},
  VOLUME = {13},
  YEAR = {2019},
  NUMBER = {1},
  PAGES = {101--148},
  ISSN = {1931-4523},
  MRCLASS = {22E50 (16T25 20G42)},
  MRNUMBER = {3951106},
  MRREVIEWER = {Ivan Mati\'{c}},
  DOI = {10.4310/CNTP.2019.v13.n1.a4},
  URL = {https://doi-org.ezp3.lib.umn.edu/10.4310/CNTP.2019.v13.n1.a4},
}

@misc{metahori,
Author = {Ben Brubaker and Valentin Buciumas and Daniel Bump and Henrik P. A. Gustafsson},
Title = {Metaplectic {I}wahori {W}hittaker functions and supersymmetric lattice models},
Year = {2020},
Eprint = {arXiv:2012.15778},
}

@article {BBBGatoms2021,
  AUTHOR = {Brubaker, Ben and Buciumas, Valentin and Bump, Daniel and
  Gustafsson, Henrik P. A.},
  TITLE = {Colored five-vertex models and {D}emazure atoms},
  JOURNAL = {J. Combin. Theory Ser. A},
  FJOURNAL = {Journal of Combinatorial Theory. Series A},
  VOLUME = {178},
  YEAR = {2021},
  PAGES = {Paper No. 105354, 48},
  ISSN = {0097-3165},
  MRCLASS = {05E05 (05E10 05E16)},
  MRNUMBER = {4165627},
  MRREVIEWER = {Laura Colmenarejo},
  DOI = {10.1016/j.jcta.2020.105354},
  URL = {https://doi-org.ezp3.lib.umn.edu/10.1016/j.jcta.2020.105354},
}

@article {BBBGIwahori,
    AUTHOR = {Brubaker, Ben and Buciumas, Valentin and Bump, Daniel and
              Gustafsson, Henrik P. A.},
     TITLE = {Colored vertex models and {I}wahori {W}hittaker functions},
   JOURNAL = {Selecta Math. (N.S.)},
  FJOURNAL = {Selecta Mathematica. New Series},
    VOLUME = {30},
      YEAR = {2024},
    NUMBER = {4},
     PAGES = {Paper No. 78},
      ISSN = {1022-1824,1420-9020},
   MRCLASS = {22E50 (05E05 11F70 16T25 81R50 82B23)},
  MRNUMBER = {4795105},
       DOI = {10.1007/s00029-024-00950-6},
       URL = {https://doi.org/10.1007/s00029-024-00950-6},
}

@article {BBBG2024,
    AUTHOR = {Brubaker, Ben and Buciumas, Valentin and Bump, Daniel and
              Gustafsson, Henrik P. A.},
     TITLE = {Iwahori-metaplectic duality},
   JOURNAL = {J. Lond. Math. Soc. (2)},
  FJOURNAL = {Journal of the London Mathematical Society. Second Series},
    VOLUME = {109},
      YEAR = {2024},
    NUMBER = {6},
     PAGES = {Paper No. e12896, 54},
      ISSN = {0024-6107,1469-7750},
   MRCLASS = {82B23 (05E05 16T25 17B37 17B69 22E50)},
  MRNUMBER = {4751865},
       DOI = {10.1112/jlms.12896},
       URL = {https://doi.org/10.1112/jlms.12896},
}

@incollection {metaplectic-ice,
    AUTHOR = {Brubaker, Ben and Bump, Daniel and Chinta, Gautam and
              Friedberg, Solomon and Gunnells, Paul E.},
     TITLE = {Metaplectic ice},
 BOOKTITLE = {Multiple {D}irichlet series, {L}-functions and automorphic
              forms},
    SERIES = {Progr. Math.},
    VOLUME = {300},
     PAGES = {65--92},
 PUBLISHER = {Birkh\"{a}user/Springer, New York},
      YEAR = {2012},
      ISBN = {978-0-8176-8333-7},
   MRCLASS = {11M32 (22E50)},
  MRNUMBER = {2952572},
MRREVIEWER = {Goran\ Mui\'{c}},
       DOI = {10.1007/978-0-8176-8334-4\_3},
       URL = {https://doi.org/10.1007/978-0-8176-8334-4_3},
}

@article{brubaker-bump-friedberg2011,
    url = {https://doi.org/10.1007/s00220-011-1345-3},
    author = {Brubaker, B. and Bump, D. and Friedberg, S.},
    title = {Schur Polynomials and the {Yang-Baxter Equation}},
    journal = {Communications in Mathematical Physics},
    year = {2011}
}

@article {frozen-pipes,
    AUTHOR = {Brubaker, Ben and Frechette, Claire and Hardt, Andrew and
              Tibor, Emily and Weber, Katherine},
     TITLE = {Frozen pipes: lattice models for {G}rothendieck polynomials},
   JOURNAL = {Algebr. Comb.},
  FJOURNAL = {Algebraic Combinatorics},
    VOLUME = {6},
      YEAR = {2023},
    NUMBER = {3},
     PAGES = {789--833},
      ISSN = {2589-5486},
   MRCLASS = {05E05 (05E14 14C35)},
  MRNUMBER = {4614163},
MRREVIEWER = {John\ Machacek},
       DOI = {10.5802/alco.277},
       URL = {https://doi.org/10.5802/alco.277},
}

@article {BS2022,
  AUTHOR = {Buciumas, Valentin and Scrimshaw, Travis},
  TITLE = {Double {G}rothendieck polynomials and colored lattice models},
  JOURNAL = {Int. Math. Res. Not. IMRN},
  FJOURNAL = {International Mathematics Research Notices. IMRN},
  YEAR = {2022},
  NUMBER = {10},
  PAGES = {7231--7258},
  ISSN = {1073-7928},
  MRCLASS = {14M15 (16T25 19L47)},
  MRNUMBER = {4418706},
  DOI = {10.1093/imrn/rnaa327},
  URL = {https://doi-org.ezp3.lib.umn.edu/10.1093/imrn/rnaa327},
}

@article {BSW2020,
  AUTHOR = {Buciumas, Valentin and Scrimshaw, Travis and Weber, Katherine},
  TITLE = {Colored five-vertex models and {L}ascoux polynomials and
  atoms},
  JOURNAL = {J. Lond. Math. Soc. (2)},
  FJOURNAL = {Journal of the London Mathematical Society. Second Series},
  VOLUME = {102},
  YEAR = {2020},
  NUMBER = {3},
  PAGES = {1047--1066},
  ISSN = {0024-6107},
  MRCLASS = {05A19 (05E05 14M15 82B23)},
  MRNUMBER = {4186121},
  MRREVIEWER = {Eric S. Egge},
  DOI = {10.1112/jlms.12347},
  URL = {https://doi-org.ezp3.lib.umn.edu/10.1112/jlms.12347},
}

@article {bump2024colored,
    AUTHOR = {Bump, Daniel and Naprienko, Slava},
     TITLE = {Colored {B}osonic models and matrix coefficients},
   JOURNAL = {Commun. Number Theory Phys.},
  FJOURNAL = {Communications in Number Theory and Physics},
    VOLUME = {18},
      YEAR = {2024},
    NUMBER = {2},
     PAGES = {441--484},
      ISSN = {1931-4523,1931-4531},
   MRCLASS = {22E50 (05E05 11F70 16T25 82B23)},
  MRNUMBER = {4776163},
}

@book{Chari-Pressley,
  title     = {A {G}uide to {Q}uantum {G}roups},
  author    = {Chari, Vyjayanthi and Pressley, Andrew},
  year      = {1994},
  publisher = {Cambridge University Press},
  address   = {Cambridge}
}

@article {chen-zhang-2022,
    AUTHOR = {Chen, Yuqun and Zhang, Zerui},
     TITLE = {A weak version of {K}irillov's conjecture on
              {H}ecke--{G}rothendieck polynomials},
   JOURNAL = {J. Combin. Theory Ser. A},
  FJOURNAL = {Journal of Combinatorial Theory. Series A},
    VOLUME = {186},
      YEAR = {2022},
     PAGES = {Paper No. 105555, 18},
      ISSN = {0097-3165,1096-0899},
   MRCLASS = {05E05},
}

@article {curran2022latticemodelsuperllt,
    AUTHOR = {Curran, Michael J. and Frechette, Claire and Yost-Wolff,
              Calvin and Zhang, Sylvester W. and Zhang, Valerie},
     TITLE = {A lattice model for super {LLT} polynomials},
   JOURNAL = {Comb. Theory},
  FJOURNAL = {Combinatorial Theory},
    VOLUME = {3},
      YEAR = {2023},
    NUMBER = {2},
     PAGES = {Paper No. 3, 52},
      ISSN = {2766-1334},
   MRCLASS = {05E05 (05E10 82B20)},
  MRNUMBER = {4646084},
       DOI = {10.5070/c63261979},
       URL = {https://doi.org/10.5070/c63261979},
}

@phdthesis{dasher_2025,
AUTHOR = {Dasher, A. Suki},
TITLE = {Universal {D}ivided {D}ifference {O}perators and {S}olvable {L}attice {M}odels},
SCHOOL ={University of Minnesota},
YEAR = {2025},
}

@article {difrancesco-zinn-justin,
    AUTHOR = {Di Francesco, P. and Zinn-Justin, P.},
     TITLE = {Inhomogeneous model of crossing loops and multidegrees of some
              algebraic varieties},
   JOURNAL = {Comm. Math. Phys.},
  FJOURNAL = {Communications in Mathematical Physics},
    VOLUME = {262},
      YEAR = {2006},
    NUMBER = {2},
     PAGES = {459--487},
      ISSN = {0010-3616,1432-0916},
   MRCLASS = {82B41},
  MRNUMBER = {2200268},
MRREVIEWER = {Sergei\ K.\ Lando},
       DOI = {10.1007/s00220-005-1476-5},
       URL = {https://doi.org/10.1007/s00220-005-1476-5},
}

@article {drinfeld_1989,
    AUTHOR = {Drinfel'd, V. G.},
     TITLE = {Quasi-{H}opf algebras},
   JOURNAL = {Algebra i Analiz},
  FJOURNAL = {Algebra i Analiz},
    VOLUME = {1},
      YEAR = {1989},
    NUMBER = {6},
     PAGES = {114--148},
      ISSN = {0234-0852},
   MRCLASS = {17B37 (16W30 57M25 81T40)},
  MRNUMBER = {1047964},
MRREVIEWER = {Ya.\ S.\ So\u ibel\cprime man},
}

@article {gorbounovkorff,
    AUTHOR = {Gorbounov, Vassily and Korff, Christian},
     TITLE = {Quantum integrability and generalised quantum {S}chubert
              calculus},
   JOURNAL = {Adv. Math.},
  FJOURNAL = {Advances in Mathematics},
    VOLUME = {313},
      YEAR = {2017},
     PAGES = {282--356},
      ISSN = {0001-8708,1090-2082},
   MRCLASS = {14M15 (05E05 14F43 19L47 55N20 55N22 82B23)},
  MRNUMBER = {3649227},
MRREVIEWER = {Xin\ Fang},
       DOI = {10.1016/j.aim.2017.03.030},
       URL = {https://doi.org/10.1016/j.aim.2017.03.030},
}

@article {kirillov2016notes,
    AUTHOR = {Kirillov, Anatol N.},
     TITLE = {Notes on {S}chubert, {G}rothendieck and key polynomials},
   JOURNAL = {SIGMA Symmetry Integrability Geom. Methods Appl.},
  FJOURNAL = {SIGMA. Symmetry, Integrability and Geometry. Methods and
              Applications},
    VOLUME = {12},
      YEAR = {2016},
     PAGES = {Paper No. 034, 56},
      ISSN = {1815-0659},
   MRCLASS = {05E05 (05A19 05E10)},
  MRNUMBER = {3478964},
MRREVIEWER = {Eric\ S.\ Egge},
       DOI = {10.3842/SIGMA.2016.034},
       URL = {https://doi.org/10.3842/SIGMA.2016.034},
}

@article {KM,
  AUTHOR = {Knutson, Allen and Miller, Ezra},
  TITLE = {Gr\"{o}bner geometry of {S}chubert polynomials},
  JOURNAL = {Ann. of Math. (2)},
  FJOURNAL = {Annals of Mathematics. Second Series},
  VOLUME = {161},
  YEAR = {2005},
  NUMBER = {3},
  PAGES = {1245--1318},
  ISSN = {0003-486X},
  MRCLASS = {05E15 (13C40 13F55 13P10 14M15 14N15)},
  MRNUMBER = {2180402},
  MRREVIEWER = {Harry Tamvakis},
  DOI = {10.4007/annals.2005.161.1245},
  URL = {https://doi-org.ezp3.lib.umn.edu/10.4007/annals.2005.161.1245},
}

@misc{KZJ2017,
Author = {Allen Knutson and Paul Zinn-Justin},
Title = {Schubert puzzles and integrability {I}: invariant trilinear forms},
Year = {2017},
Eprint = {arXiv:1706.10019},
}

@misc{KZJ2021,
Author = {Allen Knutson and Paul Zinn-Justin},
Title = {Schubert puzzles and integrability {II}: multiplying motivic {S}egre classes},
Year = {2021},
Eprint = {arXiv:2102.00563},
}

@article{kojima2013,
  doi = {10.1063/1.4799933},
  url = {https://doi.org/10.1063%2F1.4799933},
  year = 2013,
  month = {Apr},
  publisher = {{AIP} Publishing},
  volume = {54},
  number = {4},
  author = {Takeo Kojima},
  title = {Diagonalization of transfer matrix of supersymmetry {$U_q(\widehat{sl}(M+1 \vert N + 1))$} chain with a boundary},
  journal = {Journal of Mathematical Physics}
}

@article {L2007,
  AUTHOR = {Lascoux, Alain},
  TITLE = {The 6 vertex model and {S}chubert polynomials},
  JOURNAL = {SIGMA Symmetry Integrability Geom. Methods Appl.},
  FJOURNAL = {SIGMA. Symmetry, Integrability and Geometry. Methods and
  Applications},
  VOLUME = {3},
  YEAR = {2007},
  PAGES = {Paper 029, 12},
  MRCLASS = {05E15 (82B20 82B23)},
  MRNUMBER = {2299830},
  MRREVIEWER = {Thomas Fun Yau Lam},
  DOI = {10.3842/SIGMA.2007.029},
  URL = {https://doi-org.ezp3.lib.umn.edu/10.3842/SIGMA.2007.029},
}

@article {liu2022,
    AUTHOR = {Liu, Ricky Ini},
     TITLE = {Twisted {S}chubert polynomials},
   JOURNAL = {Selecta Math. (N.S.)},
  FJOURNAL = {Selecta Mathematica. New Series},
    VOLUME = {28},
      YEAR = {2022},
    NUMBER = {5},
     PAGES = {Paper No. 87, 23},
      ISSN = {1022-1824,1420-9020},
   MRCLASS = {05E14 (14N15)},
  MRNUMBER = {4493417},
MRREVIEWER = {Paolo\ Aluffi},
       DOI = {10.1007/s00029-022-00802-1},
       URL = {https://doi.org/10.1007/s00029-022-00802-1},
}

@article {MihalceaSuWhittaker,
    AUTHOR = {Mihalcea, Leonardo C. and Su, Changjian},
     TITLE = {Whittaker functions from motivic {C}hern classes},
      NOTE = {With an appendix by Mihalcea, Su and Dave Anderson},
   JOURNAL = {Transform. Groups},
  FJOURNAL = {Transformation Groups},
    VOLUME = {27},
      YEAR = {2022},
    NUMBER = {3},
     PAGES = {1045--1067},
      ISSN = {1083-4362,1531-586X},
   MRCLASS = {22E46 (14D24 14N15 19E15)},
  MRNUMBER = {4475486},
MRREVIEWER = {Changlong\ Zhong},
       DOI = {10.1007/s00031-022-09731-x},
       URL = {https://doi.org/10.1007/s00031-022-09731-x},
}

@article {MS2013,
  AUTHOR = {Motegi, Kohei and Sakai, Kazumitsu},
  TITLE = {Vertex models, {TASEP} and {G}rothendieck polynomials},
  JOURNAL = {J. Phys. A},
  FJOURNAL = {Journal of Physics. A. Mathematical and Theoretical},
  VOLUME = {46},
  YEAR = {2013},
  NUMBER = {35},
  PAGES = {355201, 26},
  ISSN = {1751-8113},
  MRCLASS = {81R12},
  MRNUMBER = {3100873},
  DOI = {10.1088/1751-8113/46/35/355201},
  URL = {https://doi-org.ezp3.lib.umn.edu/10.1088/1751-8113/46/35/355201},
}

@article {PS1981,
    AUTHOR = {Perk, Jacques H. H. and Schultz, Cherie L.},
     TITLE = {New families of commuting transfer matrices in {$q$}-state
              vertex models},
   JOURNAL = {Phys. Lett. A},
  FJOURNAL = {Physics Letters. A},
    VOLUME = {84},
      YEAR = {1981},
    NUMBER = {8},
     PAGES = {407--410},
      ISSN = {0375-9601,1873-2429},
   MRCLASS = {82A05 (82A68)},
  MRNUMBER = {627568},
       DOI = {10.1016/0375-9601(81)90994-4},
       URL = {https://doi.org/10.1016/0375-9601(81)90994-4},
}

@article {ReshetikhinMultiparameter,
    AUTHOR = {Reshetikhin, N.},
     TITLE = {Multiparameter quantum groups and twisted quasitriangular
              {H}opf algebras},
   JOURNAL = {Lett. Math. Phys.},
  FJOURNAL = {Letters in Mathematical Physics. A Journal for the Rapid
              Dissemination of Short Contributions in the Field of
              Mathematical Physics},
    VOLUME = {20},
      YEAR = {1990},
    NUMBER = {4},
     PAGES = {331--335},
      ISSN = {0377-9017},
     CODEN = {LMPHDY},
   MRCLASS = {17B37 (16W30 81R50)},
  MRNUMBER = {1077966},
MRREVIEWER = {A. O. Barut},
       DOI = {10.1007/BF00626530},
       URL = {http://dx.doi.org/10.1007/BF00626530},
}

@article {FRT1989,
    AUTHOR = {Reshetikhin, N. Yu. and Takhtadzhyan, L. A. and Faddeev, L.
              D.},
     TITLE = {Quantization of {L}ie groups and {L}ie algebras},
   JOURNAL = {Algebra i Analiz},
  FJOURNAL = {Algebra i Analiz},
    VOLUME = {1},
      YEAR = {1989},
    NUMBER = {1},
     PAGES = {178--206},
      ISSN = {0234-0852},
   MRCLASS = {17B65 (17B35 22E46 58F07 81D07 82A69)},
  MRNUMBER = {1015339},
MRREVIEWER = {Ya.\ S.\ So\u{\i}bel\cprime man},
}

@article {Sudbery,
    AUTHOR = {Sudbery, A.},
     TITLE = {Consistent multiparameter quantisation of {${\rm GL}(n)$}},
   JOURNAL = {J. Phys. A},
  FJOURNAL = {Journal of Physics. A. Mathematical and General},
    VOLUME = {23},
      YEAR = {1990},
    NUMBER = {15},
     PAGES = {L697--L704},
      ISSN = {0305-4470},
     CODEN = {JPHAC5},
   MRCLASS = {17B37 (16W30 81R50)},
  MRNUMBER = {1068228},
MRREVIEWER = {Jie Du},
       URL = {http://stacks.iop.org/0305-4470/23/L697},
}

@article {Takeuchi,
    AUTHOR = {Takeuchi, Mitsuhiro},
     TITLE = {A two-parameter quantization of {${\rm GL}(n)$} (summary)},
   JOURNAL = {Proc. Japan Acad. Ser. A Math. Sci.},
  FJOURNAL = {Japan Academy. Proceedings. Series A. Mathematical Sciences},
    VOLUME = {66},
      YEAR = {1990},
    NUMBER = {5},
     PAGES = {112--114},
      ISSN = {0386-2194},
     CODEN = {PJAADT},
   MRCLASS = {16W30 (17B37 20F05)},
  MRNUMBER = {1065785},
MRREVIEWER = {Richard Dipper},
       URL = {http://projecteuclid.org/euclid.pja/1195512514},
}

@article {wheeler-zinn-justin,
    AUTHOR = {Wheeler, Michael and Zinn-Justin, Paul},
     TITLE = {Refined {C}auchy/{L}ittlewood identities and six-vertex model
              partition functions: {III}. {D}eformed bosons},
   JOURNAL = {Adv. Math.},
  FJOURNAL = {Advances in Mathematics},
    VOLUME = {299},
      YEAR = {2016},
     PAGES = {543--600},
      ISSN = {0001-8708,1090-2082},
   MRCLASS = {05E05},
  MRNUMBER = {3519476},
MRREVIEWER = {Michael\ Orin\ Joyce},
       DOI = {10.1016/j.aim.2016.05.010},
       URL = {https://doi.org/10.1016/j.aim.2016.05.010},
}

@article {wheeler--zinn-justin-2019,
    AUTHOR = {Wheeler, Michael and Zinn-Justin, Paul},
     TITLE = {Littlewood--{R}ichardson coefficients for {G}rothendieck
              polynomials from integrability},
   JOURNAL = {J. Reine Angew. Math.},
  FJOURNAL = {Journal f\"ur die Reine und Angewandte Mathematik. [Crelle's
              Journal]},
    VOLUME = {757},
      YEAR = {2019},
     PAGES = {159--195},
      ISSN = {0075-4102,1435-5345},
   MRCLASS = {14M15 (05E05 14C35 14N15)},
  MRNUMBER = {4036573},
MRREVIEWER = {Praise\ Adeyemo},
       DOI = {10.1515/crelle-2017-0033},
       URL = {https://doi.org/10.1515/crelle-2017-0033},
}

@misc{Zemel,
Author = {Shaul Zemel},
Title = {Polynomial {D}ivided {D}ifference {O}perators {S}atisfying the {B}raid {R}elations},
Year = {2024},
Eprint = {arXiv:2404.19395},
}

\end{document}